\documentclass[platex]{amsart}
\usepackage{amssymb}
\usepackage{braket}
\usepackage{mathrsfs}
\usepackage{ifthen}
\usepackage{graphicx}
\usepackage{tikz}
\usetikzlibrary{patterns}
\usepackage{comment} 
\usepackage{mleftright}
\usepackage{cleveref}
 \usepackage[all]{xy}
 \usepackage{amscd}
 \usepackage{amsrefs}
 \usepackage{color}
 \usepackage{comment}
\usepackage{todonotes}
\usepackage{url}
\usepackage{ulem}

\theoremstyle{plain}
\newtheorem{theo}{Theorem}[section]
\crefname{theo}{Theorem}{Theorems}
\Crefname{theo}{Theorem}{Theorems}
\newtheorem{prop}[theo]{Proposition}
\crefname{prop}{Proposition}{Propositions}
\Crefname{prop}{Proposition}{Propositions}
\newtheorem{lem}[theo]{Lemma}
\crefname{lem}{Lemma}{Lemmas}
\Crefname{lem}{Lemma}{Lemmas}
\newtheorem{cor}[theo]{Corollary}
\crefname{cor}{Corollary}{Corollaries}
\Crefname{cor}{Corollary}{Corollaries}

\crefname{claim}{Claim}{Claims}
\Crefname{claim}{Claim}{Claims}

\crefname{property}{Property}{Properties}
\Crefname{property}{Property}{Properties}
\newtheorem{problem}[theo]{Problem}
\crefname{problem}{Problem}{Problems}
\Crefname{problem}{Problem}{Problems}

\theoremstyle{definition}
\newtheorem{defi}[theo]{Definition}
\crefname{defi}{Definition}{Definitions}
\Crefname{defi}{Definition}{Definitions}

\crefname{notation}{Notation}{Notations}
\Crefname{notation}{Notation}{Notations}

\crefname{convention}{Convention}{Conventions}
\Crefname{convention}{Convention}{Conventions}

\crefname{cond}{Condition}{Conditions}
\Crefname{cond}{Condition}{Conditions}

\crefname{assum}{Assumption}{Assumptions}
\Crefname{assum}{Assumption}{Assumptions}

\theoremstyle{remark}
\newtheorem{rem}[theo]{Remark}
\crefname{rem}{Remark}{Remarks}
\Crefname{rem}{Remark}{Remarks}
\newtheorem{ex}[theo]{Example}
\crefname{ex}{Example}{Examples}
\Crefname{ex}{Example}{Examples}

\crefname{section}{Section}{Sections}
\Crefname{section}{Section}{Sections}
\crefname{subsection}{Subsection}{Subsections}
\Crefname{subsection}{Subsection}{Subsections}
\crefname{figure}{Figure}{Figures}
\Crefname{figure}{Figure}{Figures}

\newtheorem*{acknowledgement}{Acknowledgement}

\newcommand{\Z}{\mathbb{Z}}

\newcommand{\R}{\mathbb{R}}
\newcommand{\C}{\mathbb{C}}
\newcommand{\quat}{\mathbb{H}}
\newcommand{\Q}{\mathbb{Q}}

\newcommand{\sign}{\mathrm{sign}}

\newcommand{\fraks}{\mathfrak{s}}
\newcommand{\frakt}{\mathfrak{t}}

\newcommand{\id}{\mathrm{id}}
\newcommand{\ind}{\mathop{\mathrm{ind}}\nolimits}

\newcommand{\wt}{\widetilde}
\newcommand{\frkt}{\mathfrak{t}}
\newcommand{\Sol}{\mathrm{Sol}}

\def\C{\mathbb{C}}

\def\wt{\widetilde}

\newcommand{\cal}{\mathcal}

\newcommand{\Int}{\mathrm{Int}}
\newcommand{\Diff}{\mathrm{Diff}}

\newcommand{\Aut}{\mathrm{Aut}}

\newcommand{\Map}{\mathrm{Map}}

\newcommand{\inc}{\hookrightarrow}

\newcommand{\del}{\partial}

\newcommand{\Ker}{\mathop{\mathrm{Ker}}\nolimits}

\newcommand{\Hom}{\mathop{\mathrm{Hom}}\nolimits}

\title[Involutions and Floer K-theory]{Involutions, knots, and Floer K-theory}

\author{Hokuto Konno}
\address{Graduate School of Mathematical Sciences, the University of Tokyo, 3-8-1 Komaba, Meguro, Tokyo 153-8914, Japan}
\email{konno@ms.u-tokyo.ac.jp}

\author{Jin Miyazawa}
\address{Graduate School of Mathematical Sciences, the University of Tokyo, 3-8-1 Komaba, Meguro, Tokyo 153-8914, Japan}
\email{miyazawa@ms.u-tokyo.ac.jp}

\author{Masaki Taniguchi}
\address{2-1 Hirosawa, Wako, Saitama 351-0198, Japan}
\email{masaki.taniguchi@riken.jp}

\begin{document}

\maketitle

\begin{abstract}
We establish a version of Seiberg--Witten Floer $K$-theory for knots, as well as a version of Seiberg--Witten Floer $K$-theory for 3-manifolds with involution.
The main theorems are 10/8-type inequalities for knots and for involutions.
The 10/8-inequality for knots yields numerous applications to knots, such as lower bounds on stabilizing numbers and relative genera.
We also give obstructions to extending involutions on 3-manifolds to 4-manifolds, and detect non-smoothable involutions on 4-manifolds with boundary.
\end{abstract}

\tableofcontents

\section{Introduction}

\subsection{Overview}
The purpose of this paper is to establish a version of Seiberg--Witten Floer $K$-theory for knots, as well as a version of Seiberg--Witten Floer $K$-theory for 3-manifolds with involutions.
Our construction is based on Manolescu's Seiberg--Witten Floer stable homotopy type \cite{Ma03} and an involutive symmetry on the Seiberg--Witten equations introduced by Kato~\cite{Ka17}.
Kato's involution on the Seiberg--Witten equations takes effects of given involutions on 4-manifolds into account,
but his framework is {\it not} an equivariant Seiberg--Witten theory in the usual sense.
(See \cref{subsection Equivariant Floer homotopy type and Floer Ktheory} for detail.)
While Kato~\cite{Ka17} considered closed spin 4-manifolds, 
in this paper, we develop a 3-dimensional version of \cite{Ka17} and extend Kato's work to 4-manifolds with boundary.

While various researchers have studied effects of group actions in several types of Floer homology and their applications to knots, such as \cite{H12,KL15,HLS16,HM17,LM182,LRS18,Kang18a,Kang18b,HL19,Large19,HLS20,AKS20,LRS20,HLS20a,HLL20,DHM21,BH21,CS99, KM11, DS19}, to the best of our knowledge, this paper is the first study of Floer $K$-theory defined for knots with a 10/8-type inequality. 
Like the original 10/8-inequality by Furuta~\cite{Fu01} and its generalization to 4-manifolds with boundary by Manolescu~\cite{Ma14},
constraints on knots obtained from our 10/8-type inequality are much different from constraints obtained from usual (i.e. ordinary cohomological) Floer theory for 3-manifolds with involutions.

Our main results are 10/8-type inequalities for spin 4-manifolds with boundary and with involution, and for surfaces bounded by knots.
The 10/8-inequality for knots yields numerous interesting applications to bounds on stablizing numbers and relative genera.
Representative applications to knots from our framework shall be described in \cref{knot main app,stabilizing app}.
\cref{stabilizing app} detects a difference of the topological category and the smooth category from a new point of view: a comparison between topological and smooth stablizing numbers.
\cref{knot main app} detects a big difference of topological and smooth minimal genera for knots embedded into the boundary of a punctured 4-manifold.
Our 10/8-type inequality also obstructs extending involutions on 3-manifolds to spin 4-manifolds, stated in \cref{theo intro nonextendable general}.
This obstruction is strong enough to detect non-smoothable group actions: this is summarized in \cref{theo: non-smoothable actions main thm 6nminus1}.

\subsection{Relative 10/8-inequality for involutions} 

Let $Y$ be an oriented rational homology 3-sphere, $\frakt$ be a spin structure on $Y$, and $\iota$ be an orientation-preserving smooth involution whose fixed-point set is non-empty and of codimension-2.
Suppose that $\iota$ preserves the spin structure $\frakt$.
(If $Y$ is a $\Z_2$-homology 3-sphere, this is the case for all $\iota$.)
We shall define a numerical invariant
\[
\kappa(Y, \frakt, \iota) \in \frac{1}{16}\Z,
\]
which we call the {\it $K$-theoretic Fr{\o}yshov invariant} of the triple $(Y, \frakt, \iota)$.
This is a version (taking effects of $\iota$ into account) of the $K$-theoretic invariant $\kappa(Y,\frakt)$ introduced by Manolescu~\cite{Ma14}, which was used in \cite{Ma14} to establish a relative version of Furuta's 10/8-inequality \cite{Fu01}.

Most of results of this paper follow from a relative 10/8-type inequality for spin 4-manifolds with involution described below.
For triples $(Y_0, \frakt_0,\iota_0)$ and $(Y_1, \frakt_1,\iota_1)$, we call a triple $(W,\fraks,\iota)$ a {\it smooth spin cobordism with involution} from $(Y_0, \frakt_0,\iota_0)$ to $(Y_1, \frakt_1,\iota_1)$ if $(W,\fraks)$ is a smooth compact connected oriented spin cobordism from $(Y_0, \frakt_0)$ to $(Y_1, \frakt_1)$ and $\iota$ is a smooth involution of $W$ which preserves the orientation and the spin structure $\fraks$ and whose restriction to the boundary is given by $\iota_0 \sqcup \iota_1$.
Let $b^+_\iota(W)$ denote the maximal dimension of $\iota$-invariant positive-definite subspaces of $H^2(W;\R)$.
The focus of the following theorem is the property (iv), which is our relative 10/8-inequality for involutions:

\begin{theo}
\label{main theo}
Let $(Y, \frakt)$ be an oriented spin rational homology 3-sphere and $\iota$ be an orientation-preserving smooth involution on $Y$ whose fixed-point set is non-empty and of codimension-2 and which preserves the spin structure $\frakt$.
We can associate an invariant
$\kappa(Y, \frakt, \iota) \in \frac{1}{16}\Z$ to every such triple $(Y, \frakt, \iota)$, with the following properties:
\begin{itemize}
\item [(i)] The mod 2 reduction of $-2\kappa(Y, \frakt, \iota)$ coincides with the Rokhlin invariant $\mu(Y,\frakt)$:
\[
-2 \kappa (Y, \mathfrak{t}, \iota) = \mu (Y , \mathfrak{t}) \text{ in } \left( \frac{1}{8} \Z  \right) / 2\Z \cong \Z/16\Z. 
\]
\item [(ii)] The quantity $\kappa (Y, \mathfrak{t}, \iota)$ is invariant under conjugation: for every diffeomorphism $f$ on $Y$ preserving the orientation and the spin structure $\frakt$, we have
\[
\kappa (Y, \mathfrak{t}, \iota) = \kappa (Y, \mathfrak{t}, f^{-1} \circ \iota \circ f).
\]
\item [(iii)] For $-Y$, the same manifold with the reversed orientation, we have
\[
\kappa(Y,\frakt,\iota) + \kappa(-Y,\frakt,\iota)
\geq 0.
\]
\item [(iv)] Let $(W,\fraks,\iota)$ be a smooth spin cobordism with involution from $(Y_0, \frakt_0,\iota_0)$ to $(Y_1, \frakt_1,\iota_1)$ with $b_1(W)=0$.
Then we have
\begin{align}
\label{eq: main rel 108}
-\frac{\sigma(W)}{16} + \kappa (Y_0, \frakt_0, \iota_0 )
\leq b^+(W)-b^+_{\iota}(W) + \kappa (Y_1, \frakt_1, \iota_1 ).
\end{align}
\end{itemize}
\end{theo}

\begin{rem}
If $Y=S^3$, regarded as a subset of $\C^2$, and $\iota : Y \to Y$ is the complex conjugation, we will see in \cref{ex: S3 case} that $\kappa(Y,\frakt,\iota)=0$ for the unique spin structure $\frakt$ on $S^3$.
Thus we can deduce from \eqref{eq: main rel 108} a similar statement for 
a one-boundary component spin 4-manifold with involution.
Also, taking $Y_0=Y_1=S^3$ with complex conjugation, \cref{main theo} recovers Kato's 10/8-inequality \cite[Theorem 2.3]{Ka17} for closed spin 4-manifolds.
  \end{rem}

Moreover, we can generalize \cref{main theo} to fixed-point free odd involutions.
See \cref{most general main theo} and \cref{most general main theo one boundary}.

\subsection{$K$-theoretic knot concordance invariant}
\label{K theoretic knot concordance invariant}

Applying Floer theory with involutive symmetry to the double branched covers, several knot (concordance) invariants are defined in Heegaard Floer homology, such as \cite{H12,KL15,HLS16,HLL20,HLS20,AKS20,DHM21}, and Seiberg--Witten Floer theory \cite{BH21}.
Via orbifold gauge theory, several versions of knot instanton Floer homology are developed \cite{CS99, KM11, DS19}. We also provide a knot concordance invariant from our Floer $K$-theory for 3-manifolds with involution.

For simplicity, we consider oriented knots throughout the paper, while our invariant of knots itself is independent of the choice of orientation.
For a given knot $K$ in $S^3$, one can associate an oriented rational homology 3-sphere $\Sigma(K)$ called the branched covering space along $K$, which is equipped with the covering involution $\iota_K$.
Define the {\it $K$-theoretic Fr{\o}yshov invariant for the knot $K$} by
\[
\kappa (K) := \kappa (\Sigma(K) , \mathfrak{t},  \iota_K) \in \frac{1}{16}\Z,
\]
where $\frakt$ is a spin structure on $\Sigma(K)$, which is unique since  it is known that $H^1(\Sigma(K);\Z/2)=0$.
As a consequence of \cref{main theo}, we prove the following properties of $\kappa(K)$, which includes a 10/8-type inequality for knots: the property (v) below.

\begin{theo}\label{main knot}
The invariant $\kappa (K)$ satisfies the following properties:
\begin{itemize}
\item[(i)] The invariant $\kappa (K)$ is a knot concordance invariant.

\item[(ii)]
For every knot $K$, we have 
$
\kappa (K) =  \kappa (-K)$, where $-K$ is the knot with the opposite orientation.
\item[(iii)] For every knot $K$ in $S^3$, we have 
\[
2 \kappa (K) =- \frac{1}{8} \sigma (K) \text{ in } \left( \frac{1}{8}  \Z \right) / 2\Z \cong \Z/16\Z.
\]
\item[(iv)]
For every knot $K$ in $S^3$, we have 
$
\kappa (K) + \kappa (K^* ) \geq 0$, 
where $K^*$ denotes the mirror image of $K$.

\item[(v)] Let $K$ and $K'$ be knots in $S^3$, $W$ be an oriented smooth compact connected cobordism from $S^3$ to $S^3$ with $H_1(W; \Z)=0$, and $S$ be an oriented compact connected properly and smoothly embedded cobordism in $W$ from $K$ to $K'$ such that the homology class $[S]$ of $S$ is divisible by $2$ and $PD(w_2(W)) = [S]/2 \operatorname{mod} 2$. Then, we have
    \begin{align} \label{ineq1}
  -\frac{\sigma(W)}{8} + \frac{9}{32}[S]^2-\frac{9}{16}\sigma(K') + \frac{9}{16}\sigma(K) \leq b^+(W) + g(S) + \kappa(K')-\kappa(K), 
    \end{align}
    where $\sigma(K)$ denotes the signature of $K$ (with the sign convention $\sigma(T(2,3))=-2$) and $g(S)$ is the genus of $S$.
  \end{itemize}
    \end{theo}
\begin{rem}
If $W$ is spin, $K, K'$ are the unknots, and $S$ is a null-homologous smoothly embedded annulus, then the inequality~\eqref{ineq1} implies $-\sigma(W)/8   \leq b^+(W)$, which recovers Furuta's original 10/8-inequality~\cite{Fu01} except for adding 1 on the left-hand side.
Note that one can deduce a constraint on surfaces bounded by knots from Manolescu's relative 10/8-inequality \cite{Ma14} applied to the branched covering spaces. In \cref{Comparison}, we shall summarize such constraints following from known results.
Also, the invariant $\kappa$ can be extended to an invariant of a pair $(Y, K)$ of an oriented homology 3-sphere $Y$ and a knot $K$ in $Y$. For more details, see \cref{generalY}.
\end{rem}

\begin{rem}
It could be interesting to consider an approach to the 11/8-conjecture due to Matsumoto~\cite{M82} from \cref{main knot}.
In short, if one can find a suitable embedded surface $S$ in a spin 4-manifold $W$, one can refine the 10/8-inequality for $W$.
See \cref{11/8-conjecture}.
\end{rem}

Next, we shall exhibit calculations of $\kappa(K)$ for 2-bridge knots and torus knots:

\begin{theo}
 \label{two bridge kappa12}
For coprime integers $p,q$ with $p$ odd, let $K(p,q)$ be the two bridge knot of type $(p,q)$, whose branched cover is the lens space $L(p,q)$.
Then we have 
\begin{align}\label{two bridge kappa1}
\kappa (K(p,q)) =  -\frac{1}{16} \sigma (K(p,q)). 
\end{align}
The equation \eqref{two bridge kappa1} also holds for every connected sum of two bridge knots.
\end{theo}

\begin{theo}\label{torus knot calculation} 
Let $p,q$ be coprime odd integers and $T(p,q)$ be the torus knot of type $(p,q)$. 
Then the equality 
\[
\kappa (  T(p,q) ) = - \frac{1}{2}  \bar{\mu} (\Sigma (2,p,q) ) 
\]
holds, 
where $\bar{\mu}$ is the Neumann--Siebenmann invariant. Moreover, the equality 
$
\kappa (K) = 
- \frac{1}{2}\bar{\mu} (\Sigma (K))
$
holds for every connected sum of torus knots $T(p,q)$'s for coprime odd integers $p$ and $q$.  \end{theo}
To compute $\kappa (K)$ more, we provide the following connected sum formula: 
 \begin{theo}\label{connected sum} 
  Let $K$ be a knot in $S^3$ and $K'$ be a connected sum of knots appearing in \cref{two bridge kappa12,torus knot calculation}. Then we have 
  \[
  \kappa (K \# K') = \kappa (K) + \kappa (K') .
  \]
 \end{theo}
We also give a crossing change formula of $\kappa$ using \cref{main knot}: 
\begin{theo} \label{crossing change}
Let $K$ and $K'$ be knots in $S^3$. 
\begin{itemize}
    \item[(i)] Suppose $K'$ is obtained from $K$ by a sequence of $n$ crossing changes. Then the inequality  
    \[
    \left|\kappa (K')- \kappa (K) + \frac{9}{16}\sigma (K') - \frac{9}{16}\sigma (K) \right|  \leq n . 
    \]
    \item[(ii)] 
    Suppose $K'$ is obtained from $K$ by a sequence of $n$ positive crossing changes. Then the inequality 
\[
 \kappa (K')- \kappa (K) \leq - \frac{9}{16}\sigma (K') +  \frac{9}{16}\sigma(K)
\]
holds, where our convention on the positive crossing change is described in Figure~\ref{Positive crossing change}. 
\end{itemize}
\end{theo}

It is not difficult to prove a similar full twist formula under a certain assumption. See \cref{full-twist}. In \cref{Kappa invariant for knots with 8 and 9-crossings}, we give calculations of $\kappa (K)$ for prime knots with 8- or 9-crossings using \cref{crossing change} combined with \cref{main knot}~(iii). The result is summarized in \cref{more computation}.
 
\subsection{Applications to stabilizing numbers} 


The results explained in \cref{K theoretic knot concordance invariant} have various applications to 4-dimensional aspects of knot theory.
We exhibit some of representative applications in this introduction.
The first one is to detect the difference of topological and smooth stablizing numbers.
A given knot $K$ in $S^3$ is said to be \textit{smoothly (resp. topologically) H-slice} in $X$ if there exists a proper and smooth (resp. locally flat) null-homologus embedding of a 2-dimensional disk bounded by $K$ in $X \setminus \operatorname{int}D^4$.
It is proven in \cite{Sch10} that, for every knot $K$ in $S^3$ whose Arf invariant $\operatorname{Arf}(K)$ is zero, there is a positive integer $N$ such that $K$ is smoothly H-slice in $\#_N S^2\times S^2$, the $N$-fold connected sum of copies of $S^2 \times S^2$. 
This result enables us to define invariants 
\[
\mathrm{sn} (K):= \min \Set{ N | \text{ $K$ is smoothly H-slice in $\#_N S^2\times S^2$} } 
\]
\text{ and } 
\[
\mathrm{sn}^{\mathrm{Top}} (K):= \min \Set{ N | \text{ $K$ is topologically H-slice in $\#_N S^2\times S^2$} } 
\]
when $\operatorname{Arf}(K)$ is zero. 
These quantities have been studied in the literature, and for example in \cite{CN20}, the invariant $\mathrm{sn}(K)$ (resp. $\mathrm{sn}^{\mathrm{Top}}(K)$) is called the {\it smooth (resp. topological) stabilizing number} of $K$. 

Our invariant $\kappa (K)$ can be used to give a lower bound on $\mathrm{sn}(K)$:

\begin{theo}
\label{H-sliceing number1}
For every knot $K$ in $S^3$ with $\operatorname{Arf}(K) =0$, we have 
    \[
        -\frac{9 }{16}  \sigma (K)  - \kappa (K) \leq  \mathrm{sn}(K).
        \]
\end{theo}

It is a natural and basic question whether the topological stabilizing number and smooth stabilizing number have an essential difference.
More concretely, it is asked in \cite[Question 1.4]{CN20} whether there exists a knot $K$ such that 
\[
 0< \mathrm{sn}^{\mathrm{Top}}(K) < \mathrm{sn}(K).
\]
We give the affirmative answer to this question:
\begin{theo}
\label{stabilizing app}
 There exists a knot $K$ in $S^3$ with $\operatorname{Arf}(K) = 0$ such that 
  $0< \mathrm{sn}^{\mathrm{Top}}(K) < \mathrm{sn}(K)$ holds. Moreover, we have $0<\mathrm{sn}(\#_n K )$ for any positive integer $n$ and 
  \[
  \lim_{n\to \infty }\left( \mathrm{sn}(\#_n K )- \mathrm{sn}^{\mathrm{Top}}(\#_n K) \right) = \infty . 
  \]
\end{theo}
For example, we can take $K$ in \cref{stabilizing app} to be $T(3,11)$.
Moreover, in fact, we may detect many examples of $K$ which satisfy the statement of \cref{stabilizing app}.
See \cref{slicing number different} for more examples.

\subsection{Applications to relative genera} 

It is one of the most classical problems in low dimensional topology to study genus bounds of embedded surfaces in 4-manifolds.
For a given knot in $S^3$, one may consider bounds on genera of surfaces bounded by the knot in a punctured 4-manifold.
We shall apply the results explained in \cref{K theoretic knot concordance invariant} to this problem.

Let $X$ be an oriented smooth closed 4-manifold with a second homology class $x \in H_2(X;\Z)$.
For a knot $K$ in $S^3$, let $g_{X,x}(K)$ be the minimum of genera of surfaces $S$ which are properly and smoothly embedded oriented connected compact surfaces in $X \setminus \operatorname{int} D^4$ such that $\partial S =K$ and $[S] =x \in H_2(X;\Z)$.
Here $[S]$ denotes the relative fundamental class of $S$, and $H_2(X;\Z)$ is naturally identified with $H_2(X \setminus \operatorname{int} D^4;\Z)$.
This quantity $g_{X,x}(K)$ is called the {\it smooth relative $(X,x)$-genus} of $K$, and
the topological version $g^{\mathrm{Top}}_{X,x}(K)$ is defined by considering locally flat embeddings instead of smooth embeddings.
The relative genera $g^{\mathrm{Top}}_{X,x}(K)$ and  $g_{X,x}(K)$ are natural generalizations of classical invariants: the topological and smooth 4-genera (or slice genera), which are defined as $g^{\mathrm{Top}}_{S^4,0}(K)$ and  $g_{S^4,0}(K)$.
There are various known lower bounds on $g^{\mathrm{Top}}_{X,x}(K)$, such as \cite{Gi81, manolescu, R65, FK78, KR20, CN20}.
Also there are known lower bounds on $g_{X,x}(K)$ for both definite 4-manifolds (for examples, see \cite{OS03, KM13, MMSW19}) and indefinite 4-manifolds \cite{MR06, HR20, IMT21}.
On the other hand, for indefinite 4-manifolds with vanishing gauge-theoretic invariant, there are few ways to obtain lower bounds on $g_{X,x}(K)$, except for an approach based on 10/8-inequality \cite{manolescu}. 

There is a big difference between the topological and smooth 4-genera: it follows from the affirmative answer to the Milnor conjecture by Kronehimer and Mrowka \cite{KM93} and a recent result \cite{BBL20} by Baader, Banfield and Lewark on topological slice genus that 
\[
\lim_{n\to \infty} \left(g_{S^4,0}(K_n )  - g^{\mathrm{Top}}_{S^4,0}(  K_n )\right) = \infty
\]
for $K_n=T(3,12n-1)$, the torus knot of type $(3,12n-1)$.
A similar observation can be applied to definite 4-manifolds $X$ and to every $x \in H_2(X;\Z)$ using the Ozv\'ath--Szab\'o $\tau$-invariant \cite{OS03}, and one obtains $\lim_{n\to \infty} (g_{X,x}(K_n )  - g^{\mathrm{Top}}_{X,x}(K_n )) = \infty$. 

We prove that, for every smooth closed 4-manifold $X$ with vanishing first homology, there is a big difference between $g_{X,x}(K)$ and $g^{\mathrm{Top}}_{X,x}(K)$:

\begin{theo} \label{knot main app} There exists a knot $K'$ such that the following result holds.  Let $X$ be a smooth closed 4-manifold with $H_1(X;\Z) =0$ and $x \in H_2(X;\Z)$ be a second homology class  which is divisible by $2$ and satisfies $x/2 = PD(w_2(X)) \operatorname{mod} 2$.  
Then, for every knot $K$ in $ S^3$, we have 
     \[
 \lim_{n\to \infty} \left(g_{X,x}( K \# (\#_n K')  )  - g^{\mathrm{Top}}_{X,x}(K \# (\#_n K') ) \right) = \infty.
 \]
\end{theo}
Concretely, we can take $K'$ in \cref{knot main app} to be, for example, $K'=T(3,11)$. 
A remarkable feature of \cref{knot main app} is that we do not have to impose any restriction on the intersection form of $X$.

In \cref{ Genus bounds from }, we give
several lower bounds on relative genera stronger than known ones are obtained for a class of 4-manifolds including $\#_n K3$ and $\#_n\C P^2\#_m\overline{\C P}^2$, summarized in \cref{H-sliceing numberK3,cp-genus} respectively. 

\subsection{Applications to non-extendable and non-smoothable actions}
\label{subsec:Applications to non-extendable and non-smoothable actions}

In \cref{section Applications to non-extendable actions}, we use \cref{main theo} to obstruct an extension of involutions on 3-manifolds to spin 4-manifolds.
Regard Brieskorn homology spheres $\Sigma(p,q,r)$ as a subset of $\C^3$ followings the standard definition.
Consider the involution $\iota$ on $\Sigma(2,q,r)$ defined by $\iota(z_1,z_2,z_3)=(-z_1,z_2,z_3)$.
A recent result by Anvari and Hambleton~\cite[Theorem~A]{AH21} showed that the standard finite cyclic group actions on $\Sigma(p,q,r)$ does not extend to any contractible smooth 4-manifold bounded by $\Sigma(p,q,r)$ (if exists), as a smooth involution.
In particular, $\iota$ does not extend to any contractible smooth 4-manifold bounded by $\Sigma(2,q,r)$ as a smooth involution.
(See also recent work by Baraglia--Hekmati~\cite[Example~7.7]{BH21}.)
As a complementary result, we obstruct to extending such $\iota$ to non-contractible 4-manifolds for some class of Brieskorn homology spheres:

\begin{theo}
\label{theo intro nonextendable general}
Let $a_{1}, \ldots, a_{n}$ be pairwise coprime natural numbers.
Suppose that $a_{1}$ is an even number.
Set $Y=\Sigma(a_{1}, \ldots, a_{n})$, and define an involution $\iota : Y \to Y$ by
\[
\iota(z_1,z_2, \ldots,z_n)=(-z_1,z_2,\ldots,z_n).
\]
Let $W$ be a compact connected smooth oriented spin 4-manifold bounded by $Y$ with $b_1(W)=0$.
Then we have the following:
\begin{itemize}
\item [(i)] The involution $\iota$ cannot extend to $W$ as a smooth involution so that
\begin{align*}
-\frac{\sigma(W)}{16} 
> b^+(W)-b^+_{\iota}(W) - \frac{\bar{\mu}(Y)}{2}.
\end{align*}
\item[(ii)] Suppose that $\sigma(W) \neq 8\bar{\mu}(Y)$.
Then $\iota$ cannot extend to $W$ as a homologically trivial smooth involution, while $\iota$ can extend to $W$ as a homologically trivial diffeomorphism.
\end{itemize}
\end{theo}

It is well-known that, if a Seifert homology sphere bounds a smooth homology 4-ball, then its Neumann--Siebenmann invariant vanishes.
This shows that \cref{theo intro nonextendable general} is complementary to the aforementioned results by Anvari and Hambleton~\cite{AH21} and by Baraglia and Hekmati~\cite{BH21} about contracible or homology ball bounds.
The results of \cref{theo intro nonextendable general} shall be extended to connected sums of Seifert homology spheres in \cref{theo: nonext connecedt sum Seifert cal most general}.

The constraint on smooth involutions given in \cref{theo intro nonextendable general} is strong enough to detect non-smoothable group actions as follows.
Let $\iota_r : S^2\times S^2 \to S^2\times S^2$ be an orientation-preserving smooth involution defined as the product of the $\pi$-rotation of $S^2$ along an axis and the identity of $S^2$.
Let $M(p,q,r)$ denote the Milnor fiber associated to the
sigularity $z_1^p+z_2^q+z_3^r=0$.

\begin{theo}
\label{theo: intro non-smoothable actions main thm 6nminus1}
Let $W$ be the 4-manifold defined by
\[
W = M(2,3,6n-1)\#_{2n+1} S^2 \times S^2
\]
for $n\geq 2$, with boundary $\Sigma(2,3,6n-1)$.
Then there exists a locally linear topological involution $\iota_W : W \to W$ with non-empty fixed-point set that satisfies the following properties:
\begin{itemize}
\item[(I)] The involution $\iota_W$ is not smooth with respect to every smooth structure on $W$.

\item[(II)] The restriction of $\iota_W$ to the boundary $\del{W}$ is the involution $\iota : \Sigma(2,3,6n-1) \to \Sigma(2,3,6n-1)$ defined by $\iota(z_1, z_2, z_3) = (-z_1, z_2, z_3)$.
In particular, $\iota_W|_{\del W}$ extends as a diffeomorphism of $W$ for every smooth structure on $W$.

\item[(III)] For any $N>1$, the equivariant connected sum 
\[
\iota_W \#_N \iota_r : W \#_N S^2 \times S^2 \to W \#_N S^2 \times S^2
\]
along fixed points is also a non-smoothable involution with respect to every smooth structure on $W \#_N S^2 \times S^2$.

\item[(IV)] The quotient orbifold $W/\iota_W$ is indefinite.
More precisely, $b^+(W/\iota_W) = b^-(W/\iota_W) = 4n$.

\end{itemize}
\end{theo}

The property (III) means that the non-smoothability of $\iota_W$ is stable under a suitable equivariant connected sum.
\cref{theo: intro non-smoothable actions main thm 6nminus1} shall be generalized in
\cref{theo: non-smoothable actions main thm 6nminus1} for connected sums of $W$ in \cref{theo: intro non-smoothable actions main thm 6nminus1}, and also $M(2,3,6n+1)$ shall be treated instead of $M(2,3,6n-1)$.

\subsection{Floer homotopy type and Floer $K$-theory}
\label{subsection Equivariant Floer homotopy type and Floer Ktheory}

The all results explained until the previous \lcnamecref{subsec:Applications to non-extendable and non-smoothable actions} are derived from \cref{main theo}.
The main ingredients to establish \cref{main theo} are versions of Floer homotopy type and of Floer $K$-theory.
Let us clarify the nature of our construction of these ingredients in this \lcnamecref{subsection Equivariant Floer homotopy type and Floer Ktheory}.

Based upon Manolescu's construction of Seiberg--Witten Floer stable homotopy type \cite{Ma03},
we shall construct versions of {\it Seiberg--Witten Floer stable homotopy type for involutions}
\[
DSWF_G(Y, \frakt,\iota)
\]
and of {\it Seiberg--Witten Floer $K$-theory for involutions}
\[
DSWFK_G(Y, \frakt,\iota),
\]
which are defined for a spin rational homology 3-sphere $(Y, \frakt)$ with an involution $\iota$ which preserves $\frakt$ and whose fixed-point set is of codimension-2.
Here $G$ stands for $G=\Z_4$, which is the subgroup of $Pin(2)$ generated by $j \in Pin(2)$, and $D$ stands for a ``doubling'' construction introduced in \cref{subsection Doubling}.

Via double branched covers, we obtain the {\it Seiberg--Witten Floer stable homotopy type of a knot $K$ in $S^3$}
\[
DSWF(K) 
\]
and the {\it Seiberg--Witten Floer $K$-theory of a knot $K$}
\[
DSWFK(K).
\]
These are closely related to the recent work by Baraglia and Hekmati~\cite{BH21}, where they established an equivariant Seiberg--Witten Floer stable homotopy type for finite group actions on spin$^c$ 3-manifolds, and an equivariant Seiberg--Witten Floer cohomology, and defined knot invariants in a similar way.
(Note also a recent combinatorial construction of knot Floer homotopy type by Manolescu and Sarkar~\cite{MS21}.)

 However, even restricting our attention to involutions and the case of spin 3-manifolds, there are significant differences between the equivariant Floer homotopy type by Baraglia--Hekmati~\cite{BH21} and our Floer homotopy type for involutions.
The first major difference is that
our construction is based on an involutive symmetry on the Seiberg--Witten equations introduced by Kato~\cite{Ka17}.
Equivariant gauge theory such as \cite{Ba19,BH21} for involutions tends to replace $b^+(W)$ with $b^+_{\iota}(W)$, the $\iota$-invariant part of the involution $\iota$ of a 4-manifold $W$.
On the other hand, complementarily, in the setup of Kato, $b^+(W)$ is replaced with $b^+(W) - b^+_{\iota}(W)$, which appears in \eqref{eq: main rel 108} in \cref{main theo} also in our setting.

To simplify complications in the construction of Seiberg--Witten Floer $K$-theory in our setup, we shall consider the ``complexification''  or  ``double" of the whole construction of the Floer homotopy type.
See \cref{subsection Doubling} for more details on this point.

\subsection{Structure of the paper}
We finish off this introduction with an outline of the contents of this paper. In \Cref{subsection: An involution on the configuration space}, following the Kato's work \cite{Ka17}, we introduce the involution $I$ on the configuration space for Seiberg--Witten equations on 3-manifolds, which we mainly use in the construction of our invariant.
We apply equivariant Conley index theory to the fixed-point part of the symmetry $I$, and this will be a main ingredient of our Seiberg--Witten Floer stable homotopy type.
In \Cref{section Gequivariant Seiberg--Witten Floer homotopy theory}, for a rational homology sphere with an involution, we construct the (doubled) Seiberg--Witten Floer stable homotopy type, Floer $K$-theory, $K$-theoretic Fr{\o}yshov invariant, and corresponding invariants for knots.
The proof of \cref{main theo} is given in  \cref{subsection Proof of main theo}.
Also, we give several computations of our Seiberg--Witten Floer homotopy type and $K$-theoretic Fr{\o}yshov invariant in \cref{subsection: Calculations}.
In \cref{Applications to knot theory}, we give applications of \cref{main theo} to 4-dimensional aspects of knot theory.
We first review several fundamental calculations related to double branched covering spaces for properly embedded surfaces in punctured 4-manifolds. Using such calculations and \cref{main theo}, we prove \cref{main knot}. We discuss several examples including $\#_n S^2\times S^2$, $\#_nK3$, and $\#_n\C P^2\#_m(- \C P^2)$.
In \Cref{section Applications to non-extendable actions}, we obtain results on non-extendable actions of Seifert homology 3-spheres and non-smoothable actions on 4-manifolds with boundary using \cref{main theo}. In particular, we prove \cref{theo intro nonextendable general}.
\Cref{Comparison} is devoted to summarizing other methods to obtain relative genus bounds such as Manolescu's relative 10/8-inequality and the Tristram--Levine signature, which are related to our inequality \eqref{ineq1} in \cref{main knot}.
In \cref{Kappa invariant for knots with 8 and 9-crossings}, we provide several computations of the kappa invariant for prime knots with $8$- or $9$-crossings. 

\begin{acknowledgement}
The authors would like to express their gratitude to David Baraglia, Mikio Furuta, Nobuhiro Nakamura, O{\u{g}}uz \c{S}avk, Motoo Tange, and Yuichi Yamada for helpful comments.

The first author was partially supported by JSPS KAKENHI Grant Numbers 17H06461, 19K23412, and 21K13785.
The second author was supported by JSPS KAKENHI Grant Number 21J22979 and WINGS-FMSP program at the Graduate school of Mathematical Science, the University of Tokyo.
The third author was supported by JSPS KAKENHI Grant Number 20K22319 and RIKEN iTHEMS Program.
\end{acknowledgement}

\section{An involution on the configuration space} 
\label{subsection: An involution on the configuration space}

Following Kato's work~\cite{Ka17} in dimensional $4$, we shall define an involution on the configuration space for the Seiberg--Witten equations in the 3-dimensional setting.
First let us recall a term on involutions \cite{AB68,Bry98}.
We call a smooth involution $\iota$ on a smooth spin manifold a {\it spin involution of odd type} if $\iota$ lifts to an automorphism of the spin structure as a $\Z_4$-action.
If the fixed-point set is non-empty, this is equivalent to that $\iota$ lifts to a spin automorphism and the fixed-point set is of codimension-2, as far as the spin manifold is of $\dim \leq 4$ \cite[Proposition~8.46]{AB68}.
{\it Throughout this paper, all involutions on spin 3- or 4-manifolds we consider are of odd type}.

Let $(Y, \frakt)$ be a spin rational homology 3-sphere and $\iota$ be an involution on $Y$ preserving the isomorphism class of $\frakt$. Take an $\iota$-invariant Riemannian metric $g$ on $Y$.
Let us denote by $\mathbb{S}$ the spinor bundle of $Y$.
Suppose that $\iota$ is of odd type.
Then we can take a lift $\wt{\iota}$ of $\iota$ to an automorphism of $\mathbb{S}$ satisfying  
\[
\wt{\iota}^2 =-1 . 
\]
This lift $\tilde{\iota}$ yields a $\Z_4$-action on $\mathbb{S}$.
There exists exactly one more lift of $\iota$, which is given by $-\tilde{\iota}$.

On the quaternionic structure of the spinor bundle,
we adopt the convention that the quaternion scalars act on the right.
Following Kato~\cite[Subsection~4.2]{Ka17}, we define an involution
\begin{align}
\label{eq: def of I}
I : T^*Y \oplus \mathbb{S} \to T^*Y \oplus \mathbb{S}
\end{align}
by
\[
I  ( a, \phi) = ( -\iota^*a, \wt{\iota}( \phi)\cdot j ). 
\]
Here $a \in T_y^\ast Y$ and $\phi \in \mathbb{S}_y$ for $y \in Y$.
This involution $I$ is a direct sum of involutions $I : T^*Y \to T^*Y$ and $I : \mathbb{S} \to \mathbb{S}$, and the involution $I$ on the spinors commutes with the right action of $j \in Pin(2)$, but anti-commutes with the right action of $i$. 

We define the action of $I$ to the gauge group $\mathcal G(Y)=\{ u \colon Y \to U(1) \}$ as follows. Let $u \colon Y \to U(1)$ be an element in $\mathcal G(Y)$. We set
\[
I(u)(x):=u(\iota(x))^{-1}. 
\]
It is easy to see that this action is compatible with the action of $I$ to the spinor bundle. 
\begin{rem}
This involution $I$ is the composition of the action of $(\iota^{\ast}, \tilde{\iota})$ and the right action of $(-1, j)$. The second action was considered by Furuta~\cite{Fu01} in his proof of the 10/8-inequality. The action of $-1$ to the form part is necessary to get the equivariance for the Seiberg--Witten equations.
This is because, for any spinor $\phi$, we have $(\phi j (\phi j)^{\ast})_0 = -(\phi \phi^{\ast})_0$.
\end{rem}
The above $I$ induces an involution on the space of sections of $TY \oplus \mathbb{S}$, and
restricting this, we obtain an involution
\[
I : \operatorname{Ker}(d^* : \Omega^1(Y) \to \Omega^0(Y)) \times \Gamma(\mathbb{S}) \to  \operatorname{Ker}(d^* : \Omega^1(Y) \to \Omega^0(Y))  \times \Gamma(\mathbb{S}). 
\]
Here we consider the Sobolev norms $L^2_{k-\frac{1}{2}}$ for the spaces $\operatorname{Ker}(d^* : \Omega^1(Y) \to \Omega^0(Y))$ and $\Gamma(\mathbb{S})$ defined using the $\iota$-invariant metric and $\iota$-invariant connections for a fixed integer $ k \geq 3$.
Henceforth, $\operatorname{Ker}(d^* : \Omega^1(Y) \to \Omega^0(Y))$ will be simply denoted by $\operatorname{Ker}d^*$. 

\begin{rem}
There are two choices of lift of $\iota$, namely, $\tilde{\iota}$ and $-\tilde{\iota}$. Thus we have two involutions $I$, $I'$ from the lift $\tilde{\iota}$ and $-\tilde{\iota}$ respectively. We have an isomorphism from the $I$ fixed part to the $I'$ fixed part which is given by $(a, \phi) \mapsto (a, \phi \cdot \sqrt{-1})$. This isomorphism preserves the formal gradient flow of the Chern--Simons--Dirac functional. Therefore, the Conley index 
\cref{subsection $G$-equivariant Seiberg--Witten Floer homotopy type}, 
made from the $I$-fixed part and the doubled Conley index made from the $I'$-fixed part are same. 
\end{rem}

\begin{lem}\label{ooo1}
The formal gradient of $CSD$
\[
\sigma   : \operatorname{Ker}d^* \times \Gamma(\mathbb{S})  \to \operatorname{Ker}d^* \times \Gamma(\mathbb{S}) 
\]
is equivariant with respect to the action $I$. 
\end{lem}
\begin{proof}
The action $I$ is the composition of the $(\iota^{\ast}, \tilde{\iota})$ and the right $(-1, j)$ action. It is obvious that the action of $(-1, j)$ preserves the formal gradient flow of CSD because this is an element of the ordinary $\text{Pin}^-(2)$ action to $\operatorname{Ker}d^* \times \Gamma(\mathbb{S})$. 
Let $P$ is the principal $\text{Spin}$ bundle which gives the spin structure $\mathfrak{t}$. 
The automorphism $\tilde{\iota}$ on $P$ is a lift of the involution on the $\text{SO}(TY)$ given by $\iota$. The vector bundles $TY$ and $\mathbb{S}$ are all associated bundle of $P$ and the Clifford multiplication $TY \otimes \mathbb{S} \to \mathbb{S}$ and the map $\mathbb{S} \ni \phi \mapsto (\phi \phi^{\ast})_0 \in \sqrt{-1}\Lambda^2 T^{\ast}Y$ are given by a multiplication of elements $\text{Spin}(4)$ representation space. Thus we have that the formal gradient of CSD is equivariant under the action of $(\iota^{\ast}, \tilde{\iota})$.  
\end{proof}
Define 
\[
V(Y, \frakt,  \iota) := ( \operatorname{Ker}d^* \times \Gamma(\mathbb{S}))^I . 
\]
The vector field $\sigma$ induces a vector field 
\begin{align}
\label{eq: equiv flow}
 \sigma' : V(Y, \frakt,  \iota) \to V(Y, \frakt,  \iota),     
\end{align}
The vector field $\sigma'$ decomposes into the linear part $l$ and quadratic part $c$:
$\sigma' = l+ c$.
Here $l$ is given as $(*d, D)$,
where $D$ is the spin Dirac operator.
For $\lambda < 0 \leq \mu$,
define the subspace $V^\lambda_\mu$ of $V(Y, \frakt,  \iota)$ as the direct sum of eigenspaces for $l$ whose eigenvalues lie in $(\lambda, \mu]$.
Note that $V^\lambda_\mu$ is finite-dimenional for any $\lambda, \mu$, and we think of $V^\lambda_\mu$ as a finite-dimensional approximation of $V(Y, \frakt,  \iota)$.

Also, we may define an involution $I$ on $Pin(2)$
by
\[
I \cdot g := j g j^{-1},
\]
and it is easy to see that
 \[
  Pin(2)^I = \{ 1, j, -1, -j\} \cong \Z_4. 
 \]
 Set 
 \[
 G := \{ 1, j, -1, -j\} \subset Pin(2).
 \]
It follows from a direct calculation that
\[
I g^* (a,\phi) = (I \cdot g)^* (a,\phi). 
\]
In summary, the flow \eqref{eq: equiv flow} inherits $G$-equivariance from the original $Pin(2)$-symmetry of the Seiberg--Witten flow.

For this $G$-equivariant flow $\sigma'$, we may repeat the construction by Manolescu~\cite{Ma03} of the equivariant Conley index associated with the Seiberg--Witten flow.
Thus we have a $G$-equivariant Conley index 
 \[
 I^\mu_\lambda (Y, \frakt, \iota, \tilde{\iota}, g) 
 \]
 for sufficiently large $-\lambda$, $\mu$.

\begin{rem}
While $I^\mu_\lambda (Y, \frakt, \iota, \tilde{\iota}, g)$ depends also on the choice of lift $\tilde{\iota}$ of $\iota$, this does not affect our construction of stable Floer homotopy type in an essential way.
See \cref{lem: indep tilde iota}.
\end{rem}

\begin{rem}
Kato's \cite{Ka17} and our setup is related to the Spin$^{c-}$ structure, which is introduced by Nakamura in~\cite{Na13}. 
When the involution $\iota$ is free, our setting coincides with the Spin$^{c-}$ structure on $X/\iota$ which comes from a spin structure. 
In that situation, our involution $I$ coinsides with the involution $I$ in~\cite[Section 3, (iii)]{Na13} in the case of the Spin$^c$ structure of $\tilde X$. 

Our situation is also related to the situation dealt with by  Bryan~\cite{Bry98}. Our involution $I$ corresponds to $[j, x] \in G_{\text{odd}}=(Pin(2)\times \Z_4)/\Z_2$ in Bryan's notation, where $x$ is a generator of $\Z_4$.
\end{rem}

\section{Seiberg--Witten Floer homotopy theory for involutions}
\label{section Gequivariant Seiberg--Witten Floer homotopy theory}

\subsection{Representations}

Recall that we defined the group $G$ to be the cyclic group of order $4$ generated by $j \in Pin(2)$, i.e. 
\[
G = \{1, j, -1, -j\}.
\]
Define a subgroup $H$ of $G$ by
\[
H = \{1,-1\} \subset G.
\]
Let $\tilde{\R}$ be the 1-dimensional real representation space of $G$ defined by the surjection $G \to \Z_2=\{1,-1\}$ and the scalar multiplication of $\Z_2$ on $\R$.

Let $\tilde{\C}$ be a 1-dimensional complex representation of $G$ defined via the surjection $G \to \Z_2$ and the scalar multiplication of $\Z_2$ on $\C$. 
Note that, if $s$ is even, say $2t$, then there is an isomorphism of real representations $\tilde{\R}^{2s} \cong \tilde{\C}^t$.

We introduce also $G$-representations 
 \begin{align}
 \label{rep four}
 \R,\  \tilde{\R},\ \C_+,\ \C_-,
 \end{align}
 where $\R, \tilde{\R}$ are real 1-dimensional representations and $\C_+, \C_-$ are complex 1-dimensional representations, defined by assigning to $j \in G$ multiplication of $1, -1, i, -i$ respectively.
Note that we have the relation
\[
\C_- = \tilde{\C} \cdot \C_+
\]
in $R(G)$.

We can think of these four spaces \eqref{rep four} as $H$-representations through the inclusion $H \subset G$, and they correspond to
 \[
 \R,\  \R,\ \tilde{\C},\ \tilde{\C}
 \]
 as $H$-representations, respectively.
 
It is straightforward to check that the representation ring $R(G)$ is given by
\[
R(G) = \Z[w,z]/(w^2-2w, w-2z+z^2).
\]
Here the generators $w,z$ are given as the $K$-theoretic Euler classes of $\tilde{\C}, \C_+$, namely, 
\[
w=1-\tilde{\C},\quad z=1-\C_+.
\]
The augumentation map $R(G) \to \Z$ is given by
\[
w,z \mapsto 0,
\]
so the augmentation ideal is given by $(w,z) \subset R(G)$.

Compared with the standard expression of $R(G)$, given by $R(G)=\Z[t]/(t^4-1)$,
an isomorphism $\Z[w,z]/(w^2-2w, w-2z+z^2) \to \Z[t]/(t^4-1)$ is given by $w \mapsto 1-t^2$ and $z \mapsto 1-t$.

It is also straightforward to check that
\begin{align}
\label{eq: ker rep g h}
\mathrm{Ker}(R(G) \to R(H)) = \Set{cw \in R(G) | c \in \Z}.
\end{align}

In a finite-dimensional approximation of the configuration space on which the (finite-dimensional approximation of) $G$-equivariant Seiberg--Witten flow acts, only three of the representations \eqref{rep four} may appear:
$\tilde{\R},\ \C_+,\ \C_-$.
Here $\C_+, \C_-$ appear in the following way.
Originally a finite-dimensional approximation of the configuration space is of the form $\tilde{\R}^N \oplus \quat^{N'}$ as $Pin(2)$-representation for large $N, N'$.
Restricting attention to the $I$-invariant part, the remaining symmetry on the configuration space is given by $G$, and as $G$-representation space, $\quat$ splits into $\C_+ \oplus \C_-$.
More concretely, let us write the set of quaternions by 
\begin{align}
\label{eq: split quat}
\quat 
= \C \oplus j\C
= \R \oplus i\R \oplus j\R \oplus (-k)\R
\cong (\R \oplus j\R) \oplus i(\R\oplus(-j)\R).    
\end{align}
Let us equip $\R \oplus j\R$ with the complex structure by $j$, and $\R \oplus (-j)\R$ with the complex structure by $-j$.
Then the $(\R \oplus j\R)$-component of \eqref{eq: split quat} corresponds to $\C_+$, and the $(\R \oplus (-j)\R)$-component corresponds to $\C_-$.

\subsection{Space of type $G$-SWF}

\begin{defi}
Let $X$ be a pointed finite $G$-CW complex.
We call $X$ a {\it space of type $G$-SWF} if 
\begin{itemize}
    \item $X^H$ is $G$-homotopy equivalent to $(\tilde{\R}^s)^+$ for some $s \geq 0$.
    \item $G$ acts freely on $X \setminus X^H$.
\end{itemize}
The natural number $s$ is called the {\it level} of $X$.
\end{defi}

Let $X$ be a space of type $G$-SWF at even level $s=2t$.
Let $\iota : X^H \to X$ denote the inclusion.
The image of $\iota^\ast : \tilde{K}_G(X) \to \tilde{K}_G(X^H)$ is expressed of the form $\mathfrak{J}(X) \cdot b_{t\tilde{\C}}$, where $\mathfrak{J}(X)$ is an ideal of $R(G)$ and $b_{t\tilde{\C}}$ is the Bott element.

Repeating the proof of \cite[Lemma~1]{Ma14} under replacing $Pin(2)$ and $S^1$ with $G$ and $H$ respectively,
the $G$-equivariant localization theorem deduces the following \lcnamecref{lem: localization}:

\begin{lem}
\label{lem: localization}
Let $X$ be a a space of type $G$-SWF at even level $s=2t$.
Then there exists $k \geq 0$ such that
$w^k, z^k \in \mathfrak{J}(X)$.
\end{lem}

\begin{defi}
For a space of type $G$-SWF, define
\[
k(X) := \min\Set{k \geq 0 | \exists x \in \mathfrak{J}(X), wx = 2^k w}.
\]
By \cref{lem: localization},
there exists $k \geq 0$ such that $wx = 2^k w$ for some $x \in \mathfrak{J}(X)$,
and thus $k(X)$ is a well-defined natural number.
\end{defi}

The rest of this subsection is devoted to establishing basic properties of the quantity $k(X)$.

\begin{lem}
\label{lem: origin of local eq}
 Let $X$ and $X'$ be spaces of type $G$-SWF at the same even level.
Suppose that there exists a $G$-equivariant map $f : X \to X'$ whose $H$-fixed-point set map is a $G$-homotopy equivalence.
Then we have
\[
k(X) \leq k(X').
\]
\end{lem}

\begin{proof}
We have a commutative diagram
\begin{align*}
  \begin{CD}
     \tilde{K}_G(X')   @>{f^\ast}>> \tilde{K}_G(X) \\
  @V{}VV    @VVV \\
\tilde{K}_G((X')^H)   @>{(f^{H})^\ast}>> \tilde{K}_G(X^H),   \\
  \end{CD}
\end{align*}
where $(f^H)^\ast$ is an isomorphism.
This implies $\mathfrak{J}(X') \subset \mathfrak{J}(X)$, and hence $k(X) \leq k(X')$.
\end{proof}

We need to know the behavior of the ideal $\mathfrak{J}(X)$ under suspension to define the  $K$-theoretic Fr{\o}yshov invariant $\kappa(Y,\iota)$:

\begin{lem}
\label{lem: shift by suspension}
 Let $X$ be a space of type $G$-SWF at an even level.
 Then we have 
 \begin{align}
     \label{ideal rel}
 \mathfrak{J}(\Sigma^{\tilde{\C}}X) = \mathfrak{J}(X),\quad
 \mathfrak{J}(\Sigma^{\C_+}X) = z\cdot \mathfrak{J}(X),\quad
 \mathfrak{J}(\Sigma^{\C_-}X) = (w+z-wz)\cdot \mathfrak{J}(X),
 \end{align}
 and 
  \begin{align}
  \label{k shift by suspension}
    k(\Sigma^{\tilde{\C}}X) = k(X),\quad
    k(\Sigma^{\C_+ \oplus \C_-}X) = k(X)+1.
  \end{align}
 \end{lem}

\begin{proof}
The proof is similar to the proof of \cite[Lemma~2]{Ma14}.
First we shall prove \eqref{ideal rel}.
The statement about $\tilde{\C}$ follows from
$(\Sigma^{\tilde{\C}}X)^H =  \Sigma^{\tilde{\C}}(X^H)$
 and the naturality of the Bott element.
 
 Next, to prove the statement about $\C_+$, consider the diagram induced from the inclusion $i : X \to \Sigma^{\C_+}X$ and the inclusions from the $H$-invariant part:
 \begin{align}
\label{diagram k sus}
  \begin{CD}
     \tilde{K}_G(\Sigma^{\C_+}X)   @>{i^\ast}>> \tilde{K}_G(X) \\
  @V{}VV    @VVV \\
\tilde{K}_G((\Sigma^{\C_+}X)^H)   @>{(i^{H})^\ast}>> \tilde{K}_G(X^H).   \\
  \end{CD}
\end{align}
Since $(\Sigma^{\C_+}X)^H = X^H$, the bottom map $(i^H)^\ast$ is the identity.
By the Bott isomorphism, the top map $i^\ast$ is given by multiplication by the Euler class $z$ associated to the representation $\C_+$.
This implies that
\[
\mathfrak{J}(\Sigma^{\C_+} X) = z \cdot \mathfrak{J}(X)
\]
in $R(G)$.

Similarly, about the statement on $\C_-$, 
we have
 \begin{align*}
  \begin{CD}
     \tilde{K}_G(\Sigma^{\C_-}X)   @>{i^\ast}>> \tilde{K}_G(X) \\
  @V{}VV    @VVV \\
\tilde{K}_G((\Sigma^{\C_-}X)^H)   @>{=}>> \tilde{K}_G(X^H).   \\
  \end{CD}
\end{align*}
The top map $i^\ast$ is given by multiplication by the Euler class $1-\C_-$, and by the relation $\C_-=\tilde{\C} \cdot \C_-$, 
this Euler class is given by
\[
1-\C_-
= 1-(1-w)(1-z)
= w+z-wz.
\]

Next we prove \eqref{k shift by suspension}.
The statement about $\tilde{\C}$ is clear by \eqref{ideal rel}.
To show the statement about $\C_+ \oplus \C_-$, 
first note that we have
\[
\mathfrak{J}(\Sigma^{\C_+ \oplus \C_-}X) = z(w+z-wz) \cdot \mathfrak{J}(X)
\]
by \eqref{ideal rel}.
Using the relations $w^2-2w=0, w-2z+z^2=0$,
we deduce that
\begin{align}
\label{rel wz...}
    w(z(w+z-wz))
    =zw^2+wz^2-w^2z^2
    =2zw-wz^2
    =2zw-2zw+w^2
    = 2w.
\end{align}
For $x \in \mathfrak{J}(X)$ and $k \geq 0$,
it follows from \eqref{rel wz...} that $wx = 2^{k}w$ holds if and only if
$w(z(w+z-wz))x = 2^{k+1}w$ holds.
Thus we have $k(\Sigma^{\C_+ \oplus \C_-}X) = k(X)+1$.
\end{proof} 
 
\begin{rem}
\label{rem: why not C_+ C_-}
We do not state the behavior of $k(X)$ under suspension by $\C_+$ and by $\C_-$ in
\cref{lem: shift by suspension}.
In fact, $k(X)$ does not behave well under these suspensions, not as in Manolescu's original argument.
(Compare this with \cite[Lemma~2]{Ma14}.)
For $\C_+$, the reason why $k(X)$ does not behave well is that the relation $zw-2w=0$ does not hold in $R(G)$, not as in $R(Pin(2))$.
For $\C_-$, the reason is that the behavior of $\mathfrak{J}(X)$ is already complicated as seen in  \eqref{ideal rel}.
\end{rem}

\begin{ex}
By definition, it is easy to see that $k(S^0)=0$. Combining this with \cref{lem: shift by suspension}, we have that
\begin{align}
\label{eq: easiest cal of k}
k\left((\tilde{\C}^s \oplus (\C_+\oplus \C_-)^l)^+\right)=l    
\end{align}
for $s, l \geq 0$.
\end{ex}

\begin{lem}
\label{lem: suspension how much}
 Let $X$ and $X'$ be spaces of type $G$-SWF at levels $2t$ and $2t'$, respectively, such
that $t < t'$.
Suppose that there exists a $G$-equivariant map $f : X \to X'$ whose $G$-fixed-point set map is a homotopy equivalence.
Then we have
\[
k(X) + t \leq k(X') + t'.
\]
\end{lem}

\begin{proof}
The proof is similar to that of \cite[Lemma~5]{Ma14}, but for readers' convenience, we spell out the proof.
Let us start with the diagram
\begin{align}
\label{diagram k}
  \begin{CD}
     \tilde{K}_G(X')   @>{f^\ast}>> \tilde{K}_G(X) \\
  @V{}VV    @VVV \\
\tilde{K}_G((X')^H)   @>{(f^{H})^\ast}>> \tilde{K}_G(X^H)   \\
 @V{}VV    @VVV \\
\tilde{K}_G((X')^G)   @>{(f^{G})^\ast}>> \tilde{K}_G(X^G).   \\ 
  \end{CD}
\end{align}
Since $\tilde{K}_G((X')^H), \tilde{K}_G(X^H), \tilde{K}_G((X')^G), \tilde{K}_G(X^G)$ are free $R(G)$-modules of rank $1$,
we may regard the four maps among them as multiplications by elements of $R(G)$.
Note that $(f^G)^\ast$ is just the identity map, since $f$ is supposed so that $f^G$ is a homotopy equivalence.
The vertical maps $\tilde{K}_G((X')^H) \to \tilde{K}_G((X')^G)$ and $\tilde{K}_G(X^H) \to \tilde{K}_G(X^G)$ are given by multiplication with the $K$-theoretic Euler classes $w^{t'}, w^{t}$ respectively.
Thus we have that $(f^H)^\ast$ is given as multiplication by an element $y \in R(G)$ satisfying that
\begin{align}
\label{relation wy}
    w^{t} \cdot y = w^{t'}.
\end{align}

We claim that $(f^H)^\ast$ in the diagram \eqref{diagram k} is given by multiplication by $2^{t'-t-1}w$, in other words, we show that $y=2^{t'-t-1}w$.
First, since $t<t'$, the map 
\[
(f^H)^\ast : \tilde{K}_H((X')^H) \to \tilde{K}_H(X^H)
\]
is the zero map.
Hence the image of $y$ under the restriction $R(G) \to R(H)$ is zero.
From this and \eqref{eq: ker rep g h} we deduce that $y = cw$ for some $c \in \Z$.
Then it follows from \eqref{relation wy} and the relation $w^2=2w$ that
\[
2^tcw = cw^{t+1} = w^{t'}
= 2^{t'-1}w,
\]
and thus we have $c=2^{t'-t-1}$, and $y=2^{t'-t-1}w$, as claimed.

Take $x \in \mathfrak{J}(X')$ such that $wx=2^{k'}w$, where $k'=k(X')$.
By the previous paragraph, we have 
\[
  (f^H)^\ast(x) = 2^{t'-t-1}wx=2^{k'+t'-t-1}w.
\]
It follows from this that
\[
w\cdot(f^H)^\ast(x) = 2^{k'+t'-t}w.
\]
On the other hand, by the commutativity of the diagram \eqref{diagram k}, $(f^H)^\ast(x)$ belongs to $\mathfrak{J}(X)$.
Thus we obtain $k'+t'-t \geq k(X)$.
\end{proof}

Next, following \cite{Ma14}, we consider duality in our setup.
See also \cite[Subsection 2.2]{Ma16}.

\begin{defi}[{\cite[Section XVI. 8]{May96}}]
Let $V$ be a finite dimensional representation of $G$ and $X, X'$ be pointed finite $G$-CW complexes.
We say that $X$ and $X'$ are {\it equivariantly $V$-dual} if there exist continuous pointed $G$-maps $\epsilon : X' \wedge X \to V^+$ and $\eta : V^+ \to X' \wedge X$ that make the following diagrams stably homotopy commutative, where $r : V^+ \to V^+$ denotes the map defined by $r(v)=-v$, and $\gamma$ denotes the transpositions:
\[
   \xymatrix{
    V^+ \wedge X \ar[r]^{\eta \wedge {\rm id}} \ar[rd]_\gamma & X\wedge X' \wedge X \ar[d]^{{\rm id} \wedge \epsilon}
    & X' \wedge V^+ \ar[r]^{{\rm id} \wedge \eta} \ar[d]_{\gamma} & X' \wedge X \wedge X \ar[d]^{\epsilon \wedge {\rm id}}\\
     & X\wedge V^+ & V^+ \wedge X' \ar[r]_{r \wedge {\rm id}} & V^+ \wedge X'
   }.
\]
\end{defi}

The following \lcnamecref{lem: duality origin} is is modeled on \cite[Lemma~7]{Ma14}.

\begin{lem}
\label{lem: duality origin}
 Let $X$ and $X'$ be spaces of type $G$-SWF at levels $2t$ and $2t'$, respectively.
Suppose that $X$ and $X'$ are equivariantly $V$-dual for some $G$-representation $V \cong \tilde{\C}^s \oplus (\C_+\oplus \C_-)^{l}$ with $s, l\geq 0$.
Then we have
\[
k(X) + k(X') \geq l.
\]
\end{lem}

\begin{proof}
Let $\epsilon$ and $\eta$ are the maps associated to the $V$-duality between $X$ and $X'$.
The restrictions of $\epsilon$ and $\eta$ induce a $V^{H}$-duality between $X^H$ and $(X')^{H}$, and this implies that $t+t'=s$.
Think of $\epsilon^H$ and $\eta^H$ as $G/H \cong \Z_2$-equivariant self-maps of $(\tilde{\C}^s)^+$.
The restriction of $\epsilon^H$ and $\eta^H$ on $G/H$-fixed-point sets gives rise to a self-bijection of $S^0$, which is a self-duality of $S^0$.
Thus we see that $\epsilon^H$ and $\eta^H$ are unreduced suspensions of self-maps of the unit sphere $S(\tilde{\C}^s)$, up to $G/H$-equivalence.
Since self-maps of $S(\tilde{\C}^s)$ are determined by their degree, so are $\epsilon^H$ and $\eta^H$.
The duality diagrams imply that the product of the degree of $\epsilon^H$ and that of $\eta^H$ is $\pm1 \in \Z$, and hence each degree is also $\pm1$.
This implies that $\epsilon^H$ and $\eta^H$ are $G/H$-homotopy equivalence, which is equivalent to that they are $G$-homotopy equivalence.
By \cref{lem: origin of local eq} applied to $\epsilon$ and $\eta$ and \eqref{eq: easiest cal of k}, we have that
\begin{align}
\label{eq: l dual ingredient 1}
k(X \wedge X') = k(V^+) = l.
\end{align}

On the other hand, for $x \in \mathfrak{J}(X)$ and $x' \in \mathfrak{J}(X')$ with $wx=2^{k(X)}w$
and $wx'=2^{k(X')}w$, we have that $xx' \in J(X \wedge X')$, where $xx'$ is the image of $x \otimes x'$ under a natural map $\tilde{K}_G(X) \otimes \tilde{K}_G(X') \to \tilde{K}_G(X \wedge X')$.
We also have $w(xx') = 2^{k(X)}wx'=2^{k(X)+k(X')}w$, and thus obtain $k(X \wedge X') \leq k(X)+k(X')$.
This combined with \eqref{eq: l dual ingredient 1} implies the desired conclusion.
\end{proof}

\subsection{Doubling construction and spectrum classes}
\label{subsection Doubling}

Given a spin rational homology 3-sphere $Y$ with a smooth odd involution $\iota$, in \cref{subsection $G$-equivariant Seiberg--Witten Floer $K$-theory},
we shall define the  Seiberg--Witten Floer $K$-theory and the $K$-theoretic Fr{\o}yshov invariant.
This is based on a space-valued 3-manifold invariant constructed in \cref{subsection $G$-equivariant Seiberg--Witten Floer homotopy type}, which is an analogue of Manolescu's stable homotopy type.
In this \lcnamecref{subsection Doubling}, we prepare a set in which the space-valued invariant takes value.

Fixing an $\iota$-invariant metric $g$ on $Y$, we obtain a $G$-equivariant Conley index $I^\mu_\lambda (Y, \iota, g)$ as explained.
One option of the definition of the Seiberg--Witten Floer $K$-theory is just the  $K$-theory of (a certain degree shift of) $I^\mu_\lambda (Y, \iota, g)$,
but it is not convenient to define the $K$-theoretic Fr{\o}yshov invariant.
The reason is that both of representations $\C_+$ and $\C_-$ may appear in the $I$-invariant part of the space of spinors,
and as we have seen in \cref{lem: shift by suspension}, the ideal $\mathfrak{J}(X)$ of $R(G)$ associated to a space of type $G$-SWF $X$ does not behave neatly under the suspension by $\C_-$,
and $k(X)$ does not behave well for both $\C_+$ and $\C_-$: see \cref{rem: why not C_+ C_-}.

However, under the suspension by the direct sum $\C_+ \oplus \C_-$, the quantity $k(X)$ {\it does} behave neatly, as seen in \cref{lem: shift by suspension}.
Moreover, $\C_+ \oplus \C_-$ appears as the complexification of $\C_+$, and of $\C_-$ as well.
This observation leads us to consider
the `complexification' or `double' of $I^\mu_\lambda (Y, \iota, g)$ and of the relative Bauer--Furuta invariant among them.

To do this, let us define the double of a general space of type $G$-SWF.
Define a group automorphism $\alpha : G \to G$ by $\alpha(j)=-j$.

\begin{defi}
Let $X$ be a space of type $G$-SWF at level $t$.
Denote by $X^\dagger$ the space of type $G$-SWF at level $t$ defined as the same topological space with $X$, but the $G$-action on $X$ is given by composing the original $G$ action on $X$ with $\alpha$.
Then $X \wedge X^\dagger$ is also a space of type $G$-SWF, at level $2t$.
Define the space of type $G$-SWF $D(X)$ by 
\[
D(X) = X \wedge X^\dagger,
\]
which we call the {\it double} of $X$.

Similarly, for a real or complex representation $V$ of $G$, define a representation $V^\dagger$ by the same vector space with $V$, but with $G$-action obtained by composing the original $G$ action on $V$ with $\alpha$.
Define a representation $D(V)$ of $G$ by 
\[
D(V) = V \oplus V^\dagger.
\]
\end{defi}

\begin{ex}
\label{ex double of rep}
Since the automorphism $\alpha : G \to G$ does not affect the real representation $\tilde{\R}$, we have
\[
D(\tilde{\R}^t) = \tilde{\R}^{2t} \cong \tilde{\C}^t
\]
for $t \geq 0$.

On the other hand, $\alpha$ swaps $\C_+$ for $\C_-$:
\begin{align*}
&(\C_+)^\dagger = \C_-,\\
&(\C_-)^\dagger = \C_+.
\end{align*}
Thus we have $G$-equivariant homeomorphisms
\begin{align*}
&D(\C_+) \cong D(\C_-) \cong \C_+ \oplus \C_- \cong (\C_+)_\C \cong (\C_-)_\C,
\end{align*}
where $(\C_\pm)_\C$ denotes the complexification of $\C_\pm$.
More generally, we have
\[
D(\C_+^m \oplus \C_-^n) \cong (\C_+\oplus \C_-)^{m+n}
\]
for $m,n \geq 0$.
\end{ex}

Following \cite[Section~4]{Ma14}, 
consider a triple $(D,m,n)$, where 
$D$ is a space of type $G$-SWF at an even level, and $m \in \Z$ and $n \in \Q$.
We have in mind the case that $D$ is given as $D=D(X)$ for some $X$, not necessarily at an even level.

\begin{defi}
For such triples $(D,m,n), (D',m',n')$,
we say that they are {\it $G$-stably equivalent} to each other if $n-n' \in \Z$ and there exist $M,N \geq 0$ and a $G$-homotopy equivalence
\[
\Sigma^{(M-m)\tilde{\C}} \Sigma^{(N-n)(\C_+ \oplus \C_-)}D
\to \Sigma^{(M-m')\tilde{\C}} \Sigma^{(N-n')(\C_+ \oplus \C_-)}D'.
\]
Define $\mathfrak{C}_{\iota}$ as the set of $G$-stable equivalence classes of triples $(D,m,n)$.
An element of $\mathfrak{C}_{\iota}$ is called  a {\it spectrum class}.
\end{defi}

Informally, we may think of the triple $(D,m,n)$ as the formal desuspension of $X$ by $m \tilde{\C}$ and by $n(\C_+ \oplus \C_-)$, so symbolically one may write
\[
(D,m,n) = \Sigma^{-m \tilde{\C}}\Sigma^{-n(\C_+ \oplus \C_-)}D,
\]
while we need to keep in mind that $n$ may not be an integer.

As well as the non-equivariant case, we can define the notion of local equivalence, which was introduced by Stoffregen~\cite{Sto20}, in our $G$-equivariant setting:

\begin{defi}
Let $(D,m,n), (D',m',n')$ be triples as above.
A {\it $G$-stable map} $(D,m,n) \to (D',m',n')$ is a based $G$-map
\[
\Sigma^{(M-m)\tilde{\C}} \Sigma^{(N-n)(\C_+ \oplus \C_-)}D
\to \Sigma^{(M-m')\tilde{\C}} \Sigma^{(N-n')(\C_+ \oplus \C_-)}D'
\]
for some $M,N \geq 0$.
A $G$-stable map $(D,m,n) \to (D',m',n')$ is called a {\it $G$-local map} if it induces a $G$-homotopy equivalence on the $H$-fixed-point sets.
We say that $(D,m,n)$ and $(D',m',n')$ are {\it $G$-locally equivalent} if there exist $G$-local maps $(D,m,n) \to (D',m',n')$ and $(D',m',n') \to (D,m,n)$.

The $G$-local equivalence is evidently an equivalence relation, and we call an equivalence class for this relation a {\it $G$-local equivalence class}.
The set of $G$-local equivalence classes is denoted by $\mathcal{LE}_G$. We write an element of $\mathcal{LE}_G$ by $[(D,m,n)]_{\text{loc}}$.  
Evidently the $G$-stable equivalence implies the $G$-local equivalence,
and we have a natural surjection $\mathfrak{C}_{\iota} \to \mathcal{LE}_G$.
\end{defi}

Once we have a $G$-stable map $f: (D, m,n) \to (D', m',n')$ written as  
\[
f: \Sigma^{(M-m)\tilde{\C}} \Sigma^{(N-n)(\C_+ \oplus \C_-)}D
\to \Sigma^{(M-m')\tilde{\C}} \Sigma^{(N-n')(\C_+ \oplus \C_-)}D', 
\]
by suspensions, we can enlarge $M$ and $N$ as arbitrary large positive integers. 
So, we can define the formal (de)suspension $\Sigma^{ l (\mathbb{C}_+ \oplus \mathbb{C}_-) } f$ of $f: (D, m,n-l) \to (D', m',n'-l)$ by 
 \[
\Sigma^{ l (\mathbb{C}_+ \oplus \mathbb{C}_-) } f : \Sigma^{(M-m)\tilde{\C}} \Sigma^{(N-n+l)(\C_+ \oplus \C_-)}D
\to \Sigma^{(M-m')\tilde{\C}} \Sigma^{(N-n'+ l )(\C_+ \oplus \C_-)}D'.
\]
\begin{rem}\label{local map desuspension}
It is easy to check that if $f$ is a $G$-local map, then $\Sigma^{ l (\mathbb{C}_+ \oplus \mathbb{C}_-) } f$ is a $G$-local map again. 
\end{rem}

For a triple $(D,m,n)$ above, we define
\[
\tilde{K}_G^\ast(D,m,n) := \tilde{K}^{\ast+m+2n}(D)
\]
and 
\[
k(D,m,n) := k(D) - n.
\]

The second statement of the following \lcnamecref{lem: inv k kappa} is
the main motivation that we work with $\mathfrak{C}_{\iota}$, allowing only suspension by $\tilde{\C}$ and $\C_+ \oplus \C_-$:

\begin{lem}
\label{lem: inv k kappa}
 Let $(D, m, n)$ be a triple as above. Then the followings are invariants of
the equivalence class $\mathcal{D} = [(D, m, n)] \in \mathfrak{C}_{\iota}$:
\begin{enumerate}
    \item The isomorphism class of equivariant $K$-cohomology,
    \[
    \tilde{K}_G^\ast(\mathcal{D}) = [\tilde{K}_G^\ast(D,m,n)],
    \]
    as a graded $R(G)$-module.
    \item The rational number $k(\mathcal{D}) = k(D,m,n) \in \Q$.
\end{enumerate}
\end{lem}

\begin{proof}
The first statement immediately follows from the Bott periodicity about suspensions by complex representations.
The second statement follows from \cref{lem: shift by suspension}.
\end{proof}

\begin{rem}
Instead of considering elements of $\mathfrak{C}_{\iota}$,
one may define a `$G$-equivariant suspension spectrum' by allowing suspension only by $\C_+$ and only by $\C_-$, not necessarily by the pair $\C_+\oplus \C_-$.
However, then the statement corresponding to (2) in \cref{lem: inv k kappa} cannot be obtained.
See \cref{rem: why not C_+ C_-}.
\end{rem}

The statement for (2) in \cref{lem: inv k kappa} can be improved as follows:

\begin{lem}\label{lem:loc eq}
The rational number $k(\mathcal{D})$ is a $G$-local equivalence invariant, i.e. $k(\mathcal{D})$ depends only on the $G$-local equivalence class of $\mathcal{D} \in \mathfrak{C}_{\iota}$.
\end{lem}

\begin{proof}
Given two triples $(D, m, n), (D', m', n')$, suppose that
they are $G$-local equivariant.
Let
\[
\Sigma^{(M-m)\tilde{\C}} \Sigma^{(N-n)(\C_+ \oplus \C_-)}D
\to \Sigma^{(M-m')\tilde{\C}} \Sigma^{(N-n')(\C_+ \oplus \C_-)}D'
\]
be a $G$-local map from $(D,m,n)$ to $(D',m',n')$, where $M,N \geq 0$.
Applying \cref{lem: origin of local eq} to this,
we obtain
\[
k(\Sigma^{(M-m)\tilde{\C}} \Sigma^{(N-n)(\C_+ \oplus \C_-)}D)
\leq k(\Sigma^{(M-m')\tilde{\C}} \Sigma^{(N-n')(\C_+ \oplus \C_-)}D').
\]
By \cref{lem: shift by suspension}, this is equivalent to that $k(D,m,n) \leq k(D',m',n')$.
Similarly we obtain $k(D',m',n') \leq k(D,m,n)$ from a $G$-local map $(D', m', n') \to (D', m', n')$, and thus have $k(D,m,n) = k(D',m',n')$.
\end{proof}

\subsection{Doubled  Seiberg--Witten Floer stable homotopy type for involutions}
\label{subsection $G$-equivariant Seiberg--Witten Floer homotopy type}

Now we are ready to construct an invariant of $3$-manifolds with involution.
Let $(Y,\frakt)$ be a spin rational homology 3-sphere and $\iota$ be a smooth orientation-preserving involution on $Y$.
Suppose that $\iota$ also preserves the given spin structure $\frakt$ on $Y$ and is of odd type.
Fix an $\iota$-invariant metric $g$ on $Y$, and
choose a lift $\tilde{\iota}$ of $\iota$ of order 4 on the spin structure.
Then we obtain a $G$-equivariant Conley index $I^\mu_\lambda (Y, \iota, \tilde{\iota}, g)$, once we fix $\lambda \ll 0 \ll \mu$.
Recall that, 
the finite-dimensional approximation of
the configuration space is decomposed into 
\begin{align}
\label{eq: fin. dim. approx V}
V_\lambda^\mu = 
V(\tilde{\R})^\mu_\lambda \oplus V(\C_+)^\mu_\lambda\oplus
V(\C_-)^\mu_\lambda.
\end{align}
Here each direct summand is isomorphic to the direct sum of some copies of $\tilde{\R}, \C_+, \C_-$ respectively.
The $H$-invariant part of $I^\mu_\lambda (Y, \iota, g)$ is given by $V(\tilde{\R})^\mu_\lambda$.

First, we see that the dependence on the choice of lift $\tilde{\iota}$.
Recall that $\iota$ has exactly two lifts, and once we pick a lift $\tilde{\iota}$, the other lift is $-\tilde{\iota}$.

\begin{lem}
\label{lem: indep tilde iota}
The double $D(I^\mu_\lambda (Y,\frakt, \iota, \tilde{\iota}, g))$ is independent of the choice of $\tilde{\iota}$.
Namely, there is a canonical $G$-equivariant homeomorphism
\begin{align}
\label{eq: indep tilde iota}
D(I^\mu_\lambda (Y,\frakt, \iota, \tilde{\iota}, g))
\cong D(I^\mu_\lambda (Y,\frakt, \iota, -\tilde{\iota}, g)).    
\end{align}
\end{lem}

\begin{proof}
To record the dependence,
denote the involution $I$ introduced in \eqref{eq: def of I} by $I^+$ if it is defined using $\tilde{\iota}$, and by $I^-$ if it is defined using $-\tilde{\iota}$.
Then it is straightforward to check that 
\[
I^+(a,\phi i) = I^-(a,\phi)i
\]
for all $(a,\phi) \in \Omega^1(Y) \oplus \Gamma(\mathbb{S})$.

Recall also how the representations $\C_+, \C_-$ appear in 
the configuration space for the Seiberg--Witten Floer theory.
Decompose $\quat$ so that
\begin{align*}
\quat 
= \C \oplus j\C
= \R \oplus i\R \oplus j\R \oplus (-k)\R
\cong (\R \oplus j\R) \oplus i(\R\oplus(-j)\R).    
\end{align*}
Then $(\R \oplus j\R, j)$ and $(i(\R\oplus(-j)\R), -j)$ are isomorphic to $\C_+$ and $\C_-$ respectively.
Note that the right multiplication of $i$ gives complex linear isomorphisms
\[
(\R \oplus j\R, j) \to (i(\R\oplus(-j)\R), -j),\quad (i(\R\oplus(-j)\R), -j) \to (\R \oplus j\R, j).
\]

Denote a finite-dimensional approximation of the $I^\pm$-invariant part of the configuration space by
\[
V(\tilde{\R})^\mu_\lambda \oplus V^\pm(\C_+)^\mu_\lambda\oplus
V^\pm(\C_-)^\mu_\lambda.
\]
(Note that the choice of lift of $\iota$ does not affect the $H$-invariant part $V(\tilde{\R})^\mu_\lambda$.)
By the above two paragraphs, we see that the right multiplication of $i$ gives complex linear isomorphisms
\[
V^+(\C_+)^\mu_\lambda \to V^-(\C_-)^\mu_\lambda,\quad
V^+(\C_-)^\mu_\lambda \to V^-(\C_+)^\mu_\lambda.
\]
Therefore the right multiplication of $i$ induces a $G$-equivariant homeomorphism
\[
I^\mu_\lambda (Y,\frakt, \iota, \tilde{\iota}, g)
\to I^\mu_\lambda (Y, \frakt,\iota, -\tilde{\iota}, g)^\dagger,
\]
and this gives rise to the desired $G$-equivariant homeomorphism \eqref{eq: indep tilde iota}.
\end{proof}

In view of \cref{lem: indep tilde iota}, henceforth we drop from our notation the choice of lift of $\iota$.

We also need a kind of correction term to absorb the dependence of the invariant metric.
Recall a correction term introduced in \cite{Ma03}:
\[
n(Y, \frakt,g) := \frac{1}{2}(\eta_{dir}(Y, \frakt,g)-\dim_{\C}\Ker D_{(Y, \frakt,g)}-\frac{1}{4}\eta_{sign}(Y,g)).
\]
Here $\eta_{dir}(Y, \frakt,g)$ and $\eta_{sign}(Y,g)$ are the eta invariants of the Dirac operator and the signature operator respectively, and $D_{(Y, \frakt,g)}$ is the Dirac operator on $(Y, \frakt,g)$.
Alternatively, we can write $n(Y, \frakt,g)$ as 
\begin{align}
\label{eq: expression of n for Rohlin}
n(Y, \frakt,g) = \ind_\C D_W + \frac{\sigma (W)}{8},    
\end{align}
where $W$ is a compact spin Riemann 4-manifold bounded by $(Y, g)$. 
About orientation reversing, we have 
\begin{align}
\label{eq: nminusn}
n(Y, \frakt,g) + n(-Y, \frakt,g) = \dim_{\C}\Ker D_{(Y, \frakt,g)},    
\end{align}
where $D_{(Y, \frakt,g)}$ is the 3-dimensional Dirac operator on $(Y, \frakt,g)$ (see \cite[Page 167]{Ma16}).

In our $G$-equivariant setting,
a direct analogue of Manolescu's Fleor stable homotopy type is a triple
\begin{align}
\label{eq: candidate}
(\Sigma^{-V(\tilde{\R})^0_\lambda}
\Sigma^{-V(\C_+)^0_\lambda}
\Sigma^{-V(\C_-)^0_\lambda}I^\mu_\lambda (Y, \frakt, \iota, g),0,n(Y, \frakt,g)/4), 
\end{align}
where the division by $4$ for the last factor will be explained in \cref{rem: divide by 4}.
However, the triple \eqref{eq: candidate} does not lie in $\mathfrak{C}_{\iota}$.
So, instead, we consider the `double' of this triple:
\begin{align}
\label{triple before formal def}
(\Sigma^{-D(V(\tilde{\R})^0_\lambda)}
\Sigma^{-D(V(\C_+)^0_\lambda)}
\Sigma^{-D(V(\C_-)^0_\lambda)}D(I^\mu_\lambda (Y, \frakt, \iota, g)),0,n(Y, \frakt,g)/2).
\end{align}
Rewriting desuspensions to make \eqref{triple before formal def} precise, in view of \cref{ex double of rep}, we arrive at:

\begin{defi}\label{def: equivariant Seiberg--Witten Floer stable homotopy type}
Given $Y, \frakt, \iota, g$ as above, define an element $DSWF_G(Y, \frakt,\iota) \in \mathfrak{C}_{\iota}$ by 
\begin{align*}
&DSWF_G(Y, \frakt,\iota)\\
&:= 
[(D(I^\mu_\lambda (Y, \frakt, \iota, g)),
\dim_{\R}V(\tilde{\R})^0_\lambda,
\dim_{\C}V(\C_+)^0_\lambda
+\dim_{\C}V(\C_-)^0_\lambda
+n(Y, \frakt,g)/2)].
\end{align*}
We call $DSWF_G(Y, \frakt,\iota)$
the {\it doubled $G$-equivariant Seiberg--Witten Floer stable homotopy type} or {\it doubled Seiberg--Witten Floer $G$-spectrum class}.
\end{defi}

Recall that, symbolically, $DSWF_G(Y, \frakt,\iota)$ can be thought of as 
\[
\Sigma^{-(\dim_{\R}V(\tilde{\R})^0_\lambda)\tilde{\C}}
\Sigma^{-(\dim_{\C}V(\C_+)^0_\lambda
+\dim_{\C}V(\C_-)^0_\lambda
+n(Y, \frakt,g)/2)(\C_+ \oplus \C_-)}
D(I^\mu_\lambda (Y, \frakt, \iota, g)).
\]

\begin{rem}
\label{rem: divide by 4}
The number $n(Y, \frakt,g)/4$ in \eqref{eq: candidate} was chosen by the following observation.
First, since $I$ and $i$ anti-commute, the dimension of $I$-invariant part of the kernel/cokernel of the Dirac operator is half of the dimension of the original kernel/cokernel.
Second, passing to the double, the dimension turns into the double of the original one.
Lastly, for suspensions corresponding to the third factor of triples, we use the direct sum $\C_+ \oplus \C_-$, rather than a $1$-dimensional complex vector space.
So in total
\[
\frac{1}{2}\cdot2\cdot \frac{1}{2}n(Y, \frakt,g) = \frac{1}{2}n(Y, \frakt,g)
\]
should be put at the third factor of the triple \eqref{triple before formal def}.
Adopting this number is necessary to prove the invariance of the doubled $G$-Floer homotopy type, shown in the following \lcnamecref{lem inv lamda mu}.
\end{rem}

We shall show the invariance of $DSWF_G(Y, \frakt,\iota)$.
Before that, we note a technicality about trivializations of representations.
Since $\tilde{\C}$ is a complex representation of $G$ and $GL_(N,\C)$ is connected for all $N$, for any triple $(D,m,n)$ and
for any complex $N$-dimensional representation $V$ of $G$ which is isomorphic to the direct sum of copies of $\tilde{\C}$, we have 
a $G$-stable  equivalence
\[
(V^+ \wedge D, N + m, n) \simeq (D, m, n),
\]
and this $G$-stable  equivalence is canonical up to homotopy.
A similar remark applies also to the desuspension by $\C_+\oplus \C_-$ since $\C_+\oplus \C_-$ is also a complex representation.

\begin{prop}
\label{lem inv lamda mu}
The spectrum class $DSWF_G(Y, \frakt,\iota) \in \mathfrak{C}_{\iota}$ is an invariant of $(Y, \frakt,\iota)$, independent of $\lambda, \mu$, and $g$.
\end{prop}

\begin{proof}
We basically follow the original argument by Manolescu~\cite[Proof of Theorem~1]{Ma14}.
First we fix $g$ and show the independence on $\lambda, \mu$. 
Recall the behavior of the Conley index under suspension.
That is, for $\lambda<\lambda'\ll0\ll \mu'<\mu$,
\begin{align*}
   &I^{\mu}_{\lambda} = I^{\mu'}_{\lambda},\\
   &I^{\mu}_{\lambda} \cong I^{\mu}_{\lambda'} \wedge (V^{\lambda'}_{\lambda})^+
   = I^{\mu}_{\lambda'} \wedge (V(\tilde{\R})^{\lambda'}_{\lambda} \oplus V(\C_+)^{\lambda'}_{\lambda}\oplus
   V(\C_-)^{\lambda'}_{\lambda})^+.
\end{align*}
This combined with \cref{ex double of rep} implies that
\begin{align*}
   D(I^{\mu}_{\lambda}) 
   &= D(I^{\mu'}_{\lambda}),\\
   D(I^{\mu}_{\lambda}) 
   &\cong D(I^{\mu}_{\lambda'}) \wedge (D(V^{\lambda'}_{\lambda}))^+\\
   &\cong D(I^{\mu}_{\lambda'}) \wedge \left(\wt{\C}^{\dim_{\R}V(\tilde{\R})^{\lambda'}_{\lambda}} \oplus 
   (\C_+ \oplus \C_-)^{\dim_{\C}V(\C_+)^{\lambda'}_{\lambda}+
   \dim_{\C}V(\C_-)^{\lambda'}_{\lambda}}\right)^+,
\end{align*}
and from this the independence on $\lambda, \mu$ follows.

Next, we show the independece on $g$.
Fix $\lambda, \mu$, and to record the choice of $g$ used in the construction, we temporarily denote by $DSWF_G(Y, \frakt,\iota,g)$ the doubled $G$-spectrum class constructed being used $\lambda, \mu$.
Take two $\iota$-invariant metrics $g_0$ and $g_1$.
Since the space of $G$-invariant metrics is contractible, we may take a path of $\iota$-invariant metrics $\{g_t\}_{t \in [0,1]}$ between $g_0$ and $g_1$.
The assumption that $b_1(Y)=0$ tells us that the de Rham operator does not involve here, as in the non-equivariant case, and we focus on the family of Dirac operators $D_{t}$ associated with this path of metrics.

Recall that the $I$-invariant part of each Dirac operator $D_t^I$ is $G$-equivariant, and we obtain $G$-equivariant spectral flow
\[
{\rm sf}_G(\{D_{t}^I\}) \in R(G),
\]
which can be written as a linear combination only of $\C_+, \C_-$.
Under the natural map $R(G) \to \Z$ induced from taking the dimension of the representation, the image of the equivariant spectral flow ${\rm sf}_G(\{D_{t}^I\})$ is given by the usual (non-equivariant) spectral flow ${\rm sf}(\{D_{t}^I\}) \in \Z$.
The difference between the doubled $G$-spectrum classes $DSWF_G(Y, \frakt,\iota,g_0)$ and $DSWF_G(Y, \frakt,\iota,g_1)$ is given by the `doubled' equivariant spectral flow
\[
D({\rm sf}_G(\{D_{t}^I\})) \in R(G),
\]
which is given by
\begin{align}
\label{doubled equiv spec}
 (\C_+ \oplus \C_-)^{{\rm sf}(\{D_{t}^I\})}.   
\end{align}
Recall that we have a formula
\begin{align}
\label{eq: split spec}
{\rm sf}(\{D_{t}\}) = n(Y, \frakt,g_1) - n(Y, \frakt,g_0),
\end{align}
and thus have
\[
{\rm sf}(\{D_{t}^I\}) = n(Y, \frakt,g_1)/2 - n(Y, \frakt,g_0)/2.
\]
Note that ${\rm sf}(\{D_{t}\})$ is an even integer because of the existence of the action $I$.
Thus we finally get
\[
D({\rm sf}_G(\{D_{t}^I\})) 
= (\C_+ \oplus \C_-)^{n(Y, \frakt,g_1)/2 - n(Y, \frakt,g_0)/2}
\in R(G),
\]
and this completes the proof.
\end{proof}

\begin{rem}
As noted in \cite{BH21}, it is subtle to split the equivariant spectral flow ${\rm sf}_G(\{D_{t}\})$ as in the non-equivariant case \eqref{eq: split spec}.
But here we consider the doubled homotopy type, and the doubling construction makes the subtlety from the equivariance disappear:
the difference between two doubled homotopy types corresponding to two choices of metrics is determined only by the non-equivariant spectral flow, as seen in \eqref{doubled equiv spec}.
\end{rem}

\subsection{Seiberg--Witten Floer $K$-theory for involutions}
\label{subsection $G$-equivariant Seiberg--Witten Floer $K$-theory}

As in the last subsection, let $(Y, \frakt)$ be a spin rational homology 3-sphere and $\iota$ be a smooth orientation-preserving involution $\iota$.
Suppose that $\iota$ also preserves the given spin structure $\frakt$ and is of odd type.

\begin{defi}
Define the {\it doubled  Seiberg--Witten Floer $K$-cohomology} by
\[
DSWFK_G(Y, \frakt,\iota) := \tilde{K}_G(DSWF_G(Y, \frakt,\iota)),
\]
defined as the isomorphism class of an $R(G)$-graded module.
We define also the {\it  $K$-theoretic Fr{\o}yshov invariant} by
\[
\kappa(Y, \frakt,\iota) := k(DSWF_G(Y, \frakt,\iota)) \in \Q.
\]
\end{defi}


\begin{lem}
\label{lem: invariance of k-theory}
The isomorphism class $DSWFK_G(Y, \frakt,\iota)$ and the rational number $\kappa(Y, \frakt,\iota)$ are invariant of $(Y, \frakt,\iota)$.
\end{lem}

\begin{proof}
This is a direct consequence of \cref{lem: inv k kappa,lem inv lamda mu}.
\end{proof}

If we allow us to use additional non-topological data $g, \lambda, \mu$, the invariant $\kappa(Y,\frakt,\iota)$ is concretely described as
\begin{align}
\label{eq: kappa describe}   
\kappa(Y,\frakt,\iota)
= k(D(I^\mu_\lambda(Y,\frakt,\iota,g)))-\dim_\C(V(\C_+)^0_\lambda)-\dim_\C(V(\C_-)^0_\lambda)-n(Y,\frakt,g)/2.
\end{align}

\begin{rem}
In this paper, we apply $K$-theory to the doubled or ``complexification" of the $I$-fixed point part of the Seiberg--Witten Floer stable homotopy type to define the invariant $\kappa$. In our setting, it could be able to apply $KR$-theory, which was introduced by Atiyah in~\cite{At66}, to the whole space of Seiberg--Witten Floer stable homotopy type with involution $I$.
\end{rem}

\begin{prop}
\label{prop: from duality}
Let $(Y,\frakt,\iota)$ be an oriented spin rational homology sphere with odd involution.
Then we have
\[
\kappa(Y,\frakt,\iota) + \kappa(-Y,\frakt,\iota)
\geq 0.
\]
\end{prop}

\begin{proof}
Denote by $\bar{V}_{-\lambda}^\mu$ for the finite-dimensional vector space for $-Y$ used in the definition of the $I$-invariant part Conley index.
Set $V(\C)^{\mu}_{\lambda} = V(\C_+)^\mu_{\lambda} \oplus V(\C_-)^\mu_{\lambda}$ in the expression \eqref{eq: fin. dim. approx V}.
Note that we have 
\[
\dim_{\C} V(\C)^0_{\lambda} + \dim_{\C} \bar{V}(\C)^0_{-\mu} + \frac{1}{2}\dim_{\C}\Ker D_{(Y,\frakt,g)} = \dim_{\C} V(\C)^\mu_{\lambda},
\]
where $D_{(Y, \frakt,g)}$ is the 3-dimensional Dirac operator on $(Y, \frakt,g)$.
(The coefficient $1/2$ emerges from that we took the $I$-invariant part.)
Combining this with \eqref{eq: nminusn}, we obtain
\begin{align}
\label{eq: n n pm}
\dim_{\C} V(\C)^0_{\lambda} + \dim_{\C} \bar{V}(\C)^0_{-\mu} + \frac{1}{2}(n(Y,\frakt,g)+n(-Y,\frakt,g)) = \dim_{\C} V(\C)^\mu_{\lambda}.
\end{align}

Now, in the expression \eqref{eq: fin. dim. approx V}, take $\mu=-\lambda$.
As described in \cite[Subsection~3.6]{Ma16},
the doubled $I$-invariant part Conley indices $D(I_{-\mu}^\mu(Y, \frakt, \iota,g))$ and $D(I_{-\mu}^\mu(-Y, \frakt, \iota,g))$ are equivariantly $D(V^\mu_{-\mu})$-dual.
It follows from this combined with \cref{lem: duality origin}, \eqref{eq: kappa describe}, and \eqref{eq: n n pm} that
\begin{align*}
&\kappa(Y,\frakt,\iota) + \kappa(-Y,\frakt,\iota)\\
=&k(D(I^\mu_{-\mu}(Y,\frakt,\iota,g)))+k(D(I^\mu_{-\mu}(-Y,\frakt,\iota,g)))
\\
&-(\dim_\C(V(\C)^0_{-\mu})+\dim_\C(\bar{V}(\C)^0_{-\mu}))-(n(Y,\frakt,g)+n(-Y,\frakt,g))/2\\
=&k(D(I^\mu_{-\mu}(Y,\frakt,\iota,g)))+k(D(I^\mu_{-\mu}(-Y,\frakt,\iota,g)))-\dim_\C(V(\C)^\mu_{-\mu}) \\
\geq& \dim_\C(V(\C)^\mu_{-\mu})  -\dim_\C(V(\C)^\mu_{-\mu})=0.
\end{align*}
This completes the proof.
\end{proof}

\subsection{Cobordisms}
\label{subsection: cobordisms}
Let $(Y_0, \frakt_0,\iota_0)$ and $(Y_1, \frakt_1,\iota_1)$ 
be spin closed $3$-manifold with $b_1(Y)=0$. We do not assume that $Y_0$ and $Y_1$ are connected. 
Suppose that we have an involution $\iota_i$ on each of $Y_i$.
Let $(W,\fraks)$ be a smooth spin 4-dimensional oriented cobordism with $b_1(W)=0$.
We assume that there is a involution $\iota$ on $W$ such that $\iota|_{Y_i}=\iota_i$ for $i=0,1$, and suppose that $\iota$ preserves $\fraks$ and of odd type.
We may take an $\iota$-invariant Riemannian metric $g$ on $W$ so that $g$ is a cylindrical metric near $\partial W$.
Then the metrics defined by $g_i=g|_{Y_i}$ on $Y_i$ are $\iota_i$-invariant metrics.
Then $\iota$ lifts to some $\tilde{\iota}$, which is a $\Z_4$-lift of $\iota$ to the spinor bundle of $\mathfrak{s}$. 
As well as the case of dimension 3, 
following \cite{Ka17},
we may define the involutions 
\begin{align*}
&I : \Omega^\ast(W) \to \Omega^\ast(W),\\  
&I : \Gamma(\mathbb{S}^{\pm}) \to \Gamma(\mathbb{S}^{\pm})
\end{align*}
by
\begin{align*}
&I(a)=(-\iota^{\ast}a),\\
&I(\phi) = \wt{\iota}( \phi)\cdot j,
\end{align*}
where $\mathbb{S}^{\pm}$ are positive and negative spinor bundles.
Here we consider the Sobolev norms $L^2_{k}$ for the spaces $\Omega^\ast(W)$ and $\Gamma(\mathbb{S}^{\pm})$ obtained from $\iota$-invariant metrics and $\iota$-invariant connections for a fixed integer $ k \geq 3$.
The relative Bauer--Furuta invariant of $W$ introduced by Manolescu~\cite{Ma03} gives a map between the Seiberg--Witten Floer stable homotopy types of $Y_0$ and $Y_1$.
This is obtained from the Seiberg--Witten map on $W$, which is given as a finite-dimensional approximation of a map
\begin{align}
    \label{eq: SW map}
SW : \Omega_{CC}^1(W) \times \Gamma(\mathbb{S}^+)
\to \Omega^+(W) \times \Gamma(\mathbb{S}^-) \times \hat{V}^{\mu}_{-\infty}(-Y_0) \times \hat{V}^{\mu}_{-\infty}(Y_1),
\end{align}
for large $\mu$. Here \[\Omega_{CC}^1(W)=\left\{ a \in \Omega^1(W) \mid d^{\ast} a=0, d^{\ast}\textbf{t}_i a=0, \int_{Y_i} \textbf{t}_i \ast a=0. \right\}\] is the space of all $1$-forms satisfying the double coulomb condition. This is introduced by T.~Khandhawit in \cite{Kha15}. 
Here for a general rational spin homology sphere $Y$, $\hat{V}(Y, \frakt)^\mu_{-\infty}$ is a subspace of
\[
\hat{V}(Y, \frakt) := \operatorname{Ker}d^* \times \Gamma(\mathbb{S}),
\]
which is defined as the direct sum of eigenspaces whose eigenvalues are less than $\mu$.
The $\Omega^+(W) \times \Gamma(\mathbb{S}^-)$-factor of the map $SW$ is given as the Seiberg--Witten equations,
and the $V^{\mu}_{-\infty}(-Y_0) \times V^{\mu}_{-\infty}(Y_1)$-factor is given, roughly, as the restriction of 4-dimensional configurations to 3-dimensional ones.
Taking the $I$-invariant part of \eqref{eq: SW map},
we obtain a $G$-equivariant map, and a finite-dimensional approximation of this gives us a $G$-equivariant map of the form
\begin{align}
\label{eq: cob map}
f : \Sigma^{m_0\tilde{\R}}\Sigma^{n_0^+\C_+}\Sigma^{n_0^-\C_+}I_{-\mu}^{-\lambda} (Y_0) \to \Sigma^{m_1\tilde{\R}}\Sigma^{n_1^+\C_+}\Sigma^{n_1^-\C_+}I^\mu_\lambda (Y_1),
\end{align}
where $I^\mu_\lambda (Y_i) = I^\mu_\lambda (Y_i, \frakt_i, \iota_i, g_i)$, and $m_i, n_i^\pm \geq 0$ and $-\lambda, \mu$ are sufficiently large.
Taking the double of $f$, we obtain the `doubled cobordism map' or the `doubled relative Bauer--Furuta invariant' 
\begin{align}
\label{eq: doubled cob map}
D(f) : 
\Sigma^{m_0\tilde{\C}}\Sigma^{n_0(\C_+ \oplus \C_-)}D(I_{-\mu}^{-\lambda} (Y_0)) \to \Sigma^{m_1\tilde{\C}}\Sigma^{n_1(\C_+ \oplus \C_-)}D(I^\mu_\lambda (Y_1)),    
\end{align}
where $n_i=n_i^+ + n_i^-$.
Denote by $V_0(\tilde{\R})_\lambda^\mu$ is the vector space $V(\tilde{\R})_\lambda^\mu$ for $Y_0$, and let us use similar notations also for other representaions and for $Y_1$.

\begin{rem}
\label{rem: integrality}
Note that, since the Dirac index $\ind_{\C}D$ is even by the existence of the quaternionic structure, we have that
\[
- \frac{\sigma(W)}{16}+ \frac{n(Y_1, \mathfrak{t}_1, g_1)}{2}-\frac{n(Y_0, \mathfrak{t}_0, g_0)}{2}
\]
is an integer.
\end{rem}

\begin{lem}
We have
\begin{align}
\label{eq:m difference}
m_0-m_1=&\dim_\R (V_1(\tilde{\R})^0_{\lambda})-\dim_\R(V_0(\tilde{\R})^0_{-\mu}) - b^+(W) + b^+_\iota(W), \\
\begin{split}
n_0-n_1=&\dim_\C(V_1(\C_+)^0_{\lambda}) +\dim_\C(V_1(\C_-)^0_{\lambda})\\
&-\dim_\C(V_0(\C_+)^0_{-\mu})-\dim_\C(V_0(\C_-)^0_{-\mu})\\
&- \frac{\sigma(W)}{16}+ \frac{n(Y_1, \mathfrak{t}_1, g_1)}{2}-\frac{n(Y_0, \mathfrak{t}_0, g_0)}{2}.\label{eq:n difference}
\end{split}
\end{align}
\end{lem}
\begin{proof}
It is sufficient to show the equality when $\mu$ and $\lambda$ are not eigenvalue. The proof of this lemma is parallel to the proof of \cite[Proposition 2]{Kha15}. 
We will denote by $U_i$ the space $\text{Im}(d)_{Y_i}^I$. Let $r_i$ be the restriction to $Y_i$ and $\Pi^-$ be the projection from $\Omega_1(Y_i)^I \times \Gamma(\mathbb{S})^I$ to $\hat{V}^{0}_{-\infty}(Y_i)^I$ and let $\Pi_2$ be the projection from $\Omega_1(Y_i)^I \times \Gamma(\mathbb{S})^I$ to $U_i$. Let $\textbf{t}a$ be the tangent component of the restriction of a $1$-form $a$ to the boundary. 
When $\mu=0$, the index of the linearlization of the $I$-invariant part of the map (\ref{eq: SW map}) coincides with the index of the map
\begin{align}
 \label{eq:ExSWmap}
 F \colon \Omega^1(W)^I \times \Gamma(\mathbb{S}^+)^I
\to 
\Omega^+(W)^I \times \Gamma(\mathbb{S}^-)^I \times \Omega^0(W)^I \times  \hat{V}^{0}_{-\infty}(-Y_0)^I \times \hat{V}^{0}_{-\infty}(Y_1)^I \times U_0
 \times U_1
\end{align}
given by 
\[
F(a, \phi)=(d^+a,\; D\phi,\; d^{\ast}a,\; \Pi^- \circ r_0(a,\phi),\; \Pi^- \circ r_1(a, \phi),\; \Pi_2 \circ r_0(a, \phi), \;\Pi_2 \circ r_1(a, \phi))
\]
because the kernel of the map $d^{\ast}\oplus \Pi_2 \circ r_0 \oplus \Pi_2 \circ r_1$ coincides with $\Omega_{CC}^1(W)^I \times \Gamma(\mathbb{S})^I$ and the cokernel of $d^{\ast}\oplus \Pi_2 \circ r_0 \oplus \Pi_2 \circ r_1$ is $0$ because constant functions on $W$ are not fixed by $-\iota^{\ast}$. Note that $\int_{Y_i} \ast \textbf{t}a=0$ for all $a \in \Omega^1(W)^I$. Let $L_1$ be the operator acting on $ \text{Im}(d)_{Y_i}^I \times 
\Omega_0(Y_i)^I$ by 
\[L_1=
\begin{pmatrix}
0&-d\\
-d^{\ast}&0
\end{pmatrix}.
\]
Let $\Pi^-_1$ be the projection to the non-negative eigenspace of the operator $L_1$. 
The kernel of the projection $\Pi_2 |_{\Omega_1(Y_i)^I}$ is $\text{Ker}(d^{\ast})^I_{Y_i}$ and the image of the $\Pi_1^-$ is 
\[
\{(b, d^{\ast}(dd^{\ast})^{-1/2}b) \mid b \in \Omega_1 (Y_i)^I \}.
\]
The kernel of $(\Pi_2|_{\Omega_1(Y_i)^I})\circ r_i$ is complementary to the non-positive eigenspace of $L_1$ in the space $\Omega_1(Y_i)^I \times \Omega_0(Y_i)^I$. By the \cite[Proposition 17.2.6]{KM07}, we have the operator $F$ is Fredholm and the index coincides with the index of the following operator: 
\begin{align}
 \label{eq:APSmap}
 F' \colon \Omega^1(W)^I \times \Gamma(\mathbb{S}^+)^I
\to 
\Omega^+(W)^I \times \Gamma(\mathbb{S}^-)^I \times \Omega^0(W)^I \times  
\hat{U}^0_{-\infty}(-Y_0) \times \hat{U}^0_{-\infty}(Y_1)
\end{align}
Here $\hat{U}^0_{-\infty}(-Y_0)$ and $\hat{U}^0_{-\infty}(Y_1)$ are the sum of the non-positive eigenspaces in $\Omega_1(Y_0)^I \times \Omega_0(Y_0)^I \times \Gamma(Y_0, \mathbb{S})^I$ and $\Omega_1(Y_1)^I \times \Omega_0(Y_1)^I \times \Gamma(Y_1, \mathbb{S})^I$ of $(L_1, D)^I$ respectively. We have the index of the spinor part of $F'$ is the half of the Atiyah-Patodi-Singer index of the $\text{Spin}^c$ Dirac operator because $I$ is anti-commutes with $i$. We see that the index of the form part of $F'$ is given by  $- b^+(W) + b^+_\iota(W)$ as follows: Elements in the kernel and the cokernel of $F'$ are able to be extended to a $L^2$ harmonic form of $\hat{W}=(-\infty, 0] \times Y_0 \cup_{Y_0} W \cup_{Y_1} Y_1 \times [0, \infty) $ because $Y_i$ are rational homology spheres. We see that the $L^2$ harmonic form on $\hat{W}$ coincides with the de Rham cohomology of $\hat{W}$ from in \cite[Proposition 4.9]{APSI} and the kernel and cokernel of $F'$ are coinsides with the $-\iota^{\ast}$-invariant part of the space of $L^2$ harmonic forms. Thus we have the index of the form part of $F'$. 

In the case $\mu >0$, we have the index of the $\text{Spin}^c$ Dirac part of $F'$ is $\text{ind}_{\C}(D^+)-\dim_{\C}(\hat{V}^{\mu}_0(-Y_0)^I\cap \Gamma(\mathbb{S})^I)-\dim_{\C}(\hat{V}^{\mu}_0(Y_1)^I\cap \Gamma(\mathbb{S})^I)$ 
and the index of the form part of $F'$ is $- b^+(W) + b^+_\iota(W)-\dim(\hat{V}^{\mu}_0(-Y_0)^I\cap \text{Ker}(d^{\ast})^I)-\dim(\hat{V}^{\mu}_0(Y_1)^I\cap \text{Ker}(d^{\ast})^I)$. Recall that the map \eqref{eq: cob map} is given the same way in \cite[Section 3.6]{Ma16} and that $\hat{V}^\mu_\lambda(-Y_0)=\hat{V}_{-\mu}^{-\lambda}(Y_0)$. From the construction of the map \eqref{eq: doubled cob map}, we easily see that \begin{align*}
&(\text{the index of the form part of}\;F')+ \dim(V_0(\tilde{\R})_{-\mu}^{-\lambda})-\dim(V_1(\tilde{\R})_{0}^{\mu})-\dim(V_0(\tilde{\R})^{0}_{-\mu})\\
&=(\text{the index of the form part of}\;F')+ \dim(V_0(\tilde{\R})_{0}^{-\lambda})-\dim(V_1(\tilde{\R})_{0}^{\mu})\\
&=m_0+\dim(V_0(\tilde{\R})_{-\mu}^{-\lambda})-m_1
-\dim(V_1(\tilde{\R})^{\mu}_{\lambda}), \\
&(\text{the index of the}\;\text{Spin}^c \;\text{Dirac part of}\; F')+\dim_\C(V_0(\C_+)_{-\mu}^{-\lambda})+\dim_\C(V_0(\C_-)_{-\mu}^{-\lambda})\\
&-\dim_\C(V_0(\C_+)_{-\mu}^{0}) -\dim_\C(V_1(\C_+)^{\mu}_{0})\\
&-\dim_\C(V_0(\C_-)_{-\mu}^{0}) -\dim_\C(V_1(\C_-)^{\mu}_{0})\\
&=(\text{the index of the}\;\text{Spin}^c \;\text{Dirac part of}\; F')+\dim_\C(V_0(\C_+)_{0}^{-\lambda}) -\dim_\C(V_1(\C_+)^{\mu}_{0})\\
&+\dim_\C(V_0(\C_-)_{0}^{-\lambda}) -\dim_\C(V_1(\C_-)^{\mu}_{0})\\ &=n_0+\dim_\C(V_0(\C_+)_{-\mu}^{-\lambda}) +\dim_\C(V_0(\C_-)_{-\mu}^{-\lambda})-n_1-\dim_\C(V_1(\C_+)^{\mu}_{\lambda}) -\dim_\C(V_1(\C_-)^{\mu}_{\lambda})
\end{align*}
respectively. Hence we have the formula \eqref{eq:m difference} and \eqref{eq:n difference}. 
\end{proof}

\begin{rem}
To calculate the differences $n_0^+ - n_1^+, n_0^- - n_1^-$, we need to use an additional equivariant index theorem.
But for our purpose, it is enough to know only the difference $n_0-n_1$.
\end{rem}

\subsection{Proof of \cref{main theo}} 
\label{subsection Proof of main theo} 

First, we state and prove the main result of this paper in the most general form:

\begin{theo}
\label{most general main theo}
Let $(Y_0, \frakt_0)$, $(Y_1, \frakt_1)$ be spin rational homology 3-spheres.
Let $\iota_0, \iota_1$ be smooth involutions on $Y_0, Y_1$.
Suppose that $\iota_0, \iota_1$ preserve the given orientations and spin structures $\frakt_0, \frakt_1$ on $Y_0, Y_1$ respectively, and suppose that $\iota_0, \iota_1$ are of odd type.
Let $(W,\fraks)$ be a smooth compact oriented spin cobordism with $b_1(W)=0$ from $(Y_0, \frakt_0)$ to $(Y_1, \frakt_1)$.
Suppose that there exists a smooth involution $\iota$ on $W$ such that $\iota$ preserves the given orientation and spin structure $\fraks$ on $W$, and that the restriction of $\iota$ to the boundary is given by $\iota_0, \iota_1$.
Then we have
\begin{align}
\label{eq: main rel 108 general}
-\frac{\sigma(W)}{16} + \kappa (Y_0, \frakt_0, \iota_0 )
\leq b^+(W)-b^+_{\iota}(W) + \kappa (Y_1, \frakt_1, \iota_1 ),
\end{align}
where $b^+_\iota(W)$ denotes the maximal dimension of $\iota$-invariant positive-definite subspaces of $H^2(W;\R)$.
\end{theo}

\begin{proof}
As in \cref{subsection: cobordisms},
take an $\iota$-invariant metric $g$ on $W$ so that $g$ is a cylindrical metric near $\del W$.
The involution $\iota$ lifts to the spin structure $\fraks$ as $\Z_4$-action, and we have the doubled cobordism map \eqref{eq: doubled cob map}. In this proof, we take $\mu=\nu$ and $\lambda=-\nu$. 
Recall that the $H$-invariant part of $I^\nu_{-\nu} (Y_i)$ is of (real) dimension $\dim_\R(V_i(\tilde{\R})^\nu_{-\nu})$, so the $H$-invariant part of the double
$D(I^\nu_{-\nu} (Y_i))$ is of dimension $2\dim_\R(V_i(\tilde{\R})^\nu_{-\nu})$ and there is a canonical homotopy which collapses $(V_i(\tilde{\R})^\nu_0)^2$ part in $D(I^\nu_{-\nu}(Y_i))$ to the base point.
Hence the level of $\Sigma^{m_i\tilde{\C}}\Sigma^{n_i(\C_+ \oplus \C_-)}D(I^\nu_{-\nu} (Y_i))$ is given by $2(m_i + \dim_\R(V_i(\tilde{\R})^0_{-\nu}))$.

First, consider the case that $b^+(W) - b^+_\iota(W)>0$.
Then, by \eqref{eq:m difference},  the level of the domain of the doubled cobordism map \eqref{eq: doubled cob map} is smaller than that of the codomain of \eqref{eq: doubled cob map}.
Therefore we can apply \cref{lem: suspension how much} to \eqref{eq: doubled cob map}, and obtain that
\begin{align*}
    k(\Sigma^{m_0\tilde{\C}}\Sigma^{n_0(\C_+ \oplus \C_-)}D(I^\nu_{-\nu} (Y_0))) + m_0 + \dim_\R(V_0(\tilde{\R})^0_{-\nu})\\
    \leq k(\Sigma^{m_1\tilde{\C}}\Sigma^{n_1(\C_+ \oplus \C_-)}D(I^\nu_{-\nu} (Y_1))) + m_1 + \dim_\R(V_1(\tilde{\R})^0_{-\nu}).
\end{align*}
By \cref{lem: shift by suspension}, this is equivalent to
\begin{align*}
        \begin{split}
        &k(D(I^\nu_{-\nu}(Y_0)))+n_0+m_0+\dim_\R(V_0(\tilde{\R})^0_{-\nu})\\
        \leq& k(D(I^\nu_{-\nu}(Y_1)))+n_1+m_1+\dim_\R(V_1(\tilde{\R})^0_{-\nu}).
        \end{split}
\end{align*}
From this combined with \eqref{eq: kappa describe}, \eqref{eq:m difference}, \eqref{eq:n difference}, we obtain the desired inequality \eqref{eq: main rel 108}.

Next, consider the case that $b^+(W) - b^+_\iota(W)=0$.
Then, by \eqref{eq:m difference},  the level of the domain of the doubled cobordism map \eqref{eq: doubled cob map} is the same as that of the codomain of \eqref{eq: doubled cob map}.
Then we can apply \cref{lem: origin of local eq} to \eqref{eq: doubled cob map}, instead of \cref{lem: suspension how much}, and obtain the desired inequality \eqref{eq: main rel 108} by the same argument for the case that $b^+(W) - b^+_\iota(W)>0$ above.
\end{proof}

\begin{proof}[Proof of \cref{main theo}]
The property (i) follows from the equality
\[
n(Y,\frakt,g) = \mu(Y,\frakt) \mod 2,
\]
which follows from the expression of $n(Y,\frakt,g)$ given as \eqref{eq: expression of n for Rohlin}, and the expression of $\kappa(Y,\frakt,\iota)$ given as \eqref{eq: kappa describe}.

Next, we prove the property (ii).
Denote by $I_\iota$ the involution on the configuration space discussed in \cref{subsection: An involution on the configuration space} defined by $\iota$ and an $\iota$-invariant metric $g$.
Denote by the configuration space $\cal{C}(Y,g)=\Ker(d^{*_g} \times \Gamma(\mathbb{S}))$ defined by the metric $g$.
Similarly, denote by $I_{f^{-1} \circ \iota \circ f}$ the involution defined by $f^{-1} \circ \iota \circ f$ and the metric $f^{\ast}g$.
Then the pull-back under $f$ induces a $G$-homeomorphism
\[
\cal{C}(Y,g) \to \cal{C}(Y,f^{\ast}g),
\]
which is equivariant under the action of $I_\iota$ and $I_{f^{-1} \circ \iota \circ f}$.
This induces an $G$-equivariant homomorphism
\begin{align*}
I^\mu_\lambda (Y, \frakt,\iota, g)
\to I^\mu_\lambda (Y, \frakt, f^{-1} \circ \iota \circ f, f^{\ast}g),    
\end{align*}
and an equivalence between  representatives of $DSWF_G(Y, \frakt,\iota)$ and of $DSWF_G(Y, \frakt,f^{-1} \circ \iota \circ f)$.
Thus we have that 
\[
DSWF_G(Y, \frakt,\iota)
= DSWF_G(Y, \frakt,f^{-1} \circ \iota \circ f)
\]
in $\mathfrak{C}_{\iota}$.
This and \cref{lem: inv k kappa} imply the property (ii) of \cref{main theo}.

The property (iii) is already proved in \cref{prop: from duality}.

Lastly, we give the proof of the property (iv).
In general, for a spin manifold of $\dim \leq 4$ and an involution $\iota$ on this manifold preserving the spin structure with non-empty fixed-point set, the condition that 
$\iota$ is of odd type is equivalent to that the fixed-point set of $\iota$ is of codimension-2 \cite[Proposition~8.46]{AB68}.
Thus \cref{main theo} immediately follows from \cref{most general main theo}.
\end{proof}

\begin{rem}
In this paper, we only consider $\Z_2$-branched covering spaces. 
On the other hand, $\Z_p$-branched covering spaces have been used in gauge theory (see for examples \cite{Ja12, BH21, AH21}). 
However, it seems not to be straightforward to extend the $\Z/p$ action to the settings in this paper. 
\end{rem}

Note that for a statement similar to \cref{most general main theo} for a spin 4-manifold $W$ with one boundary component follows from \cref{most general main theo} by removing a ball in $W$ near a fixed point, as far as the involution has non-empty fixed-point set.
But, for free involutions, it seems that the case of one boundary component is not deduced from the two component case \cref{most general main theo}.
However, we can carry out the proof of \cref{most general main theo} also for the one boundary component case without any essential change.
We record this as a statement as follows.
We wish to thank David Baraglia for pointing this out.
See also Nakamura's work \cite{Na13} in his $Pin^-(2)$-monopole setting for free involutions.

\begin{theo}
\label{most general main theo one boundary}
Let $(Y, \frakt)$ be a spin rational homology 3-sphere.
Let $\iota$ be a smooth involution on $Y$.
Suppose that $\iota_Y$ preserves the given  orientation and spin structure $\frakt$ respectively, and suppose that $\iota_Y$ is of odd type.
Let $(W,\fraks)$ be a smooth compact oriented spin 4-manifold $b_1(W)=0$ bounded by $(Y, \frakt)$.
Suppose that there exists a smooth involution $\iota$ on $W$ such that $\iota$ preserves the given orientation and spin structure $\fraks$ on $W$, and that the restriction of $\iota$ to the boundary is given by $\iota_Y$.
Then we have
\begin{align}
\label{eq: main rel 108 general one boundary}
-\frac{\sigma(W)}{16} \leq b^+(W)-b^+_{\iota}(W) + \kappa (Y, \frakt, \iota_Y),
\end{align}
where $b^+_\iota(W)$ denotes the maximal dimension of $\iota$-invariant positive-definite subspaces of $H^2(W;\R)$.
\end{theo}

\subsection{Stronger 10/8-inequalities for SWF-spherical triples}
\label{subsection SWF-spherical}

For some class of spin rational homology sphere with odd involution, we can refine the 10/8-type inequality stated in \cref{main theo}.

To state this without loss coming from the doubling construction, we need to recall our construction of the Seiberg--Witten Floer stable homotopy type.
Let $(Y,\frakt)$ be a spin rational homology sphere with odd involution $\iota$.
Fix a $\iota$-invariant metric $g$ on $Y$ and a lift $\tilde{\iota}$ of $\iota$ to the spin structure.
Then we consider a finite-dimensional approximation of the $I$-invariant part of the configuration space on which we have the formal gradient flow for the Chern--Simons--Dirac functional, and obtained the Conley index $I^\mu_\lambda (Y, \iota, \tilde{\iota}, g)$.
Denote by $\C$ the standard cyclic 1-dimensional complex representation of $G=\Z_4$.

\begin{defi}\label{SWF-spherical}
We say that a triple $(Y, \mathfrak{t}, \iota)$ of a rational homology 3-sphere, spin structure, and an odd involution is {\it SWF-spherical} if
there exist an $\iota$-invariant metric $g$ on $Y$ and a lift $\tilde{\iota}$ of $\iota$ to the spin structure such that the Conley index $I^\mu_\lambda (Y, \iota, \tilde{\iota}, g)$ is $G$-homotopy equivalent to $S(\C^n)$, the unit sphere of $\C^n$.
\end{defi}

The definition of SWF-spherical triples may seem to be modeled on Seiberg--Witten Fleor stable homotopy type of 3-manifolds with positive scalar curvature metric.
However, in our setup with involution, many of Seifert homology spheres are examples of this kind:
In \cref{ex: S3 case,ex: lens} and \cref{theo:Seifert cal most general}, we will see that some natural involutions on $S^3$, lens spaces, and Seifert homology spheres give examples of SWF-spherical triples.

\begin{lem}
\label{lem: strong SWF spherical kappa}
If $(Y, \mathfrak{t}, \iota)$ is SWF-spherical so that $I^\mu_\lambda (Y, \iota, \tilde{\iota}, g) \cong S(\C^n)$, then we have
\[
\kappa (Y, \mathfrak{t}, \iota) = -n(Y, \frakt,g)/2.
\]
\end{lem}

\begin{proof}
We have
\[
\dim_{\C}V(\C_+)^0_\lambda
+\dim_{\C}V(\C_-)^0_\lambda
=n,
\]
and thus obtain
\begin{align*}
DSWF_G(Y, \frakt,\iota)
=& 
[((\C_+ \oplus \C_-)^n,0,\dim_{\C}V(\C_+)^0_\lambda
+\dim_{\C}V(\C_-)^0_\lambda
+n(Y, \frakt,g)/2)]\\
=&
[(S^0,0,n(Y, \frakt,g)/2)].
\end{align*}
The assertion of \lcnamecref{lem: strong SWF spherical kappa} follows from this.
\end{proof}

For $N \in \Z$, define 
$A(N) \in \{1,2,3\}$ by
\begin{align}\label{A-definition}
A(N) = 
\begin{cases}
&1, \quad N=0,2 \mod 8\\
&2, \quad N=1,3,4,5,7 \mod 8\\
&3, \quad N=6 \mod 8.
\end{cases}    
\end{align}
Denote by $\tilde{\R}$ the non-trivial real 1-dimensional representation of $G$.
The main input in this \lcnamecref{subsection SWF-spherical} is the following Borsuk--Ulam-type theorem:

\begin{theo}[Crabb~\cite{Cra89}, Stolz~\cite{Sto89}]
\label{CrabbStolz}
Suppose that $N \geq 2, N' \geq 1$.
Then there exists a pointed continuous $G$-map
\[
f : (\C^N)^+ \to (\tilde{\R}^{N'})^+,
\]
if and only if
\[
N+A(N) \leq N'.
\]
\end{theo}

\begin{rem}
\label{rem: weak Borsuk-Ulam}
Note that, from the argument of the proof of \cref{lem: suspension how much}, we already have the following Borsuk--Ulam-type theorem:
if there exists a pointed continuous $G$-map
\[
f : (\C^{n_0} \oplus \tilde{\R}^{m_0})^+  \to (\C^{n_1} \oplus \tilde{\R}^{m_1})^+
\]
that induces a homotopy equivalence on $G$-fixed-point sets,
then we have
\begin{align}
\label{eq:weak BU}
n_0 - n_1 \leq m_1-m_0.
\end{align}
\cref{CrabbStolz} means that, if $f$ is ``desuspended'' so that $f$ is a map from the  complex representation space $\C^N$ to the real representation space $\tilde{\R}^{N'}$,
then we may get a stronger constraint.
\end{rem}

Let us also recall a classical fact from equivariant stable homotopy theory.
This is a version of equivariant Freudenthal suspension theorem:

\begin{theo}[See, for example, {\cite[Theorem~1.5]{GM95}}]
\label{theo: desuspension}
Let $G'$ be a compact Lie group, $X_0, X_1$ be pointed $G'$-spaces, and $V$ be a representation of $G'$.
Denote by $\mathrm{Sub}(G')$ the set of subgroups of $G'$ and by $\mathrm{conn}(X)$ the connectivity of a given space $X$.
Suppose that there exists a conjugation-invariant function $\mathcal{N} : \mathrm{Sub}(G') \to \{-1, 0, 1, 2,  \ldots, +\infty\}$ that satisfies the following properties:
\begin{itemize}
\item[(i)] For all $H \in \mathrm{Sub}(G')$ with $V^H\neq \{0\}$, we have $\mathcal{N}(H) \leq 2\mathrm{conn}(X_1^H)+1$.
\item[(ii)] For all $H \in \mathrm{Sub}(G')$ and subgroup $K \subset H$ with $V^K \neq V^H$, we have $\mathcal{N}(H) \leq \mathrm{conn}(X_1^K)$.
\item [(iii)] For all $H \in  \mathrm{Sub}(G')$, we have $\dim X^H \leq \mathcal{N}(H)$.
\end{itemize}
Then the suspension map
\[
\Sigma^{V} : [X_0, X_1]_{G'} \to [\Sigma^{V}X_0, \Sigma^{V}X_1]_{G'}
\]
is surjective.
\end{theo}

\begin{rem}
If we replace the condition (iii) of \cref{theo: desuspension} with the condition that $\dim X^H \leq \mathcal{N}(H)+ 1$ for all $H \in  \mathrm{Sub}(G')$, we have that $\Sigma^{V} : [X_0, X_1]_{G'} \to [\Sigma^{V}X_0, \Sigma^{V}X_1]_{G'}$ is isomorphic.
But we do not have to use this version of suspension theorem.
\end{rem}

The next \lcnamecref{lem: from weak BU to strong BU} refines the weak Borsuk--Ulam theorem~\eqref{eq:weak BU}:

\begin{prop}
\label{lem: from weak BU to strong BU}
Suppose that there exists a pointed continuous $G$-map
\[
f : (\C^{n_0} \oplus \tilde{\R}^{m_0})^+  \to (\C^{n_1} \oplus \tilde{\R}^{m_1})^+
\]
that induces a homotopy equivalence on the $G$-fixed-point sets for natural numbers $n_0, n_1, m_0, m_1$.
Suppose that $n_0-n_1 \geq 2$ and $m_1-m_0 \geq 1$.
Then we have
\begin{align}
\label{eq: refined BU}
n_0-n_1+A(n_0-n_1) \leq m_1-m_0.    
\end{align}
\end{prop}

\begin{proof}
For $G=\Z_4$,
define a function $\mathcal{N} : \mathrm{Sub}(G) \to \mathbb{N}$ by
\[
\mathcal{N}(H)
=
\begin{cases}
m_1-m_0-2, & H=\Z_2  \text{ or } G\\
2(m_1-m_0)-2, &H=\{1\}.
\end{cases}
\]
We shall apply \cref{theo: desuspension}
to this $\mathcal{N}$ and
\[
G'=G,\quad X_0 = (\C^{n_0-n_1})^+,\quad X_1 = (\tilde{\R}^{m_1-m_0})^+, \quad V=\C^{n_1}\oplus \tilde{\R}^{m_0}.
\]
Let us check that these satisfy the assumptions of \cref{theo: desuspension}.
For $H=\Z_2$ or $G$, it is straightforward to check
that the assumptions (i), (ii), (iii) in
\cref{theo: desuspension} are satisfied.
For $H=\{1\}$, again it is easy to check that the assumptions (i), (ii) are satisfied, and the assumption (iii) is equivalent to that
\[
n_0 - n_1 \leq m_1 -m_0 -1.
\]
But this inequality follows from the weaker Borsuk--Ulam-type theorem mentioned in \cref{rem: weak Borsuk-Ulam}.
Thus we can apply \cref{theo: desuspension} and deduce that there exists a pointed continuous $G$-map
\[
f' : (\C^{n_0-n_1})^+  \to (\tilde{\R}^{m_1-m_0})^+.
\]
Applying \cref{CrabbStolz} to this map $f'$, we obtain the inequality \eqref{eq: refined BU}. 
\end{proof}

Let us apply the refined Borsuk--Ulam-type theorem~\cref{lem: from weak BU to strong BU} to our setup:

\begin{theo}
\label{most general main theo strong SWF spherical kappa}
Let $(Y_0, \frakt_0)$, $(Y_1, \frakt_1)$ be spin rational homology 3-spheres.
Let $\iota_0, \iota_1$ be smooth involutions on $Y_0, Y_1$.
Suppose that $\iota_0, \iota_1$ preserve the given orientations and spin structures $\frakt_0, \frakt_1$ on $Y_0, Y_1$ respectively, and suppose that $\iota_0, \iota_1$ are of odd type.
Suppose that $(Y_0, \frakt_0, \iota_0)$, $(Y_1, \frakt_1,\iota_1)$ are SWF-spherical. 
Let $(W,\fraks)$ be a smooth compact oriented spin cobordism with $b_1(W)=0$ from $(Y_0, \frakt_0)$ to $(Y_1, \frakt_1)$.
Suppose that there exists a smooth involution $\iota$ on $W$ such that $\iota$ preserves the given orientation and spin structure $\fraks$ on $W$, and that the restriction of $\iota$ to the boundary is given by $\iota_0, \iota_1$.
Set
\[
N(W, Y_0, \frakt_0, \iota_0, Y_1, \frakt_1, \iota_1):=
-\frac{\sigma(W)}{16} + \kappa (Y_0, \frakt_0, \iota_0 ) - \kappa (Y_1, \frakt_1, \iota_1).
\]
Then $N(W, Y_0, \frakt_0, \iota_0, Y_1, \frakt_1, \iota_1)$ is an integer, and
if we have
\[
N(W, Y_0, \frakt_0, \iota_0, Y_1, \frakt_1, \iota_1) \geq 2 \quad \text{and} \quad 
b^+(W) - b^+_\iota(W) \geq 1,
\] 
then the inequality
\begin{align}
\label{eq: main rel 108 general strong}
\begin{split}
&-\frac{\sigma(W)}{16} + \kappa (Y_0, \frakt_0, \iota_0 ) + A(N(W, Y_0, \frakt_0, \iota_0, Y_1, \frakt_1, \iota_1))\\
\leq& b^+(W)-b^+_{\iota}(W) + \kappa (Y_1, \frakt_1, \iota_1 )
\end{split}
\end{align}
holds.
\end{theo}

\begin{proof}
Under the SWF-spherical assumption,
the $I$-invariant-relative Bauer--Furuta invariant \eqref{eq: cob map} is given by 
\[
f : (\C^{n_0} \oplus \tilde{\R}^{m_0})^+  \to (\C^{n_1} \oplus \tilde{\R}^{m_1})^+
\]
with
\begin{align*}
&m_1 - m_0 = b^+(W) - b^+_\iota(W),\\    
&n_0 - n_1 = - \frac{\sigma(W)}{16}+ \frac{n(Y_1, \mathfrak{t}_1, g_1)}{2}-\frac{n(Y_0, \mathfrak{t}_0, g_0)}{2}
= N(W, Y_0, \frakt_0, \iota_0, Y_1, \frakt_1, \iota_1).
\end{align*}
By \cref{rem: integrality} and this expression of $n_0-n_1$, $N(W, Y_0, \frakt_0, \iota_0, Y_1, \frakt_1, \iota_1)$ is an integer.
The assertion of \cref{most general main theo strong SWF spherical kappa} follows from \cref{lem: from weak BU to strong BU} applied to this $f$.
\end{proof}

\begin{rem}
\label{rem: SWF spherical disjoint}
Without any essential change, \cref{most general main theo strong SWF spherical kappa} holds for $(Y_i, \frakt_i, \iota_i)$ that are disjoint union of spin rational homology spheres with odd involution that are SWF-spherical.
\end{rem}

\subsection{Connected sum formula}

First, define equivariant connected sums for spin rational homology spheres equipped with odd involutions that have nono-empty-fixed-point set as follows:

\begin{defi}\label{Connected sum equi}
Let $(Y_0, \mathfrak{t}_0, y_0)$ and $ (Y_1,\mathfrak{t}_1, y_1)$ be spin rational homology $3$-spheres with base points. Let $\iota_0$ and $\iota_1$ are involutions on $Y_0$ and $Y_1$ respectively. Suppose that 
\begin{itemize}
\item fixed points of $\iota_i$ are codimension-2 and $\iota_i(y_i)=y_i$, 
\item $\iota_i$ preserves the spin structure $\mathfrak{t}_i$ for $i=0, 1$. 
\end{itemize}
We give an orientation $o(\iota_i)$ of the set of fixed points of the involution $\iota_i$. We define the connected sum $(Y_0, \mathfrak{t}_0, \iota_0, o(\iota_0)) \# (Y_1, \mathfrak{t}_1, \iota_1, o(\iota_1))$ as follows. 
Let $f_i \colon (D^3, 0) \to (Y_i, y_i)$ is the embedding of the $3$-disk to $Y_i$ which satisfies following conditions.  
\begin{itemize}
    \item $\iota_i(f_i(D^3))=f_i(D^3)$.  
    \item 
    $f_i^{-1}\circ \iota_i \circ f_i(x, y, z)=(x, -y, -z)$. 
    \item The orientation of $Y_i^{\iota_i}$ coincides with the orientation of $f_i(D^3) \cap Y_i^{\iota_i}$ induced by the natural orientation of $\{ (x, 0,0) \in D^3 \}$.
\end{itemize}
Take a trivialization of $TY|_{f_i(D^3)}$ using the coordinate of $D^3$. This trivialization gives a trivialization of the spin structure $\mathfrak{t}_i$ on $f_i(D^3)$. We set $D^3_{1/2}=\{(x, y, z) \in D^3 \mid x^2+y^2+z^2 \le 1/2 \}$. 
We set $\phi \colon \partial D^3_{1/2} \to \partial D^3_{1/2}$ by $(x, y, z) \mapsto (x, -y, z)$ and  $Y^{\#}=(Y_0 \setminus f_0(\Int(D^3_{1/2}))) \cup_{\phi} (Y_1 \setminus f_1(\Int(D^3_{1/2})))$. If that we can define the involution $\iota^{\#}$ on $Y^{\#}$ such that $\iota^{\#} |_{Y_i \setminus f_i(D^3_{1/2})}=\iota_i$ and the spin structure $\mathfrak{t}^{\#}$ on $Y^{\#}$ such that $\mathfrak{t}^{\#} |_{Y_i \setminus f_i(D^3_{1/2})}=\mathfrak{t}_i$. 
We define $(Y_0, \mathfrak{t}_0, y_0,\iota_0, o(\iota_0)) \# (Y_1, \mathfrak{t}_1, y_1, \iota_1, o(\iota_1))\colon=(Y^{\#}, \mathfrak{t}^{\#}, \iota^{\#})$. 
\end{defi}

\begin{rem}
\label{rem: connected sum dependence on base point}
If the fixed-point set of $\iota_i$ is connected for $i=0, 1$, we have that the $\Z/2$-equivariant isotopy class of the connected sum $(Y_0, \mathfrak{t}_0, y_0,\iota_0) \# (Y_1, \mathfrak{t}_1, y_1, \iota_1)$ is independent of the choice of base points $y_0$ and $y_1$, and we sometimes drop $y_0, y_1$ from our notation.
Similarly, if the orientations $o(\iota_i)$ of the fixed-point sets are given, we drop $o(\iota_i)$ from our notation.
\end{rem}

\begin{lem}
Let $K$ and $K'$ be oriented knots in $S^3$.
There is a $\Z_2$-equivariant orientation preserving diffeomorphism between $(\Sigma (K\# K'), \iota_{K\# K'} )$ and $(\Sigma (K), \iota_{K} )\# (\Sigma (K'), \iota_{K'} )$.  
\end{lem}
\begin{proof}
It is clear by the construction of the branched double covering and \cref{Connected sum equi}. 
\end{proof}

\begin{defi}\label{connected sum cob}
Let $(Y_0, \mathfrak{t}_0, y_0,\iota_0)$, $(Y_1, \mathfrak{t}_1, y_1, \iota_1)$ and $(Y_0, \mathfrak{t}_0, y_0,\iota_0) \# (Y_1, \mathfrak{t}_1, y_1, \iota_1)$ are spin rational $3$-spheres with involutions as \cref{Connected sum equi}. We give an orientation of $Y_i^{\iota_i}$. 
We define a spin structure and a involution on $I \times Y_i$ as follows: 
\begin{itemize}
    \item We define the involution $\tilde{\iota_i}$ on $I \times Y_i$ by $(t, y) \mapsto (t, \iota_i(y))$. 
    \item We denote $\mathfrak{s}_i$ by the spin structure on $I \times Y_i$ given by the pull back of $\mathfrak{t}_i$. We see that the involution $\tilde{\iota_i}$ preserves $\mathfrak{s}_i$. 
\end{itemize}
We set $D^4_+=\{ (t, x, y, z) \in D^4 \mid t \ge 0\}$ and $D^3_0 := \{ (0, x, y, z) \in D^4_+ \}$. Let $\tilde{f}_i \colon (D^4_+, D^3_0, 0) \to (I \times Y_i, \{0\} \times Y_i, (0, y_i))$ be an embedding of $D^4_+$ to $I \times Y_i$ which satisfies the following conditions.
\begin{itemize}
    \item $\tilde{f}_i(D^4_+) \cap (\{1\} \times Y_i)=\emptyset$. 
    \item $\tilde{f}_i^{-1}\circ \tilde{\iota_i} \circ \tilde{f}_i(t, x, y, z)=(t, x, -y, -z)$. 
    \item The orientation of $I \times Y_i^{\iota_i}$ coincides with the orientation of $\tilde{f}_i(D^4_+) \cap Y_i^{\iota_i}$ induced by the orientation of $\{(t, x, 0, 0) \in D^4_+ \}$ given by $dt \wedge dx$. 
\end{itemize}
It is easy to see that $\tilde{f}_i |_{D^3_0} \colon D^3_0 \to \{0\} \times Y_i$ is $\Z_2$-equivariant isotopic to $f_i$. 
We define a cobordism $W_{01} := W((Y_0, \mathfrak{t}_0, y_0,\iota_0), (Y_1, \mathfrak{t}_1, y_1,\iota_1))$ from $(Y_0, \mathfrak{t}_0, y_0,\iota_0)\#  (Y_1, \mathfrak{t}_1, y_1,\iota_1)$ to $(Y_0, \mathfrak{t}_0, y_0,\iota_0)\cup  (Y_1, \mathfrak{t}_1, y_1,\iota_1)$ as follows:
\begin{itemize}
    \item Let $D^4_{+, 1/2}=\{ (t, x, y, z) \in D^4_+ \mid t^2+x^2+y^2+z^2 \le 1/2 \}$, $D^3_{0, 1/2}\colon =\partial D^4_{+, 1/2} \cap D^3_0$ and $D^3_{+, 1/2}\colon =\{ (t, x, y, z) \in D^4_+\mid t^2+x^2+y^2+z^2=1/2 \}$. 
    \item Let $\tilde{\phi} \colon D^3_{+,1/2} \to D^3_{+,1/2}$ be a map given by $(t, x, y, z) \mapsto (t, x, -y, z)$. Note that $\tilde{\phi}|_{D^3_{+,1/2} \cap D^3_0} = \phi$. 
    \item We denote $W_{01}$ by $(I \times Y_0 \setminus \tilde{f}_0(\Int(D^4_{+, 1/2}) \cup D^3_{0, 1/2} ))\cup_{\tilde{\phi}}(I \times Y_1 \setminus \tilde{f}_1(\Int(D^4_{+, 1/2}) \cup D^3_{0, 1/2}))$.
\end{itemize}
We define the spin structure $\mathfrak{s}$ and involution $\iota$ on $W_{01}$ in a similar way in \cref{Connected sum equi}. 
\end{defi}

For given two knots $K$ and $K'$ in $S^3$.
We consider the following knot cobordism $(W, S)$ from $K \# K'$  to $K \cup K'$: 
\begin{itemize}
    \item The cobordism $W$ is a 3-handle cobordism attached to $I \times S^3$. 
    \item The cobordism $S$ inside $W$ is a 1-handle cobordism attached to $K\# K' $ corresponding to the connected sum decomposition of $K \cup K'$.
\end{itemize}
    
\begin{lem}
Let $K$ and $K'$ be oriented knots in $S^3$. 
Then there is an orientation preserving $\Z_2$-equivariant diffeomorphism from $\Sigma (S)$ to $W_{01} = W((\Sigma (K) ,\iota_K), (\Sigma (K') ,\iota_{K'}))$. 
\end{lem} 
\begin{proof}
It is clear from the construction of the branched double covering and \cref{connected sum cob}. 
\end{proof}

We prove a certain connected sum formula of $[DSWF(Y, \mathfrak{t}, \iota )]$. For the connected sum formula of local equvalence classes of usual Seiberg--Witten Floer homotopy types, see \cite{Sto20}. 
\begin{theo}\label{conn sum of local eq}
Let $(Y_0, \mathfrak{t}_0, y_0, \iota_0)$ and $ (Y_1,\mathfrak{t}_1, y_1, \iota_1)$ are spin rational homology $3$-spheres with involutions $\iota_0$ and $\iota_1$ and with base points $y_0$ and $y_1$. 
Then, we have 
\[
[DSWF_G (Y_0, \mathfrak{t}_0, \iota_0) \wedge DSWF_G (Y_1, \mathfrak{t}_1, \iota_1) ]_{\rm{loc}} = [DSWF_G(Y^\#, \mathfrak{t}^\#, \iota^\#)]_{\rm{loc}}, 
\]
where $(Y^\#, \mathfrak{t}^\#, \iota^\#) := (Y_0, \mathfrak{t}_0, y_0,\iota_0) \# (Y_1, \mathfrak{t}_1, y_1, \iota_1)$. 
\end{theo}

\begin{proof}
It is sufficient to construct local maps 
\[
f : DSWF (Y_0, \mathfrak{t}_0, y_0, \iota_0) \wedge DSWF (Y_1, \mathfrak{t}_1, y_1, \iota_1) \to DSWF(Y^\#, \mathfrak{t}^\#, \iota^\#) \]
and 
\[
f' : DSWF(Y^\#, \mathfrak{t}^\#, \iota^\#) \to DSWF (Y_0, \mathfrak{t}_0, y_0, \iota_0) \wedge DSWF (Y_1, \mathfrak{t}_1, y_1, \iota_1). 
\]
We consider the doubled relative Bauer--Furuta invariant for the cobordism $W_{01}$ constructed in \eqref{eq: doubled cob map}: 
\begin{align*}
f  : 
\Sigma^{m_0\tilde{\C}}\Sigma^{n_0(\C_+ \oplus \C_-)}D(I_{-\mu}^{-\lambda} (Y_0)) \wedge D(I_{-\mu}^{-\lambda} (Y_1)) \to \Sigma^{m_1\tilde{\C}}\Sigma^{n_1(\C_+ \oplus \C_-)}D(I_{-\mu}^{-\lambda} (Y^\#)), 
\end{align*}
where 
\begin{align*}
m_0-m_1=&\dim_\R (V^\#(\tilde{\R})^0_{\lambda})-\dim_\R(V_0(\tilde{\R})^0_{-\mu}) - \dim_\R(V_1(\tilde{\R})^0_{-\mu}) - b^+(W_{01}) + b^+_\iota(W_{01}), \\
\begin{split}
n_0-n_1=&\dim_\C(V_1(\C_+)^0_{\lambda}) +\dim_\C(V^\#(\C_-)^0_{\lambda})\\
&-\dim_\C(V_0(\C_+)^0_{-\mu})-\dim_\C(V_0(\C_-)^0_{-\mu}) -\dim_\C(V_1(\C_+)^0_{-\mu})-\dim_\C(V_1(\C_-)^0_{-\mu})\\
&- \frac{\sigma(W)}{16}+ \frac{n(Y^\#, \mathfrak{t}^\#, g^\#)}{2}-\frac{n(Y_0, \mathfrak{t}_0, g_0)}{2}- \frac{n(Y_1, \mathfrak{t}_1, g_1)}{2}, 
\end{split}
\end{align*} (See \eqref{eq:n difference} and \eqref{eq:m difference}), and
$V^\#(U)^0{_{\lambda}}$, $V_0(U)^0_{-\mu}$ and $V_1(U)^0_{-\mu}$ are finite dimensional approximations of sliced configuration spaces for $Y^\#$, $Y_0$ and $Y_1$ respectively, $g'$, $g_0$ and $g_1$ are invariant Riemann metrics on $Y^\#$, $Y_0$ and $Y_1$ respectively. 
Now, we can easily see $b^+(W_{01})=0$ and $b^+_\iota(W_{01})=0$. So, $f$ is a local map. The doubled relative Bauer--Furuta invariant for $-W_{01}$ gives a local map $f'$. This completes the proof. 
\end{proof}

Let us introduce a slight variant of the notion of SWF-spherical:

\begin{defi}
\label{defi: DSWF-spherical}
We say that a spin rational homology 3-sphere $(Y, \mathfrak{t}, \iota)$ with odd involution preserving the spin structure is {\it locally DSWF-spherical} if we have
\[
[DSWF_G(Y, \mathfrak{t}, \iota) ]_{\text{loc}} = [(S^0,0, n)]_{\text{loc}} 
\]
for some $n$. 
Note that, in this case, we have 
\[
\kappa (Y, \mathfrak{t}, \iota) = -n. 
\]
\end{defi}

The difference of locally DSWF-spherical from SWF-spherical is that we consider the double and local equivalence class in the definition of locally DSWF-spherical.
It is clear that, if $(Y, \mathfrak{t}, \iota)$ is SWF-spherical, then this triple is locally DSWF-spherical.
As a corollary of \cref{conn sum of local eq}, we prove: 

\begin{cor}
\label{locally DSWF-spherical conn sum}
Let $(Y, \mathfrak{t}, \iota)$ and $(Y', \mathfrak{t}', \iota')$ be 
spin rational homology 3-spheres with odd involutions preserving the spin structures respectively. 
If $(Y, \mathfrak{t}, \iota)$ and $(Y', \mathfrak{t}', \iota')$ are locally DSWF-spherical, then $(Y, \mathfrak{t}, \iota)\#(Y', \mathfrak{t}', \iota')$ and $(-Y, \mathfrak{t}, \iota)$ are also locally DSWF-spherical. 
\end{cor}

\begin{proof}
The first claim follows from \cref{conn sum of local eq}.
We prove the latter statement. By considering a similar construction in \cref{connected sum cob}, we have a cobordism $W_{01}$ from $(Y, \mathfrak{t}, \iota)\#(-Y, \mathfrak{t}, \iota)$ to $(S^3,\mathfrak{t}_0 ,\iota_0)$, where $\mathfrak{t}_0$ is the unique spin structure and $\iota_0$ is the complex conjugation. We can easily check that the doubled Bauer-Furuta invariants of $W_{01}$ and $-W_{01}$ give a local equivalence 
\[
[DSWF ( (Y, \mathfrak{t}, \iota)\#(-Y, \mathfrak{t}, \iota) )]_{\text{loc}}  = [(S^0, 0, 0)]_{\text{loc}}. 
\]
Then, by \cref{conn sum of local eq}, 
\[
[ ( I(-Y), m',n' + n )]_{\text{loc}}  = [(S^0, 0, 0)]_{\text{loc}}, 
\]
where $DSWF(-Y, \mathfrak{t}, \iota)= ( I(-Y), m',n' )$ and $[DSWF ( (Y, \mathfrak{t}, \iota)]_{\text{loc}} =[(S^0, 0, n)]_{\text{loc}}$. 
So, we have 
\[
[ DSWF(-Y, \mathfrak{t}, \iota)]_{\text{loc}} = [ ( I(-Y), m',n'  )]_{\text{loc}}  = [(S^0, 0, -n)]_{\text{loc}}.
\]
from \cref{local map desuspension}. 
\end{proof}

\begin{cor}
\label{cor: connected sum formula kappa}
Let $(Y, \mathfrak{t}, \iota)$ be a spin rational homology sphere with odd involution.
Suppose that $(Y, \mathfrak{t}, \iota)$ is locally DSWF-spherical.
Then, for any spin rational homology 3-sphere $(Y', \mathfrak{t}', \iota')$ with odd involution, we have 
\[
\kappa ( Y\# Y', \mathfrak{t} \# \mathfrak{t}' , \iota\# \iota') = \kappa (Y, \mathfrak{t}, \iota) + \kappa (Y', \mathfrak{t}', \iota').  
\]
Here the connected sum of $(Y, \mathfrak{t}, \iota)$ and $(Y', \mathfrak{t}', \iota')$ can be taken along any fixed points. 
\end{cor}
\begin{proof}
By \cref{conn sum of local eq}, we see 
\[
[DSWF (Y, \mathfrak{t}, \iota)  \wedge DSWF(Y', \mathfrak{t}', \iota') ]_{\text{loc}} = [DSWF( Y\# Y', \mathfrak{t} \# \mathfrak{t}' , \iota\# \iota')]_{\text{loc}}.  
\]
Thus, we have 
\[
\kappa (DSWF (Y, \mathfrak{t}, \iota) \wedge DSWF (Y', \mathfrak{t}', \iota')) = \kappa (DSWF( Y\# Y', \mathfrak{t} \# \mathfrak{t}' , \iota\# \iota')). 
\]
Since we are assuming that 
\[
[DSWF(Y, \mathfrak{t}, \iota) ]_{\text{loc}} = [(S^0,0, m )]_{\text{loc}}
\]
for some $ m$, we see 
\[
[DSWF (Y, \mathfrak{t}, \iota) \wedge DSWF (Y', \mathfrak{t}', \iota')]_{\text{loc}} = [(I(Y', \mathfrak{t}', \iota'),0,m'+ m )]_{\text{loc}}, 
\]
where $I(Y', \mathfrak{t}', \iota')$, $m'$ represents $DSWF (Y, \mathfrak{t}, \iota)$, i.e. 
\[
[(I(Y', \mathfrak{t}', \iota'), 0 , m')] = DSWF (Y, \mathfrak{t}, \iota). 
\]
Also, we have 
\[
\kappa (I(Y', \mathfrak{t}', \iota'),0 ,m'+ m )= \kappa (Y', \mathfrak{t}', \iota') + m 
\]
by the definition of $\kappa$. This completes the proof. 
\end{proof}

We also note the stronger 10/8-inequality, \cref{most general main theo strong SWF spherical kappa}, holds also for equivariant connected sums of SWF-spherical triples:

\begin{theo}
\label{theo: SWF spherical connected sum}
For $N, N' \geq 0$ and $i \in \{1, \ldots, N\}, j \in \{1, \ldots, N'\}$,
let $(Y_i, \mathfrak{t}_i, \iota_i), (Y_j', \mathfrak{t}_j', \iota_j')$ be spin rational homology spheres with odd involution.
Suppose that $(Y_i, \mathfrak{t}_i, \iota_i), (Y_j', \mathfrak{t}_j', \iota_j')$ are SWF-spherical and $\iota_i, \iota_j'$ have non-empty fixed-point set.
Form the equivariant connected sums
\[
(Y, \frakt, \iota)
= \#_{i=1}^N  (Y_i, \mathfrak{t}_i, \iota_i),\quad
(Y', \frakt', \iota')
= \#_{j=1}^{N'}  (Y_j', \mathfrak{t}_j', \iota_j')
\]
along fixed-points.
Suppose that $\iota, \iota'$ preserve the given orientations and spin structures $\frakt, \frakt'$ on $Y, Y'$ respectively, and suppose that $\iota, \iota'$ are of odd type.
Suppose that $(Y, \frakt, \iota)$, $(Y', \frakt',\iota')$ are SWF-spherical. 
Let $(W,\fraks)$ be a smooth compact oriented spin cobordism with $b_1(W)=0$ from $(Y, \frakt)$ to $(Y', \frakt')$.
Suppose that there exists a smooth involution $\iota$ on $W$ such that $\iota$ preserves the given orientation and spin structure $\fraks$ on $W$, and that the restriction of $\iota$ to the boundary is given by $\iota, \iota'$.
Set
\[
N(W, Y, \frakt, \iota, Y', \frakt', \iota'):=
-\frac{\sigma(W)}{16} + \kappa (Y, \frakt, \iota) - \kappa (Y', \frakt', \iota').
\]
Then $N(W, Y, \frakt, \iota, Y', \frakt', \iota')$ is an integer, and
if we have
\[
N(W, Y, \frakt, \iota, Y', \frakt', \iota') \geq 2 \quad \text{and} \quad 
b^+(W) - b^+_\iota(W) \geq 1,
\] 
then the inequality
\begin{align*}
\begin{split}
&-\frac{\sigma(W)}{16} + \kappa (Y, \frakt, \iota ) + A(N(W, Y, \frakt, \iota, Y', \frakt', \iota'))\\
\leq& b^+(W)-b^+_{\iota}(W) + \kappa (Y', \frakt', \iota' )
\end{split}
\end{align*}
holds.
\end{theo}

\begin{proof}
As noted in \cref{rem: SWF spherical disjoint}, 
\cref{most general main theo strong SWF spherical kappa} holds also for disjoint unions of spin rational homology spheres with odd involution that are SWF-spherical.
Applying this to $\tilde{Y} = \sqcup_{i=1}^N  (Y_i, \mathfrak{t}_i, \iota_i)$ and $\tilde{Y}' =\sqcup_{j=1}^{N'}  (Y_j', \mathfrak{t}_j', \iota_j')$, we have
\begin{align*}
\begin{split}
&-\frac{\sigma(W)}{16} + \kappa (\tilde{Y}, \frakt, \iota ) + A(N(W, \tilde{Y}, \frakt, \iota, \tilde{Y}', \frakt', \iota'))\\
\leq& b^+(W)-b^+_{\iota}(W) + \kappa (\tilde{Y}', \frakt', \iota' ).
\end{split}
\end{align*}
These disjoint unions are equivariantly homology cobordant to $Y$ and $Y'$, these induce local equivalecnes between $DSWF$.
Since $\kappa$ is local equivalence invariant, this completes the proof.
\end{proof}

\subsection{Calculations}
\label{subsection: Calculations}

We carry out calculations of the doubled Seiberg--Witten Floer stable homotopy type $DSWF_G(Y, \frakt,\iota)$ and equivariant $K$-theoretic Fr{\o}yshov invariant $\kappa(Y, \frakt,\iota)$ for some $(Y,\frakt,\iota)$.
The most general statement on our calculation in this \lcnamecref{subsection: Calculations} is summarized in \cref{theo:kappa connecedt sum Seifert cal most general}.

\begin{prop}
\label{ex: kappa in PSC case}
Let $(Y,\frakt)$ be a spin rational homology 3-sphere equipped with an involution $\iota$ preserving $\frakt$.
Assume that $Y$ admits an $\iota$-invariant positive scalar curvature metric $g$.
Construct $DSWF_G(Y, \frakt,\iota)$ using $g$, as well as non-equivariant case \cite[Subsection~5.1]{Ma03}. Then we have that
\begin{align*}
DSWF_G(Y, \frakt,\iota)=[(S^0,0,n(Y, \frakt,g)/2)]
\end{align*}
and 
\[
\kappa(Y, \frakt, \iota) = -n(Y, \frakt,g)/2,\quad{\text and}\quad
\kappa(-Y, \frakt, \iota) = -\kappa(Y, \frakt, \iota).
\]
Moreover, $(Y, \frakt, \iota)$ is SWF-spherical.
\end{prop}

\begin{proof}
The proof is essentially the same as Manolescu's original case without involution \cite{Ma03}.
(We sketch the argument also in the proof of \cref{theo:Seifert cal most general}.)

To check the equality $\kappa(-Y, \frakt, \iota) = -\kappa(Y, \frakt, \iota)$, recall the formula \eqref{eq: nminusn}.
Now we have $\Ker D_{(Y, \frakt,g)}=0$ since $g$ is a positive scalar curvature metric.
This completes the proof.
\end{proof}

\begin{rem}
In the non-equivariant case \cite[Subsection~5.1]{Ma03}, the $K$-theoretic Fr{\o}yshov invariant $\kappa(Y, \frakt)$ is given by 
\[
\kappa(Y, \frakt) = -n(Y, \frakt,g)
\]
whenever $Y$ admits a positive scalar metric $g$.
\cref{ex: kappa in PSC case} says that the equivariant $\kappa$ is the half of the non-equivariant one provided the existence of $\iota$-invariant positive scalar curvature metric.
Intuitively, this difference arises from the fact that the $I$-invariant part of the space of spinors is the `half' of the whole space of spinors.
\end{rem}

\begin{ex}
\label{ex: S3 case}
Consider $Y=S^3$.
We drop the unique spin structure on $S^3$ from our notation.
Let $g$ denote the standard metric on $S^3$, which has positive scalar curvature.
Then we have 
\begin{align*}
DSWF_G(S^3, \iota)=[(S^0,0,0)]
\end{align*}
and 
\[
\kappa(S^3, \iota) = 0
\]
for every involution $\iota$ preserving the metric $g$.
For example, if we regard $S^3$ as a subset of $\C^2$,
the complex conjugation on $\C^2$ defines such an involution $\iota$ on $S^3$.
This example can be generalized as in \cref{ex: lens}.
\end{ex}

\begin{ex}
\label{ex: lens}
Let $p,q$ be coprime natural numbers.
Regard the lens space $Y=L(p,q)$ as a subset of $\C^2$, and equip $Y$ with the standard metric $g$, which has positive scalar curvature.
The complex conjugation on $\C^2$ defines an involution $\iota$ on $Y$ that preserves $g$.
The fixed-point set of $\iota$ is non-empty and of codimension-2, wihch is called two bridge knot/link.
The lens space $Y$ admits at most two spin structures.
If $p$ is odd, a spin structure is unique.
In any case, it is easy to see that any orientation-preserving diffeomorpshim on $Y$ preserves each spin structure.
Thus we have, for a spin structure $\frakt$ on $L(p,q)$,
\begin{align*}
DSWF_G(L(p,q), \frakt,\iota)=[(S^0,0,n(L(p,q), \frakt,g)/2)]
\end{align*}
and 
\[
\kappa(L(p,q), \frakt, \iota) = -n(L(p,q), \frakt,g)/2 = -\kappa(-L(p,q), \frakt, \iota).
\]
Using the Fr{\o }yshov invariant $\delta$, we have
\[
\kappa(L(p,q), \frakt, \iota) = \delta(L(p,q), \frakt,g)/2 = -\kappa(-L(p,q), \frakt, \iota).
\]
\end{ex}

We also provide several hyperbolic examples. 
\begin{ex}\label{hyperbolic example}
 Let $Y_W$ be the Weeks manifold, which is the closed hyperbolic 3-manifold with minimal volume $0.94270 \cdots $. It is known that $Y_W$ can be written as the branched double covering space of $9_{49}$(\cite[Section A.4]{AA98}).
 Let $\iota$ be the covering involution $\iota$ on $Y_W$. By \cite[Corollary 2]{MJ01}, we can take $\iota$ so that $\iota$ preserves the hyperbolic metric on $Y_W$.
 The notion $\mathfrak{t}$ denotes the unique spin structure on $Y_W$.
 From \cite[Theorem 1]{FM21}, we see that there is no irreducible solution to the Seiberg--Witten equations for $(Y_W, \mathfrak{t})$.
So, the situation is completely the same as \cref{ex: kappa in PSC case}.
We have 
 \[
 \kappa (Y_W, \mathfrak{t}, \iota) = \frac{1}{2}\delta (Y_W, \mathfrak{t}). 
 \]
From \cite[Theorem 1.3]{MO07}, we have 
 \[
\delta (Y_W, \mathfrak{t}) = - \frac{1}{8} \sigma (9_{49})= \frac{1}{2}. 
 \]
Thus we conclude 
 \[
 \kappa (Y_W, \mathfrak{t}, \iota) = \frac{1}{4}. 
 \]
 Similarly, for the hyperbolic $\Z_2$-homology 3-spheres listed in \cite[Table 1]{FM21}, similar computations hold. 
\end{ex}

A large class of examples are obtained from Seifert homology spheres.
Recall that, for coprime natural numbers $a_{1}, \ldots, a_{n}$, the Seifert 3-manifold $\Sigma(a_{1}, \ldots, a_{n})$ is given as a subset of $\C^n$, and $\Sigma(a_{1}, \ldots, a_{n})$ is an integral homology sphere.
Denote by $\frakt$ the the unique spin structure on $\Sigma(a_{1}, \ldots, a_{n})$, but we shall often drop $\frakt$ from our notation.
A {\it Seifert metric} $g$ on $\Sigma(a_{1}, \ldots, a_{n})$ is a Riemannian metric given by
\begin{align}
\label{eq: Seifert metric}
\eta^2 +  g_{S^2(a_{1}, \ldots, a_{n}) }     
\end{align}
where $g_{S^2(a_{1}, \ldots, a_{n})}$ is the orbifold metric on $S^2$ of type $(a_{1}, \ldots, a_{n})$ and $i\eta$ is a connection on the circle bundle $Y \to S^2(a_{1}, \ldots, a_{n})$.
Henceforth we consider such a metric on $\Sigma(a_{1}, \ldots, a_{n})$.

\begin{theo}
\label{theo:Seifert cal most general}
Let $a_{1}, \ldots, a_{n}$ be coprime natural numbers.
Suppose that $a_{1}$ is even number.
Set $Y=\Sigma(a_{1}, \ldots, a_{n})$, and define an involution $\iota : Y \to Y$ by
\[
\iota(z_1,z_2, \ldots,z_n)=(-z_1,z_2,\ldots,z_n).
\]
Then we have that
\begin{align}
\label{eq: DSWF for Seifert}
DSWF_G(Y, \iota)=[(S^0,0,\bar{\mu}(Y)/2)].
\end{align}
and that
\begin{align*}
\kappa(Y, \iota) = -\bar{\mu}(Y)/2,\quad {\text and} \quad \kappa(-Y, \iota) = -\kappa(Y, \iota).    
\end{align*}
Moreover, $(Y, \iota)$ is SWF-spherical.
\end{theo}

\begin{proof}
First, let us recall a generality of Conley index theory.
We use some terms of Conley index theory explained in \cite{Ma03}.
Let $G'$ be a finite group and $H'$ be a normal subgroup of $G'$.
Let $X$ be a pointed $G'$-space.
Suppose that $X$ is equipped with an $\R$-action, and that the $\R$-action commutes with the $G'$-action.
Given an isolated invariant set $S$ in $X$.
We claim that there is a homotopy equivalence from the $G'/H'$-Conley index $I(S^{H'})$ to the $H'$-invariant part $I(S)^{H'}$ of the $G'$-Conlex index of $S$.
Indeed, it is straightforward to see that, for an index pair $(N,L)$ for $S$, the pair $(N^{H'}, L^{H'})$ is an index pair for $S^{H'}$.
Then the claim follows from the uniqueness of the Conley index up to homotopy equivalence.

Apply the argument in the last paragraph to $G'=\Z_2 \times \Z_4$ generated by $I$ and $j \in G$ and $H'=\Z_4=G$. 
Then the remaining main task to us is to calculate the Conlex index of the $I$-invariant part of the $\R$-fixed-point set of the configuration space.
By a result by Mrowka--Ozsv{\'a}th--Yu~\cite{MOY97}, the set of irreducible solutions to a Seifert 3-manifold $Y$ with a Seifert metric splits into two disjoint sets $\Sol^+(Y,\frakt,g)$ and $\Sol^-(Y,\frakt,g)$, and they are interchanged by the action of $j$.
(See also \cite{Sto20}.)

In \cref{lem: lifting path}, we will see that $(\iota, \tilde{\iota})$ is isotopic to an element of the kernel of the natural map from the automorphism group of the spin structure to the isometry group of $Y$ for a Seifert metric.
Using this fact, we will see in \cref{lem: I-inv part of Pin(2)} that
the intersection of the $I$-invariant part and the $\R$-invariant part (i.e. the critical point of the Chern--Simons--Dirac functional) of the configuration space consists only the reducible solution.

The remaining argument is almost the same as the calculation of the Seiberg--Witten Floer stable homotopy type for a 3-manifold with positive scalar curvature metric.
First, Nicolaescu~\cite[Section 2.3]{Nicol99} proved that the Dirac operator has zero kernel for general Seifert homology spheres with Seifert metrics.
Using this fact, under a homotopy to kill the quadratic term of the Seiberg--Witten flow, we may take a common isolating neighborhood of the reducible solution, and thus we may take an index pair of the reducible solution to be a standard index pair of of a linear flow.
Hence the Conley index of the $I$-invariant part is a (virtual) sphere of certain dimension.
By the assumption that $b_1(Y)=0$, the representation $\tilde{\C}$ does not appear in this sphere.
This shows that $(Y,\iota)$ is SWF-spherical: for any Seifert metric $g$ and any lift $\tilde{\iota}$, the $I$-invariant part Conley index is a (virtual) sphere consisting of copies of $\C_+\oplus \C_-$.
The contribution from the spinor part is $(\C_+ \oplus \C_-)^{\dim_{\C} V(\C_+)^0_{\lambda}+\dim_{\C} V(\C_-)^0_{\lambda}}$, but this cancels out the desuspension in the definition of the doubled Seiberg--Witten Floer $G$-spectrum class.
Thus the contribution to this sphere is only $n(Y,g)/2$, which coincides with $\bar{\mu}(Y)/2$ by a result by Ruberman--Saveliev~\cite[Theorem~1]{RS11}.
This completes the proof of \eqref{eq: DSWF for Seifert}.
The calculation of $\kappa$ immediately follows from this.
\end{proof}

\begin{rem}
\label{rem: general involusion not minus}
By the proof of \cref{theo:Seifert cal most general}, it is evident that the statement of \cref{theo:Seifert cal most general} holds for for all odd involution on $\Sigma(a_{1}, \ldots, a_{n})$ which is isotopic to the identity.

 It is known that, up to conjugation, there are not many types of involutions (or, more genrally, finite group actions) on Seifert homology spheres.
 For example, for the Brieskorn homology sphere $Y=\Sigma(p,q,r)$, any
finite group action on $Y$ is conjugate to a subgroup of the canonical $O(2)$-action generated by the standard $S^1$ action on $Y$ and the complex conjugation on $Y \subset \C^3$.
See \cite{AH21} for example.

Recall also that, except for $S^3$ and $\Sigma(2,3,5)$, the mapping class group $\Sigma(p,q,r)$ is isomorphic to $\Z_2$: $\pi_0(\Diff^+(\Sigma(p,q,r))) \cong \Z_2$.
(See, for example, \cite{BO91,MS13}.)
Therefore, the ``half'' of the orientation-preserving diffeomorphisms of $\Sigma(p,q,r)$ are isotopic to the identity.
\end{rem}

We give proofs of facts used in the proof of \cref{theo:Seifert cal most general}.

\begin{lem}
\label{lem: lifting path}
Let $(Y,\frakt)$ be a spin rational homology sphere.
Let $\iota$ be a smooth involution preserving $\frakt$.
Let $g$ be an $\iota$-invariant metric on $Y$.
Suppose that $\iota$ is isotopic to the identity of $Y$ through the isometry group of $(Y,g)$.
Then a path from $\iota$ to $\mathrm{id}_Y$ in the isometry group of $(Y,g)$ lifts to a path from $\tilde{\iota}$ to the identity $\id$ or $-\id$ on the spin structure in the automorphism group of the spin structure.
\end{lem}

\begin{proof}
Let $P_{SO}$ be the oriented orthogonal frame bundle of $Y$ and $P_{Spin}$ is the principal $Spin(3)$ bundle of the spin structure $\mathfrak t$.
We denote by the induced action on $P_{SO}$ from $\iota$ by the same notation.
Note that $P_{Spin} \times [0,1]$ is a fiberwise double covering of $P_{SO} \times [0,1]$, and the path from $\iota$ to the identity of $P_{SO}$ induces an automorphism $f$ of the fiber bundle $P_{SO} \times [0,1]$ and the automorphism of the fundamental group of $P_{SO} \times [0,1]$ induced by $f$ is the identity. Thus $f$ lifts to an automorphism $\tilde{f}$ on $P_{Spin} \times [0,1]$. If necessary, we change that map by composing covering transformation, the lift coincides with the automorphism which is given by the path from $\tilde{\iota}$ to $\id$ or $-\id$.
It is easy to see that $\tilde{f}$ is a isomorphism of the principal $Spin(3)$ bundle, since it is a lift of the isomorphism of the principal $SO(3)$ bundle on $P_{SO} \times [0,1]$ and $\tilde{\iota} \times \id_{[0,1]}$ is a isomorphism of principal $Spin(3)$ bundle. 
\end{proof}

\begin{lem}
\label{lem: I-inv part of Pin(2)}
Let $(Y,\frakt)$ be a spin rational homology sphere,
$\iota$ be a smooth odd involution preserving $\frakt$, and $g$ be an $\iota$-invariant metric on $Y$.
Suppose that the set of irreducible solutions $\Sol^{\rm irr}(Y,\frakt,g)$ to the Seiberg--Witten equations on $(Y,\frakt,g)$ consists of two disjoint set $\Sol^+(Y,\frakt,g)$ and $\Sol^-(Y,\frakt,g)$, and $\Sol^+(Y,\frakt,g)$ and $\Sol^-(Y,\frakt,g)$ are interchanged by the action of $j$.
Suppose also that $\iota$ is isotopic to the identity of $Y$ through the isometry group of $(Y,g)$.
Then we have 
\[
\Sol^{\rm irr}(Y,\frakt,g)^{I}=\emptyset.
\]
\end{lem}

\begin{proof}
First, recall some generalities on actions on the irredicible solutions $\Sol^{\rm irr}(Y,\frakt,g)$.
The automorphism group  $\Aut(Y,\frakt,g)$ of the spin structure acts on $\Sol^{\rm irr}(Y,\frakt,g)$ via pull-back.
The gauge group $\mathcal{G}=\Map(Y,S^1)$ and $\Aut(Y,\frakt,g)$ have non-empty intersection: $\{1, -1\} \subset \mathcal{G}$, which corresponds to $\{(\id_Y, \id_{\frakt}), (\id_Y, -\id_{\frakt})\} \subset \Aut(Y,\frakt,g)$.
The property of an element in $\Aut(Y,\frakt,g)$ (resp. in $\mathcal{G}$) that it does not interchange the two disjoint sets $\Sol^+(Y,\frakt,g)$ and $\Sol^-(Y,\frakt,g)$ is invariant under isotopy in $\Aut(Y,\frakt,g)$  (resp. in $\mathcal{G}$).

By \cref{lem: lifting path}, the pair $(\iota, \tilde{\iota}) \in \Aut(Y,\frakt,g)$ is isotopic to either the identity element $(\id_Y, \id_{\frakt})$ or $(\id_Y, -\id_{\frakt})$ in $\Aut(Y,\frakt,g)$.
Not only $(\id_Y, \id_{\frakt})$, the element $(\id_Y, -\id_{\frakt})$ also does not interchange $\Sol^+(Y,\frakt,g)$ and $\Sol^-(Y,\frakt,g)$.
Indeed, $(\id_Y, -\id_{\frakt})$ corresponds to $-1 \in S^1 \subset \mathcal{G}$, which is isotopic to the identity element in $\mathcal{G}$.
Thus we deduce that $(\iota, \tilde{\iota})$ does not interchange $\Sol^+(Y,\frakt,g)$ and $\Sol^-(Y,\frakt,g)$.

Recall that the involution $I : \Gamma(\mathbb{S}) \to \Gamma(\mathbb{S})$ is given as the composition of $\tilde{\iota}$ (strictly speaking, regarded as the pair  $(\iota, \tilde{\iota})$) and the right multiplication by $j$.
By the above argument, $\tilde{\iota}$ does not interchange $\Sol^+(Y,\frakt,g)$ and $\Sol^-(Y,\frakt,g)$.
However, by the assumption of the \lcnamecref{lem: I-inv part of Pin(2)}, the action of $j$ interchanges $\Sol^+(Y,\frakt,g)$ and $\Sol^-(Y,\frakt,g)$.
In total, $I$ interchanges $\Sol^+(Y,\frakt,g)$ and $\Sol^-(Y,\frakt,g)$.
Thus we obtain $\Sol^{\rm irr}(Y,\frakt,g)^{I}=\emptyset$.
\end{proof}

Recall the definition of the equivariant connected sum \cref{Connected sum equi}, \cref{rem: connected sum dependence on base point}.
We can compute $\kappa$ for equivariant connected sums of lens spaces and Seifert homology spheres by connected sum formula.
We summarize as follows:

\begin{theo}
\label{theo:kappa connecedt sum Seifert cal most general}
For $N, N' \geq 0$ and $i \in \{1, \ldots, N\}$ and $j \in \{1, \ldots, N'\}$, let $(Y_i, \frakt_i, \iota_i)$ be the following triple: $Y_i=\pm L(p,q)$ for some coprime natural numbers $p,q$, and
$\frakt_i$ is a spin structure on $Y_i$ (which is unique if $p$ is odd), and $\iota_i : Y_i \to Y_i$ is the complex conjugation.
Let $(Y_j', \frakt_j', \iota_j')$ be the following triple:
$Y_j'=\pm \Sigma(a_{1}, \ldots, a_{n})$ for some coprime natural numbers $a_{1}, \ldots, a_{n}$ with $a_1$ even,
$\frakt_j'$ is the unique spin structure on $Y_j'$, and $\iota_j' : Y_j' \to Y_j'$ is given by $\iota_j'(z_1,z_2, \ldots,z_n)=(-z_1,z_2,\ldots,z_n)$.

For each $i$, choose a connected component of the fixed-point set of $\iota_i$, which is unique if and only if $p$ is odd.
(Note that the fixed-point set of $\iota_j'$ is connected.)
Let $o(\iota_i)$ be an orientation of this component of the fixed-point set, and similarly equip the fixed-point set of $\iota_j'$ with an orientation $o(\iota_j')$. 
Then, for the equivariant connected sum
\[
(Y,\frakt,\iota) = \#_{i=1}^{N} (Y_i,\frakt_i,\iota_i,o(\iota_i)) \#_{j=1}^{N'} (Y_j',\frakt_j',\iota_j',o(\iota_j')),
\]
along these connected components of fixed-point sets, we have
\[
\kappa(Y,\frakt,\iota)
= \frac{1}{2}\sum_{i=1}^N\delta(Y_i,\frakt_i)- \frac{1}{2}\sum_{j=1}^{N'}\bar{\mu}(Y_j')
= \frac{1}{2}\delta\left(\#_{i=1}^N(Y_i,\frakt_i)\right)- \frac{1}{2}\bar{\mu}\left(\#_{j=1}^{N'}Y_j'\right).
\]
\end{theo}

\begin{proof}
This follows from calculations \cref{ex: lens}, \cref{theo:Seifert cal most general}, well-behavior of the notion of locally DSWF-spherical given in \cref{locally DSWF-spherical conn sum}, and the connected sum formula \cref{cor: connected sum formula kappa}.
\end{proof}

\subsection{Knot invariants}\label{Knot invariants}
For any oriented knot $K$ in $S^3$, we can uniquely associate a double branched cover 
\[
\Sigma (K) \to S^3
\]
and an involution $\iota_K$ on $\Sigma (K)$. It is proven that $H^1(\Sigma (K) ;\Z_2) = \{0\} $. 
So, the isomorphism class of the spin structure $\frakt$ on $\Sigma (K)$ is automatically preserved by $\iota_K$.

\begin{defi}For any knot $K$ in $S^3$, we define the {\it Seiberg--Witten Floer stable homotopy type} of $K$ by 
\[
DSWF(K) :=  DSWF_G(\Sigma (K) , \frakt,\iota). 
\]
We also define the {\it Seiberg--Witten Floer $K$-theory} of $K$ by 
\[
DSWFK(K) :=  DSWFK_G(\Sigma (K) , \frakt,\iota), 
\]
and the {\it $K$-theoretic Fr{\o}yshov invariant} of $K$ by
\[
\kappa(K) := k(DSWF(K)) \in \Q.
\]
\end{defi}

Since the $\Z_2$-equivariant 3-manifold $\Sigma (K)$ does not depend on the choices of orientations of $K$, the condition (ii) in \cref{main knot} holds, i.e. 
\[
\kappa (K) = \kappa (-K). 
\]
Also, \cref{main knot} (iii) and (iv) follows from \cref{main theo} (i) and (iii).

\begin{rem}\label{generalY}
Since the observation above (begining of \Cref{Knot invariants}) only use the homological information of $S^3$, the same observation can be applied to the case of a pair $(Y, K)$ of an oriented homology 3-sphere $Y$ and a knot $Y$. Thus, we can define the Seiberg--Witten Floer stable homotopy type $DSWF(Y,K)$, the Seiberg--Witten Floer $K$-theory $DSWFK(Y,K)$ and the $K$-theoretic Fr{\o}yshov invariant $\kappa(Y, K)$. 
In \cref{concordance}, we will see that $\kappa(K)$ is a concordance invariant. For the generalization $\kappa(Y, K)$, we can see that $\kappa(Y, K)$ is invariant under {\it homology concordance}, which is defined by the following way: if there are a homology cobordism $W$ from $Y_0$ to $Y_1$ and an embedded concordance (annulus) in $W$ from $K_0$ to $K_1$,  then we call $(Y_0, K_0)$ and $(Y_1, K_1)$ are homology concordance. 
\end{rem}

\begin{rem}\label{b+iotarem}
Before proving several properties of $\kappa$, we explain how to compute $b^+_{\iota}$ for double branched covers.  Let $X$ be a  smooth closed 4-manifold $X$ with $H_1(X; \Z) =\{ 0\}$ and an embedded oriented connected surface $S$ with $[S]$ is divisible by $2$, we have a double branched cover  
\[
\Sigma (S) \to X. 
\]
It is proven in \cite[Page 254 Lemma]{Hir69} that 
\begin{align}\label{b+iota}
b^+_{\iota }(\Sigma (S)) = b^+(X). 
\end{align}
Also, we will use \eqref{b+iota} for 4-manifolds $X$ with several $S^3$-boundaries $S^3_1 \cup \cdots \cup S^3_m$. Suppose $S$ is properly smoothly embedded surface in $X$ such that $S\cap S_i^3$ is a knot for any $i$. We can easily see that \eqref{b+iota} is still true in such a situation. Let us give a short sketch of proof:  When the covering action $\wt{X}  \to X$ is free, then the pull-back gives an isomorphism $H^* (X; \Q ) \to H^* (\wt{X}; \Q)$. Then, for the case $\pi: \Sigma (S) \to X$, we decompose $X$ into two parts: $\nu(S) \cup (X \setminus \nu(S))$, where $\nu(S)$ is a normal disk neighborhood of $S$. Since the action on $X \setminus \nu(S)$ is free, we have an isomorphism $H^* (X \setminus \nu(S); \Q ) \to H^* (\pi^{-1}(X \setminus \nu(S))  ; \Q)$. Then, the Mayer-Vietoris exact sequence combined with the five lemma, we can see $\pi^*: H^*(X)  \to H^* (\Sigma (S))$ is an isomorphism. Also, it is not difficult to check that this correspondence does not change $b^+$. 
\end{rem}

Henceforth, if we say a surface, we always assume that it is connected.

When $K$ and $K'$ are isotopic, then there is an orientation preserving diffeomorphism 
\[
\phi : \Sigma (K) \to \Sigma (K')
\]
which is $\Z_2$-equivariant. This enables us to prove the following: 
\begin{lem}
$DSWF(K)$ and $\kappa(K)$ are isotopy invariants of knots. 
\end{lem}
The invariant $\kappa(K)$ is actually invariant under knot concordance. 
Let $\mathcal{C}$ be the knot concordance group. 
\begin{lem}\label{concordance} The correspondence 
\[
\mathcal{C}\to \mathcal{LE}_G, \ K \mapsto [DSWF_G(\Sigma (K) , \frakt,\iota)]_{\text{loc}} 
\]
is a well-defined map. Moreover, since $\kappa(K)$ can be recovered from the local equivalence class 
\[
[DSWF_G(\Sigma (K) , \frakt,\iota)]_{\text{loc}},
\]
$\kappa(K)$ is a concoradance invariant. 
\end{lem}
\begin{proof}
Let $S$ be a concordance from $K$ to $K'$ in $I \times S^3$. Then the double branched cover $\Sigma (S)$ along $S$ gives a $\Z_2$-equivariant $\Z_2$-homology cobordism from $\Sigma(K) $ to $\Sigma(K')$. Note that the isomorphism class of the spin structure on $\Sigma (S)$ is preserved under the involution. 

We have an associate doubled cobordism map
\begin{align}
\label{eq: doubled cob map}
D(f) : 
\Sigma^{m_0\tilde{\C}}\Sigma^{n_0(\C_+ \oplus \C_-)}D(I_{-\mu}^{-\lambda} (\Sigma (K))) \to \Sigma^{m_1\tilde{\C}}\Sigma^{n_1(\C_+ \oplus \C_-)}D(I^\mu_\lambda (\Sigma (K' )) )).   
\end{align}
\eqref{eq:m difference} implies 
\begin{align*}
& m_0-m_1 \\ 
&=\dim_\R (V_1(\tilde{\R})^0_{\lambda})-\dim_\R(V_0(\tilde{\R})^0_{-\mu}) - b^+(\Sigma(S)) + b^+_\iota(\Sigma(S)) \\
&=\dim_\R (V_1(\tilde{\R})^0_{\lambda})-\dim_\R(V_0(\tilde{\R})^0_{-\mu}).   \\
\end{align*}
Here we used \cref{b+iotarem} to calculate $b^+_{\iota}$.
This means $D(f)$ is a local map. By considering the same discussion for $-S$, we see 
\[
[DSWF_G(\Sigma (K) , \frakt,\iota)]_{\text{loc}} = [DSWF_G(\Sigma (K' ) , \frakt,\iota)]_{\text{loc}}. 
\]
By \cref{lem:loc eq}, we saw $\kappa(K)$ is recovered from $[DSWF_G(\Sigma (K) , \frakt,\iota)]_{\text{loc}}$. This completes the proof. 
\end{proof}

Let us give calculations of $\kappa$ for two bridge knots.  
Before that, we remark on the sign convention of the knot signature.

\begin{rem}
\label{rem: sign convention of knot signature}
We use the same convention of the signature as in \cite{Sa19, Ba14}.
Namely, the signature of $T(2,3)$ is given by $\sigma(T(2,3))=-2$.
On the other hand, the sign convention of the signature in the Knot Atlas \cite{KnotAtlas} is opposite to ours.
\end{rem}

\begin{lem}\label{single two bridge}
For any two bridge knot $K(p,q)$ whose branched cover is $L(p,q)$, one has 
\begin{align}
\label{eq: DSWF Lpq}
DSWF(K)=[(S^0,0,\frac{1}{16} \sigma (K(p,q) )]
\end{align}
and 
\begin{align}
\label{two bridge kappa}
\kappa (K(p,q)) = -\frac{1}{16} \sigma (K(p,q)). 
\end{align}
\end{lem}

\begin{proof}
Recall that the two bridge knot $K(p,q)$ is regarded as the fixed-point set of the complex conjugation $\iota$ of the lens space $L(p,q)$.
Let $g$ denote the standard positive scalar curvature metric on $L(p,q)$.
By \cref{ex: lens}, we have
\begin{align}
\label{eq: kappa Lpq n 346}
\kappa (K(p,q)) =  \kappa(L(p,q) , \frakt, \iota) = -n(L(p,q), \frakt,g)/2,
\end{align}
where $\frakt$ is the spin structure on $L(p,q)$, which is unique since we supposed that $p$ is odd.
On the other hand, since $g$ is a positive scalar curvature metric again, we also have
\begin{align}
\label{eq: kappa n 346}
\delta (L(p,q), \frakt ) = -n(L(p,q), \frakt,g). 
\end{align}
(See \cite{Ma16}.)
Also, it follows from \cite[Theorem 1.2]{MO07} that 
\begin{align}
\label{eq: 2d sign n 346}
\frac{1}{2}d(\Sigma(K(p,q)), \mathfrak{t})  = -\frac{1}{8} \sigma (K(p,q)). 
\end{align}
Here we used the convetion of the Heegaard Floer $d$-invariant so that $d(\Sigma(2,3,5))=2$. (Note that the convention used in \cite[Theorem 1.2]{MO07} is $d(\Sigma(2,3,5))=\frac{1}{2}$.) 
In general, for a homology sphere $Y$, the $d$ invariant and the $\delta$ invariant are related by 
\begin{align}
\label{eq: delta d2}
\delta(Y) = d(Y)/2.
\end{align}
(See, for example, \cite[Remark~1.1]{LRS18}.)
Now it follows from \eqref{eq: kappa Lpq n 346}, \eqref{eq: kappa n 346}, \eqref{eq: 2d sign n 346}, \eqref{eq: delta d2}
that
\[
\kappa (K(p,q)) = -\frac{1}{16} \sigma (K(p,q)) = -n(L(p,q), \frakt,g)/2,
\]
and we have \eqref{eq: DSWF Lpq} and \eqref{two bridge kappa}.
\end{proof}



Now we give a proof of \cref{two bridge kappa12}. 
\begin{proof}[Proof of \cref{two bridge kappa12}]
It follows from \cref{theo:kappa connecedt sum Seifert cal most general},  \cref{single two bridge} and the fact $\Sigma (K(p,q)) = L(p,q)$ 
\end{proof}

We also provide calculations of $\kappa(Y)$ for a certain class of torus knots giving a proof of \cref{torus knot calculation}. 


\begin{proof}[Proof of \cref{torus knot calculation}]
It follows from \cref{theo:kappa connecedt sum Seifert cal most general} and the fact $\Sigma (T(p,q)) = \Sigma (2,p,q)$. 
\end{proof}

\begin{ex}
\label{cor: kappa for torus knot 6npm1}
The kappa invariants $\kappa(\#_m T(3,6k\pm1))$ are given as in Table~\ref{table: kappa for T 6npm1}.
Here $n$ and $m$ are positive integers, and $ K^*$ denotes the mirror image of $K$. 
 \begin{table}[htb]
 \caption{{$\kappa(\#_m T(3,6k\pm1))$}}
  \label{table: kappa for T 6npm1}
 \begin{tabular}{ll}
     $\kappa (\#_m T(3,12n-5)) = -m/2$, & $\kappa (\#_m T(3,12n-5)^*) = m/2$,\\
     $\kappa (\#_m T(3,12n-1)) = 0$, & $\kappa (\#_m T(3,12n-1)^*) = 0 $,\\
     $\kappa (\#_m T(3,12n-7)) = m/2$, & $\kappa (\#_m T(3,12n-7)^*) = -m/2 $,\\
     $\kappa (\#_m T(3,12n+1)) = 0$, & $\kappa (\#_m T(3,12n+1)^*) = 0 $. 
\end{tabular}
\end{table}
\end{ex}

We give calculations of $\kappa$ for prime 3-bridge knots whose branched covering space is a hyperbolic 3-manifold with small volume. 

\begin{ex}\label{several hype ex}From \cref{hyperbolic example}, we can see 
 $\kappa (9_{49}) = \frac{1}{4}$. In \cite[Table 7]{BS21}, several hyperbolic 3-manifolds and the double branched covers of knots are identified using SnapPy. 
Here we use the knots listed in \cite[Table 7]{BS21}: see \cref{table: hyperbolic}. 
Also, note that the double branched covering spaces of the knots in \cref{table: hyperbolic} are also in \cite[Table 1]{FM21}.  
Using the same argument given in \cref{hyperbolic example}, for any knot listed in Table \ref{table: hyperbolic}, we have 
\[
\kappa (K) = \frac{1}{2} \delta (\Sigma (K), \mathfrak{t}) ,  
\]
where $\mathfrak{t}$ denotes the unique spin structure. 
In \cite{L015}, it is proven that if $K$ is quasi-alternating, then the equality 
\[
\delta (\Sigma (K), \mathfrak{t}) = - \frac{1}{8 } \sigma (K)
\]
holds. Except for $K11n92$ and $K11n118$, it is checked that the knots in Table \ref{table: hyperbolic} are quasi-alternating. 

\end{ex}
 \begin{table}[htb]
 \caption{{Hyperbolic 3-manifolds with small volumes}}
  \label{table: hyperbolic}
 \begin{tabular}{ll}
     $(3, m003(-4,3)) = \Sigma (  10_{155})$ & $(14, m007(4,1)) = \Sigma ( K11n118)$\\
     $(8, m003(-4,1)) = \Sigma ( 10_{163})$ & $(15, m007(3,2) ) = \Sigma ( 9_{47}) $\\
   $(12,  m003(-5,3)) = \Sigma  (10_{156})$  & $(18, m006(-3,2)) = \Sigma ( K11n92) $ \\
    $(13, m007(1,2)) = \Sigma ( 10_{160})$   & 
\end{tabular}
\end{table} 

\begin{table}[htb]
 \caption{Kappa invariants for knots with hyperbolic branched coverings}
  \label{table: hyperbolic1}
 \begin{tabular}{ll}
     $\kappa (10_{155}) = 0 $ &  $ \kappa ( 10_{160})=-\frac{1}{4} $   \\
     $\kappa ( 10_{163}) = -\frac{1}{8}$ & $\kappa ( 9_{47}) = -\frac{1}{8} $\\
   $ \kappa  (10_{156}) =\frac{1}{8}$  & 
   
\end{tabular}
\end{table}


\section{Applications to knot theory}\label{Applications to knot theory}

\subsection{Branched covers of punctured 4-manifolds} 

For a smooth closed 4-manifold $X$ with $H_1(X; \Z) =0$ and an embedded oriented surface $S$ with $[S]$ is divisible by $2$, we have a double branched cover  
\[
\Sigma (S) \to X. 
\]
The following calculations are proven in \cite{Hi69, HS71}: 
\begin{itemize} 
\item[(i)]
$\sigma (\Sigma (S)) = 2 \sigma (X) - \frac{1}{2} [S]^2$
\item[(ii)]
$b^+(\Sigma (S)) = 2b^+(X) + g(S) -\frac{1}{4} [S]^2$, $b^-(\Sigma (S)) = 2b^-(X) + g(S) + \frac{1}{4} [S]^2$ and 
\item[(iii)] $b_1(\Sigma (S)) =0$. 
\end{itemize}

Also, in \cite[Theorem 1.1]{Na00}, it is proven that $\Sigma(S)$ has a spin structure if and only if $PD(w_2(X)) = \frac{1}{2}[S] \operatorname{mod} 2$.
We suppose $PD(w_2(X)) = \frac{1}{2}[S] \operatorname{mod} 2$. 
\begin{lem}\label{H1}
Under above assumptions, $H^1(\Sigma (S); \Z_2) =\{0\}$ holds. 
In particular, the spin structure on $\Sigma (S)$ is unique up to isomorphism. 
\end{lem}
\begin{proof}
By the arugument of the proof of \cite[Proposition 2.1]{KoMa95}, one can see that there is no 2-torsion in $H_*(\Sigma(S); \Z)$. Thus, the universal coefficient theorem implies 
\[
0 \to \operatorname{Ext}^1_\Z (H_0 (\Sigma(S); \Z), \Z_2) \to H^1(\Sigma(S); \Z_2) \to \Hom (H_1( \Sigma(S); \Z) , \Z_2) \to 0. 
\]
is exact. 
This shows $H^1(\Sigma (S); \Z_2) =\{0\}$.
\end{proof}

Let $K$ be an oriented knot in $S^3$. 
Let $X$ be a simply-connected 4-manifold $X$ bounded by $S^3$ and $S$ be an oriented compact connected surface $S$ bounded by $K$ satisfying $[S]$ is divisible by $2$.
Then one can associate the double branched cover 
\[
\Sigma (S) \to X
\]
of 4-manifolds with boundary. 

The following provides a sufficient condition so that the double branched cover has the unique spin structure. 

\begin{lem}\label{branched spin}
We suppose $PD(w_2(X)) = [S] \operatorname{mod} 2$. The double branched cover $\Sigma (S)$ along $S$ has the unique spin structure.  Moreover, the following equalities 
\begin{align}
&\sigma (\Sigma (S ))= 2 \sigma (X) - \frac{1}{2} [S]^2 + \sigma (K) ,  \\ 
& b^+ (\Sigma (S )) = 2b^+(X) + g(S) -\frac{1}{4} [S]^2  +  \frac{1}{2} \sigma (K), \\ 
& b^- (\Sigma (S )) = 2b^-(X) + g(S) +\frac{1}{4} [S]^2   - \frac{1}{2} \sigma (K), \text{ and}  \\ 
& b_1(\Sigma(S) ) =0 
 \end{align} 
hold.  
\end{lem}
\begin{rem}\label{homological cal of W}
We also use the computation of these invariants in the following setting: 
Let $(W, S)$ be a connected oriented knot cobordism from $(S^3, K)$ to $(S^3, K')$. Suppose $H_1(W; \Z)=0$, $[S]$ is divisible by $2$ and $PD(w_2(W)) = [S] \operatorname{mod} 2$.
Then, by the same discussions, we have 
\begin{align*}
&\sigma (\Sigma (W ))= 2 \sigma (W) - \frac{1}{2} [S]^2 - \sigma (K)+ \sigma (K') ,  \\ 
& b^+ (\Sigma (S )) = 2b^+(X) + g(S) -\frac{1}{4} [S]^2  -  \frac{1}{2} \sigma (K) + \frac{1}{2} \sigma (K') , \\ 
& b^- (\Sigma (S )) = 2b^-(X) + g(S) +\frac{1}{4} [S]^2   + \frac{1}{2} \sigma (K)-\frac{1}{2} \sigma (K') , \text{ and}  \\ 
& b_1(\Sigma(S) ) =0 . 
 \end{align*} 
\end{rem}

Recall that, as mentioned in \cref{rem: sign convention of knot signature}, our sign convention of the knot signature is the same as in \cite{Sa19, Ba14}.

\begin{proof}
When $X$ is a closed 4-manifold, the existence condition of a spin structure on the double branched cover is completely determined in \cite[Theorem 1.1]{Na00}. Define $X' := X \cup_{S^3} D^4$. Take a properly embedded oriented surface $S'$ in $D^4$ bounded by $-K^*$. Consider the union $S^\# := S \cup S' \subset X'$. Then we just apply \cite[Theorem 1.1]{Na00} to $(X', S^\#)$ and obtain a spin structure on $\Sigma (S^\#)$. As a restriction, we obtain a spin structure on $\Sigma (S )$. Next, we see the spin structure is unique up to isomorphism. 
It is sufficient to say $H^1(\Sigma (S) ; \Z_2) = \{0\}$. The Mayer-Vietoris exact sequence for $(\Sigma (S), \Sigma(S'))$ implies 
\begin{align*}
0\to  H^1 (\Sigma (S^\# ) ; \Z_2 ) \to  H^1 (\Sigma (S ) ; \Z_2 )\oplus H^1 (\Sigma (S' ) ; \Z_2 ) \to H^1 (\Sigma (K);  \Z_2 ) \to \cdots
\end{align*}
One can verify $H^1 (\Sigma (K);  \Z_2 )=0$. 
By \cref{H1}, we see $H^1 (\Sigma (S ) ; \Z_2 )= \{0\}$.

Next, we see the equations above. 
The double branched cover along $S^\#$ has the following decomposition: 
\[
\Sigma (S^\# ) =  \Sigma (S )\cup \Sigma (S' ).
\]
Thus, one has
\[
\sigma (\Sigma (S^\# ))  = \sigma (\Sigma (S )) +\sigma ( \Sigma (S' )) 
\]
and 
\[
b^+ (\Sigma (S^\# ))  = b^+ (\Sigma (S )) + b^+ ( \Sigma (S' )) . 
\]
On the other hand, we have 
\[
\sigma (\Sigma (S^\# ))=  2 \sigma (X) - \frac{1}{2} [S]^2 
\]
and 
\[
 b^+(\Sigma (S^\#)) = 2b^+(X) + g(S^\#) -\frac{1}{4} [S^\#]^2. 
\]
It is proven in \cite[Theorem 6]{GL78} that
\[
\sigma (\Sigma (S') ) = \sigma (-K^*)= -\sigma (K). 
\]
Also, we can verify 
\[
b^+ ( \Sigma (S' ) ) = g (S') +  \frac{1}{2} \sigma (-K^*)= g (S') -  \frac{1}{2} \sigma (K).  
\]

These equations imply 
\[
\sigma (\Sigma (S ))= 2 \sigma (X) - \frac{1}{2} [S]^2 - \sigma (-K^*)=2 \sigma (X) - \frac{1}{2} [S]^2 +  \sigma (K)
\]
and 
\[
 b^+ (\Sigma (S )) = 2b^+(X) + g(S) -\frac{1}{4} [S]^2   +   \frac{1}{2} \sigma (K). 
 \]
 The calculation $b_1(\Sigma(S) ) =0$ also follows from the Mayer--Vietoris exact sequence. 
\end{proof}

\subsection{Genus bounds from \cref{main theo}}\label{ Genus bounds from }
In this subsection, we provide a genus bound in 4-manifolds. Let $X$ be a smooth closed 3-manifold $H_1(X; \Z)= \{0\}$.
\begin{defi}For a fixed homology class $x  \in H_2(X; \Z)$ and an oriented knot $K$ in $S^3$, we define the {\it $X$-genus} of $K$ is defined by
\begin{align*}
g_{X, x} (K):= \min \{ g(S) | \text{$S$ is an oriented properly embedded connected surface } \\
\text{in $X \setminus D^4$ bounded by $K$}, \ x = [S] \in H_2(X \setminus D^4, \partial (X \setminus D^4)  ; \Z)  \}  
\end{align*}
\end{defi}
The $X$-genus has been studied in various situations. 
Also, by considering locally flat embedding, we define the topological version $g^{\mathrm{Top}}_{X, x} (K)$ of $g_{X, x} (K)$. In general, these two invariants $g^{\mathrm{Top}}_{X, x} (K)$ and $g_{X, x} (K)$ are different. 
When $X=S^4$ and $x=0$, $g_{X, x} (K)$ is called the {\it smooth slice genus} of $K$, and denoted by $ g_4(K)$. Obviously, we have 
\[
g_{X, x} (K) \leq g_4(K). 
\]
Also, when we take $K=U $(the unknot), $g_{X, x} (U)$ is called the {\it minimal genus} of $(X, x)$. 

We now prove \cref{main knot}. 
 \begin{proof}[Proof of \cref{main knot}]

 We consider the double branched cover $\Sigma(S)$ along $S$. By \cref{branched spin}, $\Sigma(S)$ has the unique spin structure. In particular, the branched involution preserves the isomorphism class of the spin structure. 
Now we apply \cref{main theo} and obtain 
\[
-\frac{\sigma(\Sigma(S))}{16}+ \kappa (K)
\leq b^+(\Sigma(S))-b^+_{\iota}(\Sigma(S)) + \kappa (K' ). 
\]
 
 Moreover, using \cref{branched spin}, \cref{homological cal of W} and \cref{b+iotarem}, we obtain the equalities
\begin{align}
&\sigma (\Sigma (S ))= 2 \sigma (W) - \frac{1}{2} [S]^2 + \sigma (K')- \sigma (K) \  \\ 
& b^+ (\Sigma (S )) = 2b^+(W) + g(S) -\frac{1}{4} [S]^2  -  \frac{1}{2} \sigma (K')- \frac{1}{2} \sigma (K) \text{ and }  \\
&  b^+_{\iota} (\Sigma (S)) = b^+(W). 
 \end{align} 
Thus, we obtain 
\begin{align*} 
&-\left(\frac{ 2 \sigma (W) - \frac{1}{2} [S]^2 + \sigma (K') -\sigma (K)}{16}\right) + \kappa (K) \\
& \leq b^+(W) + g(S) -\frac{1}{4} [S]^2   +  \frac{1}{2} \sigma (K')-\frac{1}{2} \sigma (K)+ \kappa (K ').
\end{align*}
 This completes the proof. 
 \end{proof}
 
 \begin{rem}
 When $W=I \times S^3$ and $K=U$, then \cref{main knot} implies 
 \begin{align} \label{four genus bound}
 -\frac{9}{16}\sigma(K') \leq  g_4(K') + \kappa(K')
  \end{align}
  for any knot $K'$ in $S^3$.
 For a connected sum of two bridge knots, \eqref{four genus bound} shows 
 \[
 -\frac{1}{2}\sigma(K) \leq  g_4(K) , 
 \]
 which is equivalent to a constraint from the signature given in \cref{000}. 
 Also, for torus knots, $g_4(T(p,q)) = \frac{(p-1)(q-1)}{2}$ is known as Milnor's conjecture proved by Kronheimer--Mrowka \cite{KM93}. Also see an alternative proof based on $\tau$-invariant \cite{OS03}.   On the other hand, for example, our inequalities and \cref{torus knot calculation} for $ \#_n T(3,7 )$ imply the following weak inequality: 
    \begin{align}\label{bad ineq}
    \frac{9}{2}n + \frac{1}{2}n= 5 n  \leq  g_4 ( \#_n  T(3,7) )= 6 n
    \end{align}
    for a positive integer $n$. 
Also, using \eqref{most general main theo strong SWF spherical kappa}, we obtain stronger inequality than \eqref{bad ineq} as follows: 
We now have $N(S^4, 0, \#_n T(3,7))=n$, which is defined in \eqref{most general main theo strong SWF spherical kappa}.
When $n \geq 2$, $N(S^4, 0, \#_n T(3,7)) \geq 2$ and 
\begin{align*} 
b^+(X)+ g_{X,x}(K) - \frac{1}{4}x^2 + \frac{1}{2}\sigma (K)  
\geq  5 n  - 4n  \geq 1
\end{align*} 
are satisfied. So, we can apply \eqref{most general main theo strong SWF spherical kappa} and obtain 
\begin{align*}
\begin{split}
&5n   +A(n ) \leq g_{4}(\#_n T(3,7)) = 6n \text{ when } n \geq 2, 
\end{split}
\end{align*}
 where 
\begin{align*}
A(N) = 
\begin{cases}
&1, \quad N=0,2 \mod 8\\
&2, \quad N=1,3,4,5,7 \mod 8\\
&3, \quad N=6 \mod 8.
\end{cases}    
\end{align*}
 \end{rem}
 
 \begin{rem}
When $K= K ' = U$ (the unknot), \cref{main knot} implies a version of 10/8-inequality corresponding to the inequality given in \cite{Ka17}.  
 In this case, \eqref{ineq1} gives 
 \[
  -\frac{\sigma(X)}{8} + \frac{9}{32}[S]^2 \leq b^+(X) + g(S) 
\]
for a closed oriented smooth 4-manifold $X$ with $H_1(X ;\Z) =0 $ and an oriented compact properly smoothly embedded surface $S$ in $X$ such that the homology class $[S]$ of $S$ is divisible by $2$ and $PD(w_2(X)) = [S]/2 \operatorname{mod} 2$. For example, for a positive integer $n$ and any class $x \in H_2(\#_nK3)$ satisfying that $x$ is divisible by $2$ and $0= x/2 \operatorname{mod} 2$, we have the following genus bound: 
\[
\frac{9}{32}[S]^2 \leq n + g_{\#_n K3, x}(U).
\]
Note that a similar inequality is proved using Donaldson's Theorem B and C in \cite{KoMa95}. 
\end{rem}
  \begin{rem}
 As it is remarked in \cref{generalY}, we can generalize the invariant $\kappa(K)$ to an invariant $\kappa(Y,K)$ for a pair of an oriented homology 3-sphere and an oriented knot $K$. We can also generalize \cref{main knot} to a theorem for such an invariant. 
 \end{rem}

 Now we give a proof of \cref{knot main app}. 
 \begin{proof}[Proof of \cref{knot main app}]Since it is proven in \cite[Theorem 1]{BBL20} that, for a positive integer $m$ coprime to $3$, 
\[
g^{\mathrm{Top}}_4 (T(3,m) ) = \left\lceil \frac{2m}{3} \right\rceil , 
\]
we have 
\[
g^{\mathrm{Top}}_4 ( T(3,11))= 8 .
\]
Also, from \cref{000}, we have a lower bound
\[
 8n = \frac{1}{2} | \sigma ( \#_n T(3,11)) ) |   \leq g^{\mathrm{Top}}_4 ( \#_n T(3,11)) \leq 8n. 
\]
This implies $g^{\mathrm{Top}}_4 ( \#_n T(3,11)) = 8n$. 
Thus we have 
\[
g^{\mathrm{Top}}_{X,x}(K \#  ( \#_n T(3,11) ) ) \leq g^{\mathrm{Top}}_{X,x}(K )+8n
\]
for every knot $K$ in $S^3$.
From \cref{torus knot calculation} and \cref{connected sum}, one has 
\[
\kappa (K \# \#_n T(3,11)) = \kappa(K)  . 
\]
Thus, by \cref{main knot}, we have \[
-\frac{\sigma(X)}{8} + \frac{9}{32}x^2 - b^+(X) + 9n - \kappa (K) - \frac{9}{16}\sigma (K) \leq  g_{X,x}(K \# (\#_n T(3,11)) ).
\]
This completes the proof.
\end{proof}

\begin{rem}
Note that for the torus knots $T(3,12l+1)$, $T(3,12l-1)$, $T(3,12(l+1)-5))$ and  $T(3,12(l+1)-7)$ for $l \geq 1$, the same argument of the proof of \cref{knot main app} can be applied. On the other hand, using \cref{manolescu1} and \cite[Theorem 1]{BBL20}, one can also see 
\[
\lim_{l \to \infty} ( g_{X,x}( T(3,12l+1)) - g^{\mathrm{Top}}_{X,x}(T(3,12l+1)))  = \infty .
\]
\end{rem}

\begin{rem}\label{knot main app related results}
We remark that several related results on \cref{knot main app} can be proved by using several known results. 
For definite 4-manifolds, for example, Ozv\'ath-Szab\'o's $\tau$-invariant and its genus bound (\cite{OS03}) enable us to see 
\[
\lim_{n\to \infty } g_{X, x}(K \#(\#_n T(3,11)))  - g^{\mathrm{Top}}_{X, x}(K \#(\#_n T(3,11)))) = \infty
\]
for negative definite 4-manifolds $X$ with any class $x$. For positive 4-manifolds with any class $x$, we can see 
\[
\lim_{n\to \infty } g_{X, x}(K \#(\#_n T^* (3,11)))  - g^{\mathrm{Top}}_{X, x}(K \#(\#_n T^* (3,11))) = \infty. 
\]
\end{rem}

We remark that, if one can find an embedded surface in a spin 4-manifold which is suitable in view of \cref{main knot}, the 10/8-inequality can be refined.
We record this as a potential approach to the 11/8-conjecture:
 
\begin{cor}\label{11/8-conjecture}
Let $X$ be a simply connected smooth closed spin 4-manifold $X$. 
Suppose that there is a smoothly and properly embedded connected oriented surface $S$ in $X\setminus \operatorname{int}D^4$ such that $\partial S = K$ in $S^3 = \partial (X\setminus \operatorname{int}D^4)$ and 
\[
-\frac{1}{8} \sigma(X) \leq \frac{9}{32}[S]^2-\frac{9}{16}\sigma(K) -\kappa (K). 
\]
Then the 11/8-conjecture for $X$ 
    \[
    -\frac{3}{16} \sigma(X)  \leq b^+(X)
    \]
is true.
\end{cor}



Next, we prove a crossing change formula \cref{crossing change}: 
\begin{proof}[Proof of \cref{crossing change} (i)]
Suppose $K'$ is obtained by a sequence of $n$ crossing changes of $K$. Then, we have an immersed annulus $S'$ in $I\times S^3$ from $K'$ to $K$ with $n$ immersed points. Then, by a certain surgery along these points (\cite[Lemma 9]{S74}), we obtain an embedded genus $n$ cobordism $S$ in $I\times S^3$ from $K$ to $K'$. Then, by applying \cref{main knot} to $S$, we obtain    
 \[
-\frac{9}{16}\sigma(K') + \frac{9}{16}\sigma(K) \leq  n+ \kappa(K')-\kappa(K).
 \]
 On the other hand, $-S$ also gives a genus $n$-cobordism from $K'$ to $K$, we have 
  \[
-\frac{9}{16}\sigma(K) + \frac{9}{16}\sigma(K') \leq  n+ \kappa(K)-\kappa(K').
 \]
 This completes the proof. 
 \end{proof}

\begin{figure}[htb]
\centering
\includegraphics[width=5cm,height=2.3cm]{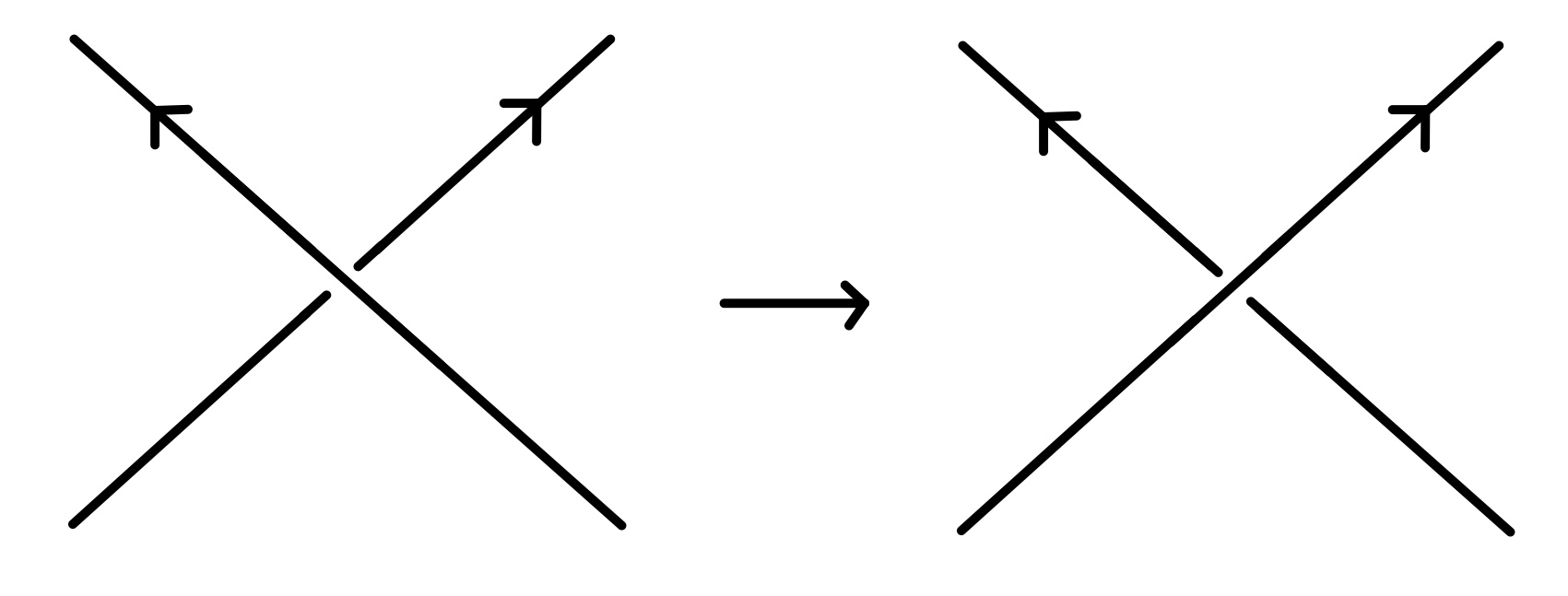}
\caption{Positive crossing change}
\label{Positive crossing change}
\end{figure}

\begin{proof}[Proof of \cref{crossing change} (ii)]
Suppose $K'$ is obtained by a sequence of $n$ positive crossing changes of $K$. Then, we have an immersed annulus $S'$ in $I\times S^3$ from $K'$ to $K$ with $n$ {\it positive} immersed points. By blow-up of these positive immersed points of $S'$, we obtain an embedded annulus $S$ in $(I \times S^3) \# \#_n (-\C P^2)$ from $K$ to $K'$ such that $[S] = (2, \cdots, 2) \in H_2( (I \times S^3) \# \#_n (-\C P^2))$. (For example, see \cite[Page 937]{Kr97} and \cite[Lemma 4.7]{DS20}.)
We apply \cref{main knot} to $-S$ as a cobordism from $K'$ to $K$ in $(I \times S^3) \# \#_n \C P^2$ and obtain
\begin{align*}
& -\frac{1}{8}n + \frac{9}{32}\cdot 4n -\frac{9}{16}\sigma(K) + \frac{9}{16}\sigma(K') \\
&  =  n -\frac{9}{16}\sigma(K) + \frac{9}{16}\sigma(K')   \\  
& \leq n  + \kappa(K)-\kappa(K'). 
\end{align*}
This completes the proof. 
\end{proof}
\begin{rem}\label{full-twist} 
As it is proven in \cite[Lemma 4.7]{DS20}, a full twist from $ K$ to $K'$ around a collection of strands with linking number $d \in \Z_{\geq 0} $ gives an annulus $S$ in $(I \times S^3) \# \#_n \overline{\C P}^2$ from $K$ to $K'$ such that $[S]^2 = -d^2$. So, when $\frac{d}{2}  \in 2\Z +1 $, then one can have a similar full-twist formula using $( (I \times S^3) \# \#_n \overline{\C P}^2, S) $ and $(- (I \times S^3) \# \#_n \overline{\C P}^2, -S) $. 
\end{rem}

In \cref{two bridge kappa12} and \cref{torus knot calculation}, we provided calculations of $\kappa$ for two bridge knots and a certain class of torus knots. Using \cref{crossing change}, we give more examples. 
\begin{theo}\label{more computation}
 Let $K$ and $K'$ be oriented knots in $S^3$.  Suppose there is a positive crossing change from $K$ to $K'$. (Under this assumption, one has $\sigma(K') = \sigma (K)$ or $\sigma(K') = \sigma (K)-2$. )
 \begin{itemize}
     \item[(i)] We also impose $K'$ is a connected sum of two bridge knots. 
     Then, we have 
     \[
     \kappa (K) = \begin{cases}
     -\frac{1}{16} \sigma (K) \text{ or } -\frac{1}{16} \sigma (K)+1 \text{ if } \sigma(K') = \sigma (K) \\
      -\frac{1}{16} \sigma (K)-1 \text{ or } -\frac{1}{16} \sigma (K) \text{ if } \sigma(K') = \sigma (K) -2. 
     \end{cases}
     \]
     \item[(ii)]We also impose $K$ is a connected sum of two bridge knots. 
     Then, we have 
     \[
     \kappa (K') = \begin{cases}
     -\frac{1}{16} \sigma (K') \text{ or } -\frac{1}{16} \sigma (K')-1 \text{ if } \sigma(K') = \sigma (K) \\
      -\frac{1}{16} \sigma (K') \text{ or } -\frac{1}{16}\sigma (K')+1  \text{ if } \sigma(K') = \sigma (K) -2. 
     \end{cases}
     \]
 \end{itemize}
\end{theo}

\begin{proof}[Proof of \cref{more computation}]
Let $K$ and $K'$ be oriented knots in $S^3$.  Suppose there is a positive crossing change from $K$ to $K'$. 
We first suppose $K'$ is a connected sum of two bridge knots. Then, we have \[
\kappa (K') = -\frac{1}{16}\sigma (K'). 
\]
There two cases: $\sigma(K') = \sigma (K)$ and $\sigma(K') = \sigma (K)-2$. We first assume $\sigma(K') = \sigma (K)$. Then, from \cref{crossing change}, we have 
\[
-1 - \frac{9}{16}\sigma (K') + \frac{9}{16}\sigma (K) \leq  \kappa (K')- \kappa (K) \leq  - \frac{9}{16}\sigma (K') + \frac{9}{16}\sigma (K). 
\]
This implies 
\[
1  - \frac{1}{16}\sigma (K) \geq  \kappa (K) \geq    - \frac{1}{16}\sigma (K). 
\]
On the other hand, \cref{main knot} (iii) implies 
\[
\kappa (K) = 
     -\frac{1}{16} \sigma (K) \text{ or } -\frac{1}{16} \sigma (K)+1. 
\]
 When $\sigma(K') = \sigma (K)-2$, a similar computation implies 
\[
 - \frac{1}{16}\sigma (K) \geq  \kappa (K) \geq    - \frac{1}{16}\sigma (K)-1. 
\]
Again, \cref{main knot} (iii) implies 
\[
\kappa (K) = 
     -\frac{1}{16} \sigma (K) \text{ or } -\frac{1}{16} \sigma (K)-1. 
\]
This completes the proof of (i). The proof of (ii) is similar. 

\end{proof}


We give a more strong estimate for $\kappa$ using \cref{most general main theo strong SWF spherical kappa}. 
In order to describe such an inequality, we introduce the following notion: 
\begin{defi}
We say that a knot $K$ in $S^3$ is {\it SWF-spherical} if $(\Sigma (K), \mathfrak{t}, \iota_K)$ is SWF-spherical, whose definition is given in \cref{SWF-spherical}. 
\end{defi}

We will give several examples of SWF-spherical knots. 

\begin{theo}
\label{most general main theo strong SWF spherical kappa knot}
Let $K_1, \cdots , K_n$ be oriented knots in $S^3$. Suppose that $K_1, \cdots , K_n$ are  SWF-spherical. Put $K:= K_1 \# \cdots \# K_n$.
Let $(W,S)$ be a smooth compact oriented cobordism with $H_1(W)=0$ from $(S^3, U)$ to $(S^3, K)$.  We also impose that the homology class $[S]$ of $S$ is divisible by $2$ and $PD(w_2(W)) = [S]/2 \operatorname{mod} 2$.
Set
\[
N(W, S, K, K'):=
-\frac{1}{8}\sigma (W) + \frac{1}{32}[S]^2 + \frac{1}{16}\sigma (K)  - \kappa (K).
\]
Then $N(W, S, K)$ is an integer, and
if we have
\[
N(W, S, K) \geq 2 \quad \text{and} \quad 
 b^+(W)+ g(S) - \frac{1}{4}[S]^2 + \frac{1}{2}\sigma (K) \geq 1,
\] 
then the inequality
\begin{align}
\label{eq: main rel 108 general strong}
\begin{split}
&-\frac{\sigma(W)}{8} + \frac{9}{32}[S]^2-\frac{9}{16}\sigma(K)  +A(N(W, S, K, K')) \\
&\leq b^+(W) + g(S) + \kappa(K)
\end{split}
\end{align}
holds, where $A$ is the function defined in \eqref{A-definition}. 
\end{theo}
\begin{proof}
We just apply \cref{theo: SWF spherical connected sum} to the double branched covers of $S$ as in the proof of \cref{main knot}. 
\end{proof}
For connected sums of two bridge knots and a certain class of torus knots, a stronger inequality than \eqref{ineq1} follows from \cref{most general main theo strong SWF spherical kappa knot}: 
\begin{theo}
\label{most general main theo strong SWF spherical kappa1}
Let $K$ be a knot in $S^3$, which is obtained by a connected sum of knots that appeared in \cref{two bridge kappa12} and \cref{torus knot calculation}. 
Let $X$ be a smooth closed oriented 4-manifold with $H_1(X; \Z )=0$.   Suppose $x$ is a second homology class of $X$, which is divisible by $2$ and $x/2 \equiv PD (w_2(X))$.
Set
\[
N(X, x, K):=
-\frac{1}{8}\sigma (X) + \frac{1}{32}x^2 - \frac{1}{16}\sigma (K)  - \kappa (K).
\]
Then $N(W, x, K)$ is an integer, and
if we have
\[
N(X, x, K) \geq 2 \quad \text{and} \quad 
 b^+(X)+ g_{X,x}(K) - \frac{1}{4}x^2 + \frac{1}{2}\sigma (K)  \geq 1,
\] 
then the inequality
\begin{align}
\label{eq: main rel 108 general strong}
\begin{split}
&-\frac{\sigma(X)}{8} + \frac{9}{32}x^2-\frac{9}{16}\sigma(K)  +A(N(W, x, K)) \\
&\leq b^+(X) + g_{X,x}(K) + \kappa(K)
\end{split}
\end{align}
holds, where $A: \Z \to \{1,2,3\} $ is the function defined by 
\begin{align*}
A(N) = 
\begin{cases}
&1, \quad N=0,2 \mod 8\\
&2, \quad N=1,3,4,5,7 \mod 8\\
&3, \quad N=6 \mod 8.
\end{cases}    
\end{align*}
\end{theo}

\begin{proof}[The proof of \cref{connected sum}]
This is a corollary of \cref{cor: connected sum formula kappa}, \cref{ex: lens} and \cref{theo:Seifert cal most general}. 
\end{proof}

\subsubsection{Results on $\#_N S^2\times S^2$}

It is proven in \cite{Sch10} that for any knot $K$ in $S^3$ whose Arf invariant $\operatorname{Arf}(K)$ is zero, there is a positive integer $N$ such that $K$ is smoothly H-slice in $\#_N S^2\times S^2$.  This result enables us to define  knot concordance invariants 
\[
\mathrm{sn}(K):= \min \{ N | \text{ $K$ is smoothly H-slice in $\#_N S^2\times S^2$} \} 
\]
and 
\[
\mathrm{sn}^{\mathrm{Top}} (K):= \min \{ N | \text{ $K$ is topologically H-slice in $\#_N S^2\times S^2$} \} 
\]
when $\operatorname{Arf}(K)$ is zero. In \cite{CN20}, $\mathrm{sn}(K)$ is defined and called {\it stabilizing number}. As it is asked in \cite[Question 1.4]{CN20}, the problem whether there exists a knot $K$ such that 
\[
0<\mathrm{sn}^{\mathrm{Top}} (K) < \mathrm{sn}(K)
\]
or not was open. We will give an answer. 
Our invariant $\kappa (K)$ can be used to give a lower bound on $s_n(K)$.

\begin{cor}[\cref{H-sliceing number1}]\label{H-sliceing number2}For any knot $ K \subset S^3$ with $\operatorname{Arf}(K) =0$, we have 
    \[
        -\frac{9 }{16}  \sigma (K)  - \kappa (K) \leq  \mathrm{sn}(K).
        \]
Suppose $K$ is obtained by a connected sum of knots appeared in \cref{two bridge kappa12} and \cref{torus knot calculation}. 
If we have
\[
- \frac{1}{16}\sigma (K)  - \kappa (K) \geq 2 \quad \text{and} \quad 
 \mathrm{sn}(K)+ \frac{1}{2}\sigma (K)  \geq 1,
\] 
then the stronger inequality
\begin{align}
\label{stabilizing number stronger ineq}
\begin{split}
&-\frac{9}{16}\sigma(K)   -\kappa (K)+A(- \frac{1}{16}\sigma (K)  - \kappa (K))  \leq \mathrm{sn}(K) 
\end{split}
\end{align}
holds, where $A: \Z \to \{1,2,3\} $ is the function used in \cref{most general main theo strong SWF spherical kappa}.
\end{cor}
\begin{proof}[Proof of \cref{H-sliceing number2}]
We just apply \cref{main knot} to the case $X= \#_N S^2\times S^2$, $[S]=0 \in H_2(S^2\times S^2)$ and $g(S)=0$ and obtain the inequality: 
\[
-\frac{1}{16} \sigma(K) \leq N  + \frac{1}{2} \sigma(K) + \kappa (K) .   
\]
The second statement follows from \cref{most general main theo strong SWF spherical kappa}. 
This completes the proof. 
\end{proof}

\begin{theo}
\label{slicing number different}
For a positive integer $l$ and a positive integer $m$, we have the following estimates:
\begin{align*}
    9 ml  &\leq \mathrm{sn}( \#_m T(3,12l-1)) ,\\ 
    9 ml  &\leq \mathrm{sn}( \#_m T(3,12l+1)), \\
   9ml- 4m &\leq \mathrm{sn}( \#_m T(3,12l-5)),\text{ and }\\ 
   9ml- 5m &  \leq \mathrm{sn}( \#_m T(3,12l-7)).  
\end{align*}
On the other hand, we also have 
\begin{align*}
    &\mathrm{sn}^{\mathrm{Top}} ( \#_{m }T(3,12l-1)) =  8ml, \\ 
     8ml \leq &\mathrm{sn}^{\mathrm{Top}} ( \#_{m }T(3,12l+1)) \leq 8ml +m,  \\ 
     8ml-4m \leq &\mathrm{sn}^{\mathrm{Top}} ( \#_{m }T(3,12l-5)) \leq  8ml -3m, \text{ and } \\
      &\mathrm{sn}^{\mathrm{Top}} ( \#_{m }T(3,12l-7))=  8ml - 4m.
\end{align*}
In particular, if $K$ is one of the following: $T(3,12l-1)$, $T(3,12(l+1)+1)$, $T(3,12(l+1)-5)$, and $T(3,12(l+1)-7)$ for $l>0$, and if we set $K_n = \#_{2n} K$ for $n>0$,
then we have
\[
0 < \mathrm{sn}^{\mathrm{Top}} (K_n)< \mathrm{sn}(K_n)
\]
for all $n$ and the sequence $\{K_n\}_{n=1}^{\infty}$ satisfies that
\[
\lim_{n\to \infty }\left(\mathrm{sn}(K_n) - \mathrm{sn}^{\mathrm{Top}} ( K_n)\right) = +\infty.  
\]
\end{theo}
The estimates for $\mathrm{sn}^{\mathrm{Top}}$ in \cref{slicing number different} follow from results by \cite[Theorem 1]{BBL20}, \cite[Theorem 5.15]{CN20},
and the lower bounds on $sn$ in \cref{slicing number different} shall be deduced from \cref{torus knot calculation,H-sliceing number1}. Also, using \eqref{stabilizing number stronger ineq}, we obtain inequalities which are stronger than the lower bounds  of $sn$ in \cref{slicing number different}. For details, see \cref{more refined ineq for sn}.

\begin{prop}\label{estimate main s2}For an even integer $m$ and a psoitive integer $l$,
\begin{itemize}
    \item $9 ml  \leq \mathrm{sn}( \#_m T(3,12l-1))$,
    \item $9 ml  \leq \mathrm{sn}( \#_m T(3,12l+1))$,
    \item $9ml- 5m   \leq \mathrm{sn}( \#_m T(3,12l-7)) $, and
    \item $9ml- 4m  \leq \mathrm{sn}( \#_m T(3,12l-5)) $. 
\end{itemize}
\end{prop}
\begin{proof}
It is well-known that 
\begin{align}
\label{eq: signature of torus knots}
\sigma ( \#_m T(3,6n-1)) =- 8nm \text{ and } \sigma (\#_m T(3,6n+1))= -8nm    
\end{align}
for positive integers $m$ and $n$.
(For example, see \cite{MO07}. ) 
Putting $n=2l$, it follows from \eqref{eq: signature of torus knots} that
\[
\sigma ( \#_m T(3,12l-1))=\sigma ( \#_m T(3,12l+1)) = -16ml. 
\]
So our inequality in \cref{H-sliceing number1} and \cref{torus knot calculation} implies that
\[
9 ml  \leq \mathrm{sn}( \#_m T(3,12l-1)) \text{ and } 9 ml  \leq \mathrm{sn}( \#_m T(3,12l+1)).
\]
We next put $n=2l-1$. 
Then  it follows from \eqref{eq: signature of torus knots} and \cref{H-sliceing number1} that
\[
\sigma ( \#_m T(3,12l-7))  = \sigma ( \#_m T(3,12l-5)) =- 16ml + 8m. 
\]
So \cref{main knot} and \cref{torus knot calculation} imply that
\[
9 ml - \frac{9}{2}m - \frac{1}{2}m =9ml- 5m  \leq \mathrm{sn}( \#_m T(3,12l-7)) 
\]
and 
\[
9 ml - \frac{9}{2 }m +  \frac{1}{2}m =9ml- 4m  \leq \mathrm{sn}( \#_m T(3,12l-5)) . 
\]
\end{proof}

\begin{lem}
\label{lem: torus Arf}
For $n \geq 1$, we have $\operatorname{Arf}(T(3, 6n\pm1))=0$.
\end{lem}

\begin{proof}
Recall a formula of the Arf invariant in terms of the Alexander polynomial:
for a general knot $K$, we have $\operatorname{Arf}(K) = 0$ if $\Delta_K(-1)=\pm1 \mod 8$.
Since the Alexander polynomial of $T(3,q)$ for $q$ coprime to $3$ is known to be
\[
\Delta_{T(3,q)}(t) = \frac{t^{2q}+t^q+1}{t^2+t+1},
\]
we have $\Delta_{T(3,6n\pm1)}(-1)=1$, and thus obtain $\operatorname{Arf}(T(3, 6n\pm1))=0$.
\end{proof}

\begin{proof}[Proof of \cref{slicing number different}]
From \cref{estimate main s2}, we obtain the desired lower bounds on $sn$, and in what follows we give bounds on $\mathrm{sn}^{\mathrm{Top}}$.
On the other hand, in \cite[Theorem 5.15]{CN20}, we have 
\begin{align}
\label{eq: sn gtop}
\mathrm{sn}^{\mathrm{Top}} (K) \leq g_4^{Top} (K)      
\end{align}
for any knot $K$ with $\operatorname{Arf}(K) =0$.
Since it is proven in \cite[Theorem 1]{BBL20} that, for a positive integer $m$ coprime to $3$, 
\[
g^{\mathrm{Top}}_4 (T(3,m) ) = \left\lceil \frac{2m}{3} \right\rceil , 
\]
we have 
\[
g^{\mathrm{Top}}_4 ( T(3,6n-1))= 4n \text{ and }  g^{\mathrm{Top}}_4 ( T(3,6n+ 1))= 4n+1, 
\]
where $g^{\mathrm{Top}}_4$ is the topological slice genus. 
This proves
\begin{align}
\label{g4top6npm1}
g^{\mathrm{Top}}_4 (\#_m T(3,6n-1))\leq  4nm ,\quad  g^{\mathrm{Top}}_4 (\#_m T(3,6n+ 1)) \leq 4nm + m .     
\end{align}
In particular, 
\[
g^{\mathrm{Top}}_4 ( \#_{m }T(3,12l-1)) \leq 8lm,\quad  g^{\mathrm{Top}}_4 ( \#_{m }T(3,12l+1)) \leq 8lm+ m 
\]
and 
\[
g^{\mathrm{Top}}_4 ( \#_{m }T(3,12l-7)) \leq 8lm - 4m, \quad g^{\mathrm{Top}}_4 ( \#_{m }T(3,12l-5)) \leq 8lm - 3m. 
\]
The desired upper bounds on $\mathrm{sn}^{\mathrm{Top}}$ follows from this combined with \eqref{eq: sn gtop} and
\cref{lem: torus Arf}.

Also, from \cref{000} and, we have 
\[
\mathrm{sn}^{\mathrm{Top}} ( \#_{m }T(3,12l-1)) \geq 8lm \text{, } \mathrm{sn}^{\mathrm{Top}}( \#_{m }T(3,12l+1)) \geq 8lm
\]
and 
\[
\mathrm{sn}^{\mathrm{Top}} ( \#_{m }T(3,12l-7)) \geq 8lm - 4m  \text{, } \mathrm{sn}^{\mathrm{Top}}( \#_{m }T(3,12l-5)) \geq 8lm - 4m. 
\]
This completes the proof. 
\end{proof}

\begin{rem}[Comparison with other methods]\label{3 methods1}
We compare \cref{H-sliceing number1} with Maloescu's relative 10/8-inequality, obstructions from Arf invariant and signature function.
\begin{enumerate}
\item Manolescu's relative 10/8 inequality \cite[Theorem 1.1]{Ma14} enable us to prove that 
for any knot $K$, 
\begin{align}\label{111}
    -\frac{5}{8} \sigma (K)- \kappa (\Sigma (K), \frakt)    \leq 2 \mathrm{sn}(K) .  
    \end{align}

For a connected sum of two bridge knots $K$ in $S^3$, we have 
   \begin{align}\label{11}
    -\frac{1}{8} \sigma (K)  \leq  2 \mathrm{sn}(K) ,  
    \end{align}
where $\kappa^*(K)$ is the Manolescu's kappa invariant $\kappa (\Sigma (K), \frkt)$ for the double branched cover.
For more details, see \cref{manolescu}.

           \item 
    From an obstruction via Arf invariant (For more details see \cref{arf inv}), we have  
       \begin{align}\label{21}
    \operatorname{Arf} (K)\equiv 0 \operatorname{mod} 2.
    \end{align}
    
    \item From an obstruction via the Tristram-Levine signature $\sigma_K(t)$ (For more details see \cref{000}), we have  
    \[
    \frac{1}{2} \max_m \max_{r \in \{1, \cdots, m-1\}} | \sigma_K (e^{2\pi r i m } )  | \leq \mathrm{sn}(K ) 
    \]
\end{enumerate}
\end{rem}

Using \eqref{stabilizing number stronger ineq}, we can give lower bounds for $sn$. 

\begin{cor}\label{more refined ineq for sn}
For a positive integer $l$ and a positive integer $m$, we have the following estimates:
\begin{align*}
    9 ml  + A ( ml ) &\leq \mathrm{sn}( \#_m T(3,12l-1)) ,\\ 
    9 ml + A ( ml ) &\leq \mathrm{sn}( \#_m T(3,12l+1)), \\
   9ml- 4m + A ( ml )&\leq \mathrm{sn}( \#_m T(3,12l-5))\  (l \geq 2),\text{ and }\\ 
   9ml- 5m + A ( ml ) &  \leq \mathrm{sn}( \#_m T(3,12l-7))\  (l \geq 2) .  
\end{align*}
\end{cor}
 
 \begin{proof}
Put $N(K):= - \frac{1}{16}\sigma (K)  - \kappa (K)$
To apply \eqref{stabilizing number stronger ineq}, we need to confirm that 
$N(K) \geq 2$ and $\mathrm{sn}(K)+ \frac{1}{2}\sigma (K) \geq 1$. 
\begin{itemize}
    \item Suppose $K_{m,l} := \#_m T(3,12l-1)$. Then $\sigma (\#_m T(3,12l-1))= -16ml$ and $\kappa (K_{m,l}) = 0$. Our inequality implied $9 ml  \leq \mathrm{sn}( \#_m T(3,12l-1))$. So, the condition $\mathrm{sn}(K)+ \frac{1}{2}\sigma (K) \geq 1$ is satisfied. We also have $N(K_{m,l} )= ml \geq 2$. Thus, \eqref{stabilizing number stronger ineq} implies 
    \[
    9 ml  + A ( ml ) \leq \mathrm{sn}( \#_m T(3,12l-1)). 
    \]
      \item Suppose $K_{m,l} := \#_m T(3,12l+1)$. Then $\sigma (\#_m T(3,12l+1))= -16ml$ and $\kappa (K_{m,l}) = 0$. Our inequality implied $9 ml  \leq \mathrm{sn}( \#_m T(3,12l+1))$. So, the condition $\mathrm{sn}(K)+ \frac{1}{2}\sigma (K) \geq 1$ is satisfied. We also have $N(K_{m,l} )= ml \geq 2$. Thus, \eqref{stabilizing number stronger ineq} implies 
    \[
    9 ml  + A ( ml ) \leq \mathrm{sn}( \#_m T(3,12l+1)). 
    \]
      \item Suppose $K_{m,l} := \#_m T(3,12l-5)$. Then $\sigma (\#_m T(3,12l+1))= -16ml + 8m $ and $\kappa (K_{m,l}) = - \frac{m}{2}$. Our inequality implied $9 ml  -4m \leq \mathrm{sn}( \#_m T(3,12l+1))$. So, the condition $\mathrm{sn}(K)+ \frac{1}{2}\sigma (K) \geq 1$ is satisfied if $l \geq 2$. We also have $N(K_{m,l} )= ml \geq 2$. Thus, \eqref{stabilizing number stronger ineq} implies 
    \[
    9 ml  + A ( ml ) \leq \mathrm{sn}( \#_m T(3,12l+1)). 
    \]
      \item Suppose $K_{m,l} := \#_m T(3,12l-7)$. Then $\sigma (\#_m T(3,12l+1))= -16ml + 8m$ and $\kappa (K_{m,l}) = \frac{m}{2}$. Our inequality implied $9 ml -5m \leq \mathrm{sn}( \#_m T(3,12l+1))$. So, the condition $\mathrm{sn}(K)+ \frac{1}{2}\sigma (K) \geq 1$ is satisfied if $l \geq 2$. We also have $N(K_{m,l} )= ml \geq 2$. Thus, \eqref{stabilizing number stronger ineq} implies 
    \[
    9 ml  + A ( ml ) \leq \mathrm{sn}( \#_m T(3,12l+1)). 
    \]
\end{itemize}
 \end{proof}

 More generally, we can also prove the following result on connected sums. 
\begin{theo}
 Let $K$ be a knot in $S^3$ with $\operatorname{Arf}(K)=0$. Then, we have 
 \[
  \lim_{n\to \infty} \left( \mathrm{sn}(K \#_{n}  T(3,11)  ) -  \mathrm{sn}^{\mathrm{Top}} (K \#_{n}  T(3,11)  ) \right) = \infty . 
 \]
 \end{theo}
 \begin{proof}
 From \cref{connected sum}, one can compute 
 \[
 \kappa (K \#_{n}  T(3,11) ) = \kappa (K). 
 \]
 On the other hand, \cref{H-sliceing number2} implies that 
    \[
        -\frac{9 }{16}  \sigma (K ) + 9n  - \kappa (K) \leq  \mathrm{sn}(K\#_{n}  T(3,11)  ).
        \]
    One can see 
    \[
    \mathrm{sn}^{\mathrm{Top}} (K \#_{n}  T(3,11)  ) \leq  \mathrm{sn}^{\mathrm{Top}} (K ) +\mathrm{sn}^{\mathrm{Top}} (K ) (  \#_{n}  T(3,11)  ) = \mathrm{sn}^{\mathrm{Top}} (K ) + 8n. 
    \]
    These inequalities imply the conclusion. 
 \end{proof}

\subsubsection{Results on $\#_n K3$} 

Next we focus on $\#_n K3$. 
Our invariant $\kappa (K)$ can be used to give a lower bound on $g_{\#_n K3}$:

\begin{theo} \label{H-sliceing numberK3}Let $n$ be a non-negative ingeter. 
Let $K$ be a connected sum of two bridge knots with $\operatorname{Arf}(K) =0$. 
Then we have 
    \[
   - \frac{1}{2}\sigma (K)-n  \leq  g_{\#_n K3,0} (K).
        \]
     Also, the same inequality holds for any homotopy $\#_n K3$. 
\end{theo}


We first give a proof of \cref{H-sliceing numberK3}. 
 \begin{proof}[Proof of \cref{H-sliceing numberK3}]
We just apply \cref{main knot} to the case $X= \#_n K3$, $[S]=0 \in H_2(\#_n K3)$ and obtain the inequality: 
\[
2n  -\frac{1}{16} \sigma(K) \leq 3n  + \frac{1}{2} \sigma(K) + \kappa (K) + g(S).   
\]
By combining this with \eqref{two bridge kappa}, we obtain 
\[
-\frac{1}{2} \sigma(K)   \leq n  + g(S). 
\]
This completes the proof. 
\end{proof}

\begin{rem}[Comparison with other methods]
\label{3 methods}
We compare \cref{main knot} with Maloescu's relative 10/8-inequality, obstructions from Arf invariant and signature function. We review these methods in \cref{Comparison}. 
\begin{enumerate}
\item Manolescu's relative 10/8 inequality \cite{Ma14} enable us to prove that 
for any connected sum of two bridge knots $K$ in $S^3$, we have 
   \begin{align}\label{1}
   -\frac{1}{2} \sigma (K)  \leq 2 n + g(S) . 
    \end{align}
For more details, see \cref{manolescu}.

           \item 
    From an obstruction via Arf invariant (For more details see \cref{arf inv}), we have  
       \begin{align}\label{2}
    \operatorname{Arf} (K)\equiv 0 \operatorname{mod} 2.
    \end{align}
    
    \item From an obstruction via the Tristram-Levine signature $\sigma_K(t)$ (For more details see \cref{000}), we have  
  \begin{align}\label{3}
    \max_m  \max_{r \in \{1, \cdots, m-1\}} \frac{1}{2}| \sigma_K (e^{2\pi r i m } ) -16n | \leq 11n +g .
    \end{align}
\end{enumerate}
\end{rem}

\begin{ex}\label{ex1}
 Let $K_m$ be the $m$-fold connected sum of $5_2$ for $m \in \Z_{>0}$. 
 Note that $5_2$ is two bridge knot $K(7,3)$ and $\operatorname{Arf}(5_2)=0$. 
 Then $\sigma(K_m)= -2m$. Suppose $S$ is a properly embedded oriented surface in $\#_n K3 \setminus \operatorname{int} D^4$ bounded by $K$ such that $[S]=0$. 
 
 Then so our inequality \cref{H-sliceing numberK3} implies 
 \[
m \leq n + g(S)
 \]
 holds. 
 On the other hand, \eqref{1} implies
 \[
m \leq 2n + g(S). 
 \]
 Also \eqref{3} implies 
 \[
   \max\{ 8n  ,  | -\frac{1}{2}m -8n  |,  |-m -8n  |\} \leq 11n +g(S).
 \]
\end{ex}

Note that a family of topologically H-slice but not smoothly H-slice knots in the punctured $\#_3 K3$ are given in \cite{IMT21}. However, the Bauer-Furuta type invariant used in \cite{IMT21} vanishes for $\#_n K3$ when $n\geq 4$. On the other hand, our invariant $\kappa(K)$ may be used to give such examples. 
\begin{problem} Let $n$ be a positive integer with $n \geq 4$.
Is there a topologically H-slice but not smoothly H-slice knot in the punctured $\#_n K3$ ? 
\end{problem}

\subsubsection{Results on $\#_n\C P^2 \#_m \C P^2$} 

We also consider $X := \#_n \C P^2 \#_m (-\C P^2)$ for positive integers $n$ and $ m$.

When $\min\{n,m\} \geq 2$, since known gauge theoretic invariant of $X$ vanishes, there is no way to obtain adjunction type inequality. 

\begin{theo}
\label{cp-genus}
Let $X$ be a homotopy $\#_n \C P^2 \#_m (-\C P^2)$ for a pair $(m,n)$ of non-negative integers. 
Let $H$ be an element in $H_2(X; \Z)$ such that 
\[
H = \sum_{1\leq i\leq n } a_i e^+ _i+ \sum_{1\leq j\leq n } b_j e^- _j 
\]
for integers $a_i, b_j$ with $a_i \equiv 2 \mod 4$ and $b_j \equiv 2 \mod 4$, where $\{e^+_i\}$ and $\{e^-_i\} $ are generators corresponding to $H_2( \C P^2)$ and $H_2( -\C P^2)$ for each summand. 
Then for any knot $ K$ in $S^3$, we have 
  \begin{align} \label{genus boundCP2}
    \frac{1}{32} \left( -36n+  4m + 9\left(\sum_{i=1}^n a_i^2 -  \sum_{j=1}^m b_j^2 \right) -18\sigma(K) -32 \kappa(K) \right) \leq g_{X, H}(K).  
    \end{align}
    In particular, when $K$ is a connected sum of two bridge knots, one has 
    \[
    \frac{1}{32} \left( -36n+  4m + 9\left(\sum_{i=1}^n a_i^2 -  \sum_{j=1}^m b_j^2 \right) -16\sigma(K)  \right) \leq g_{X, H}(K).
    \]
\end{theo}
This is just a corollary of \cref{main knot}.


\begin{rem}Under the same assumptions of \cref{cp-genus}, the following facts hold: 
\begin{enumerate}
\item Manolescu's relative 10/8 inequality \cite{Ma14}(see \cref{manolescu1}) enable us to prove that  implies 
\[
\frac{1}{16}\left( -36n + 4m  + 5\left(\sum_{i=1}^n a_i^2 -  \sum_{j=1}^m b_j^2 \right) -10\sigma (K) - 16\kappa^* (K)\right)  \leq  g_{X,H}(S). 
\]
In particular, when $K$ is a connected sum of two bridge knots, one has 
 \[
\frac{1}{16}\left( -36n + 4m  + 5\left(\sum_{i=1}^n a_i^2 -  \sum_{j=1}^m b_j^2 \right) -8\sigma (K) \right)  \leq  g_{X,H}(S). 
\]
    \item 
From an obstruction via Arf invariant (For more details see \cref{arf inv}), we have 
\[
0  = \operatorname{Arf}(K) + \operatorname{Arf}(X, S). 
\]

\item 
Since $H$ is divisible by $2$, by \cref{000}, one has
\begin{align}\label{CP2topological}
\frac{1}{2} \left(\left| \sigma (K) + n-m  + \sum_{i=1}^n a_i^2 -  \sum_{j=1}^m b_j^2 \right|-n-m \right)    \leq g^{\operatorname{Top}}_{X, H}(K)
\end{align}
\end{enumerate}
\end{rem}

\begin{theo}\label{73exampleCP2} \label{37CP2}
 Let $K_m$ be the $m$-fold connected sum of $ T(3,7)$ for a positive integer $m$ and $X$ be $\#_m\C P^2 \#_m(-\C P^2 )$.  
 Then, 
 \[
 \frac{7}{4}m  \leq g_{ ( X, x)  } (K_m), 
 \]
 where $x=(2, \cdots, 2) \in H_2( \#_m\C P^2 \#_m(-\C P^2 )) $. 
\end{theo}
\begin{proof}
We apply \eqref{genus boundCP2} to $X$, $K_m$, and $[S]= x$ and obtain
\begin{align*}
\frac{1}{32} \left( -36m+  4m  -18 (-4m) -32 (-\frac{1}{2}m) \right) \leq g_{X, x}(K_m) \\ 
\frac{7}{4} m \leq g_{X, x}(K_m). 
\end{align*}
\end{proof}

\begin{rem}\label{CP2topological}
We again compare \cref{37CP2} with the topological obstruction coming from the knot signature. Under the same assumptions in \cref{37CP2}, the inequality \eqref{CP2topological} implies  
\[
\frac{1}{2} \left(\left| -4m    + m - m + 0 \right|-2m \right) =m   \leq g^{\operatorname{Top}}_{X, x}(K_m). 
\]
\end{rem}

 



\section{Non-extendable and non-smoothable actions}
\label{section Applications to non-extendable actions}

We apply \cref{main theo} to obtain obstructions to an extension of involutions on 3-manifolds to spin 4-manifolds.
This also gives a series of examples of non-smoothable actions on 4-manifolds with boundary.
First, we give the proof of \cref{theo intro nonextendable general}.

\begin{proof}[Proof of \cref{theo intro nonextendable general}]
By \cref{theo:Seifert cal most general},
we have 
\[
\kappa(Y, \iota) = -\bar{\mu}(Y)/2  \quad \text{and}\quad \kappa(-Y, \iota)=\bar{\mu}(Y)/2.
\]
Let $W$ be a smooth oriented compact spin 4-manifold with $b_1(W)=0$ bounded by $Y$.
If $\iota$ on $Y$ extends to $W$ as a smooth involution preserving the orientation and the spin structure, it follows from \cref{main theo} that
\begin{align*}
-\frac{\sigma(W)}{16} 
\leq b^+(W)-b^+_{\iota}(W) - \bar{\mu}(Y)/2.
\end{align*}
The statement (i) of the \lcnamecref{theo intro nonextendable general} follows from this.

Now suppose that $\iota$ extends to $W$ as a homologically trivial involution.
Then we have
\[
-\frac{\sigma(W)}{16} 
\leq - \bar{\mu}(Y)/2 \quad \text{and}\quad
-\frac{\sigma(-W)}{16} 
\leq \bar{\mu}(Y)/2.
\]
These imply $\sigma(W)=8\bar{\mu}(Y)$.
To prove the statement (ii) of the \lcnamecref{theo intro nonextendable general}, it remains to show only that $\iota$ extends to $W$ as a diffeomorphism.
Note that the $\Z_2$ action on $Y$ via $\iota$ is given as a part of the standard $S^1$ action on the Seifert 3-manifold $Y$.
Namely, this $\Z_2$ action factors through the inclusion $\Z_2 \inc S^1$ and the standard $S^1$ action on $Y$.
Therefore $\iota$ is smoothly isotopic to the identity map on $Y$.
Hence $\iota$ extends to every 4-manifold bounding $Y$ as a homologically trivial diffeomorphism.
\end{proof}

\begin{rem}
One can deduce constraints on locally linear topological involutions by the $G$-signature theorem.
It is summarized in \cref{cor: G signature connstraint on involution}.
Note that this topological constraint involves the the self-intersection number and genus of the surface $S$ obtained as the fixed point, also involves the signature of the knot $K = S \cap Y$.
But the constraint on smooth involutions obtained in \cref{theo intro nonextendable general} is free from these data of fixed-point sets.
\end{rem}

\begin{rem}
\label{rem: general theo: nonext 2313}
As mentioned in \cref{rem: general involusion not minus}, for all odd involutions on $\Sigma(a_1, \ldots, a_n)$ which are isotopic to the identity, the kappa invariants are just the same as that of the above $\iota$.
Therefore the statement of \cref{theo intro nonextendable general} holds for all odd involutions on $\Sigma(a_1, \ldots, a_n)$ which are isotopic to the identity.
\end{rem}

\Cref{theo intro nonextendable general} can be generalized to connected sums of Seifert homology spheres:

\begin{theo}
\label{theo: nonext connecedt sum Seifert cal most general}
For $N \geq 0$ and $i \in \{1, \ldots, N\}$, let $(Y_i, \iota_i)$ be the following triple:
$Y_i=\pm \Sigma(a_{1}, \ldots, a_{n})$ for some coprime natural numbers $a_{1}, \ldots, a_{n}$ with $a_1$ even, and $\iota_i : Y_i \to Y_i$ is given by $\iota_i(z_1,z_2, \ldots,z_n)=(-z_1,z_2,\ldots,z_n)$.
Let $o(\iota_i)$ be an orientation of the fixed-point set of $\iota_i$.
Form the equivariant connected sum
\[
(Y,\frakt,\iota) = \#_{i=1}^{N} (Y_i,\frakt_i,\iota_i,o(\iota_i))
\]
along fixed points, and let $W$ be a compact smooth oriented spin 4-manifold bounded by $Y$ with $b_1(W)=0$.
Then we have the following:
\begin{itemize}
\item [(i)] The involution $\iota$ cannot extend to $W$ as a smooth involution so that
\begin{align*}
-\frac{\sigma(W)}{16} 
> b^+(W)-b^+_{\iota}(W) - \frac{1}{2}\sum_{i=1}^{N}\bar{\mu}(Y_i).
\end{align*}
\item[(ii)] Suppose that $\sigma(W) \neq 8\sum_{i=1}^{N}\bar{\mu}(Y_i)$.
Then $\iota$ cannot extend to $W$ as a homologically trivial smooth involution, while $\iota$ can extend to $W$ as a homologically trivial diffeomorphism.
\end{itemize}
\end{theo}

\begin{proof}
By \cref{theo:kappa connecedt sum Seifert cal most general}, we have $\kappa(Y,\frakt,\iota)
= - \frac{1}{2}\sum_{i=1}^{N}\bar{\mu}(Y_i)$.
Then the remaining proof is the same as the proof of \cref{theo intro nonextendable general}.
\end{proof}

We note a consequence of \cref{theo intro nonextendable general} for (relatively) small 4-manifolds.
A preliminary result is:

\begin{prop}
\label{prop: KT constraint}
For $k>0$,
let $W$ be a compact spin smooth 4-manifold bounded by one of $\Sigma(2,3,12k+1)$, $-\Sigma(2,3,12k+1)$, and $-\Sigma(2,3,12k-1)$ with the intersection form isomorphic to that of $K3$.
Let $f : W \to W$ be an orientation-preserving diffeomorphism whose restriction to the boundary is isotopic to the identity.
Suppose that $f$ lifts to a spin automorphism.
(If $W$ is simply-connected, this is the case for all $f$.)
Then $f$ preserves the orientation of $H^+(W)$.
\end{prop}

\begin{proof}
Set $Y=\Sigma(2,3,13)$.
Let $f'$ be a diffeomorphism of $W$ which is isotopic to $f$ and is the identity on $Y$.
Since $f$ lifts to a spin automorphism, so does $f'$.
Let $\tilde{f}'$ be a lift of $f'$ on the spin structure.
Then, as the mapping torus of $W$ by $(f', \tilde{f}')$,
we obtain a fiber bundle $E_W \to S^1$ over a circle with fiber $W$ equipped with a fiberwise spin structure.
Associated to this fiber bundle,
we have a real vector bundle $H^+(E_W) \to S^1$, whose fiber is a positive-definite subspace of the real second cohomology group of the fiber. 
If $f$ reverses the orientation of $H^+(W)$, this implies that $w_1(H^+(E_W)) \neq 0$.

Recall the comuputation of Manolescu's $\alpha$ invariant of $Y$
(see \cite[Proposition~3.8, Page 172]{Ma16}):
\begin{align*}
&\alpha(\Sigma(2,3,12k+1))=0,\\
&\alpha(-\Sigma(2,3,12k+1)) = -\gamma(\alpha(\Sigma(2,3,12k+1)))=0,\\ 
&\alpha(-\Sigma(2,3,12k-1)) = -\gamma(\alpha(\Sigma(2,3,12k-1)))=0.
\end{align*}
Thus in any case we have $\alpha(Y)=0$.
Then it follows from \cite[Theorem~1.2]{KT20a} that
\[
-\sigma(W)/8 \leq \alpha(Y)=0,
\]
but since $\sigma(W)=-16$, this is a contradiction.
\end{proof}

Bryan proved in \cite[Theorem 1.8]{Bry98} that, for a spin odd involution $\iota$ on a spin rational cohomology $K3$ surface $X$, we have $b^+_{\iota}(X) = 1$.
If $f$ in \cref{prop: KT constraint} is an involution $\iota$, then \cref{prop: KT constraint} implies that $b^+_\iota(W)=1$ or $b^+_\iota(W)=3$.
But if $\iota$ is of odd type, there is an additional constraint similar to the result by Bryan: 

\begin{cor}
\label{eq:Bryan-like constraint}
let $W$ be a compact spin smooth 4-manifold bounded by one of $\Sigma(2,3,12k+1)$, $-\Sigma(2,3,12k+1)$, and $-\Sigma(2,3,12k-1)$ with the intersection form isomorphic to that of $K3$.
Let $\iota : W \to W$ be a smooth spin involution of odd type whose restriction to the boundary is isotopic to the identity.
Then we have $b^+_\iota(W)=1$.
\end{cor}

\begin{proof}
By \cref{prop: KT constraint}, we have $b^+_\iota(W)=1$ or $b^+_\iota(W)=3$.
But the possibility that $b^+_\iota(W)=3$ is excluded by \cref{theo intro nonextendable general}.
\end{proof}

\begin{rem}\label{a new action on s2s2}
Before proving \cref{theo: non-smoothable actions main thm 6nminus1}, let us explain the action $\iota_t$ on $\#_2 S^2\times S^2$. First, we consider the complex conjugation $\iota_\C$ on $D^4$. By pasting two copies of $(D^4, \iota_\C) $ and $(-D^4, \iota_\C) $, we obtain a $\Z_2$-action $\iota_{S^4}$ on $S^4$ whose fixed-point set $S$ is a trivial 2-knot in $S^4$. Next, we fix an element $x_0$ in $S^4 \setminus S$.  
Then by considering an equivariant connected sum of $(S^4, \iota_{S^4} )$ and $(\#_2S^2 \times S^2, \iota_{\#_2S^2 \times S^2} :=\text{trivial double cover})$ along $x_0$ and $\iota_{S^4}(x_0)$ gives an involution 
\[
\iota_t : \#_2S^2 \times S^2 \to \#_2S^2 \times S^2. 
\]
Note that the quotient orbifold of $\iota_t$ is $S^2 \times S^2$ and the fixed-point set is a trivial 2-knot in $S^2 \times S^2$. 
\end{rem}

Now we give the proof of \cref{theo: intro non-smoothable actions main thm 6nminus1}, which yields examples of non-smoothable involutions.
In fact, we can prove a more general result as follows:

\begin{theo}
\label{theo: non-smoothable actions main thm 6nminus1}
Let $W$ be one of:
\begin{itemize}
\item[(i)]
$W = \natural _m M(2,3,6n-1)\#_{2nm+m} S^2 \times S^2$ for $n\geq 2, m \geq 1$, with boundary $\#_m \Sigma(2,3,6n-1)$.

    \item[(ii)]    
$W = \natural_m M(2,3,6n+1)\#_{2nm+2m} S^2 \times S^2$ for $n\geq 1, m \geq 1$, with boundary $\#_m \Sigma(2,3,6n+1)$.    
\end{itemize}
Then there exists a locally linear topological involution $\iota_W : W \to W$ with non-empty fixed-point set that satisfies the following properties:
\begin{itemize}
\item[(I)] The involution $\iota_W$ is not smooth with respect to every smooth structure on $W$.

\item[(II)] The restriction of $\iota_W$ to the boundary $\del{W}$ is $\#_m \iota$, where $\iota : \Sigma(2,3,6n\pm1) \to \Sigma(2,3,6n\pm1)$ is the involution $\iota(z_1, z_2, z_3) = (-z_1, z_2, z_3)$.
In particular, $\iota_W|_{\del W}$ extends as a diffeomorphism of $W$ for every smooth structure on $W$.

\item[(III)] For any $N>1$, the equivariant connected sum 
\[
\iota_W \#_N \iota_r : W \#_N S^2 \times S^2 \to W \#_N S^2 \times S^2
\]
along fixed points is also a non-smoothable involution with respect to every smooth structure on $W \#_N S^2 \times S^2$.

\item[(IV)] The quotient orbifold $W/\iota_W$ is indefinite.
More precisely, $b^+(W/\iota_W) = 4nm$, $b^-(W/\iota_W) = 4nm$ for the case (i), and $b^+(W/\iota_W) = b^-(W/\iota_W) = 4nm+m$ for the case (ii). 

\end{itemize}
\end{theo}

\begin{proof}
Mainly we describe the case (i): the proof for the case (ii) is similar.
First, let us describe how to construct $\iota_W$.
By \eqref{eq: sn gtop} and \eqref{g4top6npm1},
we have that $\mathrm{sn}^{\mathrm{Top}}(\#_m T(3,6n-1)) \leq 4nm$.
From this, there exists a locally flat topological proper embedding of an oriented compact genus-0 surface $S$ with boundary into $\#_{4nm}S^2 \times S^2$ so that $\del S=\#_m T(3,6n-1)$.
The double branched cover $\Sigma(S)$ is topological spin 4-manifold boundeded by $\#_{4nm}\Sigma(2,3,6n-1)$.
The calculation of Betti numbers and signatures of double branched covering given in \cref{branched spin} hold also for locally flat topological embeddings, and we have $b^+(\Sigma(S))=4nm, \sigma(\Sigma(S))=-8nm$.
Noting that $S$ is a surface with genus $0$, one can check that $\Sigma(S)$ is simply-connected by the van-Kampen theorem.
On the other hand, the intersection form of $M(2, 3,6n-1)$ is well-known, which is $n(-E_8) \oplus (2n-1)H$, where $H$ is the intersection form of $S^2 \times S^2$. 
Therefore the intersection form of $\Sigma(S)$ is isomorphic to that of $W$.
It follows from this and \cref{theo: Freedman boundary} that $\Sigma(S)$ is homeomorphic to $W$.
Let $\iota_W$ be the involution on $W$ which is induced from the covering invokution on $W$ via a homeomorhism between $\Sigma(S)$ and $W$.
We claim that this involution $\iota_W$ is the desired involution. 

First, it is clear that $\iota_W|_{\del W}$ is $\#_{nm}\iota$ by construction.
Also, since the quotient $W/\iota_W$ is homeomorphic to 
$\#_{4nm} S^2 \times S^2$, it follows that $b^+(W/\iota_W) = b^-(W/\iota_W) = 4nm$.

Next, we show that $\iota_W$ is non-smoothable for every smooth structure on $W$.
We have that $b^+(W)-b^+_{\iota_W}(W) = 4nm-4nm=0, -\sign(W)/16=nm/2$.
Also, it follows from \cref{cor: kappa for torus knot 6npm1} that  $\kappa(\del W, \#_{nm}\iota) = \kappa(\#_{nm} T(3,6n-1)) = 0$ for $n$ even, and $\kappa(\del W, \#_{nm}\iota)=m$ for $n$ odd.
Thus, for $n>1$, it follows from that $\iota_W$ cannot be a smooth involution by \cref{main theo} for every smooth structure.
Also, since $\iota_r$ is homologically trivial, we have $b^+(S^2\times S^2) - b^+_{\iota_r}(S^2 \times S^2)=0$.
Thus $\iota_W \#_N \iota_r$ is also non-smoothable for $N>0$ as well.

The proof for the case (ii) is quite similar to that for the case (i).
First, by the upper bound on $\mathrm{sn}^{\mathrm{Top}}$ from \eqref{eq: sn gtop} and \eqref{g4top6npm1}, we can find a genus-0 surface $S$ in $\#_{4nm+m}S^2 \times S^2$ with $\del S=\#_m T(3,6n+1)$.
Recalling that the intersection form of $M(2,3,6n+1)$ is given by $n(-E_8)\oplus 2n H$, it is easy to see that the double brached covering $\Sigma(S)$ is homeomorphic to $W$.
The remaining argument is the same with the case (i).
\end{proof}

Here we summarize a fact used in the proof of \cref{theo: non-smoothable actions main thm 6nminus1} on the topological classification of topological 4-manifold with boundary.

\begin{theo}[\cite{B86,B93}]
\label{theo: Freedman boundary}
Let $Y$ be an integral homology $3$-sphere.
Then the homeomorphism class of a simply-connected compact topological 4-manifolds with boundary $Y$ with even intersection form is determined by its intersection forms, up to isomorphism over $\Z$. 
\end{theo}

\section{Appendix} \label{Comparison}
In this section, we review several known genus bounds related to our result.

 \subsection{Manolescu's relative 10/8-inequality}
First,  we compare our main result and 10/8-inequality without involutions.  
 A similar discussion for closed 4-manifolds can be founded in \cite{Ha13}. We consider Manolescu's relative 10/8-inequality \cite{Ma14}. 
 The following theorem is just a corollary of usual relative 10/8-inequality combined with double branched covers. 
 
 In \cite{Ma14}, Manolescu introduced a rational valued spin rational homology cobordism invariant \[
 (Y, \fraks ) \mapsto \kappa (Y, \frakt ) \in \frac{1}{8}\Z .  
 \]
 This enables us to define a knot concordance invariant by 
 \[
 K \mapsto \kappa^*(K):=  \kappa (\Sigma (K) , \frakt ) \in \frac{1}{8}\Z, 
 \]
 where $\frakt$ is the unique spin structure on the double branched cover $\Sigma (K)$. The following is an application of the relative 10/8-inequality proven in \cite{Ma14} and \cref{branched spin}. 
 \begin{theo}[\cite{Ma14}]\label{manolescu1}
The invariant $\kappa^*  (K) $ is a knot concordance invariant satisfying the following property: 
For any compact 4-manifold $X$ bounded by $S^3$ with $H_1(X;\Z)=0$ and any oriented compact surface $S$ bounded by $K$ such that $PD(w_2(X)) = [S]/2 \operatorname{mod} 2$, we have 
    \begin{align}\label{ineq}
    -\frac{1}{4}\sigma (X) + \frac{5}{16 }[S]^2 - \frac{5}{8}\sigma (K)  - \kappa^* (K)   \leq 2 b^+(X) + g(S) 
        \end{align}
    holds, where $[S]/2$ denotes the element in $H_2(S; \Z)$ such that $2 ([S]/2) = [S]$.
\end{theo}
\begin{rem}
When $K$ is the unknot and $S$ is an embedded null-homologous disk, \eqref{ineq} implies 
\[
 -\frac{\sigma(X)}{8}   \leq b^+(X), 
 \]
 which recovers Furuta's original 10/8-inequality (\cite{Fu01}) except for adding 1 or 2 on the left hand side as in the case of \cref{main knot}. 
\end{rem}
\begin{proof}[Proof of \cref{manolescu1}] 
Consider the double branched cover $\Sigma(S)$ and apply \cite[Theorem 1.1 and Remark 4.6]{Ma14} to $\Sigma(S)$. Then, we have the following inequality 
\begin{align}\label{10/8 usual}
-\frac{\sigma (\Sigma(S)) }{8} \leq b^+(\Sigma(S)) + \kappa^* (K). 
\end{align}
 Moreover, using \cref{branched spin}, we obtain the equalities
\begin{align*}
&\sigma (\Sigma (S ))= 2 \sigma (X) - \frac{1}{2} [S]^2 + \sigma (K) \  \text{ and }\\ 
& b^+ (\Sigma (S )) = 2b^+(X) + g(S) -\frac{1}{4} [S]^2  +  \frac{1}{2} \sigma (K). \\
 \end{align*}
 Combining these two equations with \eqref{10/8 usual}, we complete the proof. 
\end{proof}
 Also, for spin rational homology 3-spheres with positive scalar curvature metrics, 
\[
\kappa (Y, \frkt)  = \delta (Y, \frkt )
\]
holds (\cite[Subsection 5.1]{Ma14}). In \cite[Theorem 1.2]{MO07}, it is proved that for any connected sum of two bridge knots, 
\[
\delta( \Sigma(K) , \frkt ) =\frac{1}{2}d (\Sigma(K) , \frkt ) =  \delta (K)  = - \frac{1}{8} \sigma (K)
\]
holds, where  $d (\Sigma(K) , \frkt )$  is the Heegaard Floer correction term. 
Since 
\[
\kappa ( \Sigma(K) , \frkt ) = \delta ( \Sigma(K) ,\frkt ), 
\]
for a connected sum of two bridge knots, we have
\[
 \kappa^* (K) = \kappa( \Sigma(K) , \frkt )= -  \frac{1}{8} \sigma (K).
 \]
 As a corollary of \cref{manolescu1}, we see 
\begin{cor} \label{manolescu} Let $K$ be a connected sum of two birdge knots. 
For any compact 4-manifold $X$ bounded by $S^3$ with $H_1(X; \Z)=0$ and any oriented compact surface $S$ bounded by $K$ such that $[S]$ is divisible by $2$ and $PD(w_2(X)) = [S]/2 \operatorname{mod} 2$, we have 
    \[
-\frac{1}{4}  \sigma (X) +  \frac{5}{16} [S]^2   -  \frac{1}{2} \sigma (K)  \leq   2b^+(X) + g(S) . 
    \]
\end{cor}
 
 \begin{rem}
 Note that for a connected sum $K$ of two bridge knots, 
 we have 
 \[
 \kappa (K) = \frac{1}{2} \kappa^* (K).
 \]
 It is natural to ask whether there exists a knot $K$ such that $\kappa (K) \neq  \frac{1}{2} \kappa^* (K)$ or not. For examples, we consider $K_n := T(3,12n-5)$ for positive integer $n$.
Then, we saw $\kappa(K_n) = -\frac{1}{2}$. On the other hand, it is proven in \cite[Theorem 1.2]{Ma14} that 
\[
\kappa^* (K_n)= \kappa ( \Sigma(2,3,12n-5))  = 1. 
\]
 \end{rem}

\subsection{Topological obstructions}  
We review several topological obstructions to sliceness of knots related to our results. 
For more details, see \cite[Section 3]{manolescu}. 

The following is a relative version of the Rochlin's result \cite{Ro71}. 
 \begin{theo}[\cite{Ki89,Ya96, K20}]\label{Arf nonspin}
Let $X$ be a smooth, closed, connected, oriented 4-manifold. If $ S \subset X^\circ$ is a
properly embedded, locally flat characteristic surface with boundary a knot $K$, then 
\[
\frac{\sigma (X) -[S]^2  } {8 } = \operatorname{Arf}(K) +  \operatorname{Arf}(X, S ). 
\]
 \end{theo}

When $X$ is spin and $[\Sigma] =0$, \cref{Arf nonspin} implies the following: 
\begin{theo}[\cite{R65}]\label{arf inv}
If a knot $K$ is topologically H-slice in a spin smooth 4-manifold, then $\operatorname{Arf}(K) = 0$.
\end{theo}

Also, we review a genus bound coming from the Tristram-Levine signature. 
Put 
\[
 S^1_{!} := \Set { t \in S^1 | \forall f \in \Z[u, u^{-1}] , f(1)=\pm 1, f(t) \neq 0 }.
 \]
 It is known that $S^1_{!}$ is a dense subset in $S^1$.
 For a knot $K\subset S^3$ and a value $t \in S^1_w$, the Tristram-Levine signature $\sigma_K(t)$ is defined as the signature of the following invertible matrix
 \[
 (1-t)M + (1-\overline{t})M^t,\] 
 where $M$ is a Seifert matrix for $K$. The value $\sigma_K(-1)$ is the usual signature. 

\begin{theo}[\cite{Vi75, Gil81, CN20}]\label{000}
 Let $X$ be a topological closed oriented 4-manifold with $H_1(X;\Z) = 0$. Let $S  \subset  X^\circ$ be a locally flat, properly embedded surface of genus $g$, with boundary a knot $K \subset S^3$. If the homology class $[S ] \in  H_ 2(X) $ is divisible by $2$, then 

 \[
 | \sigma (K) + \sigma (X) + [S]^2 | \leq b_2(X) + 2g (S)
 \]
 holds. Moreover, if the homology class $[S] \in  H_ 2(X) $ is divisible by a prime power $m= p^k$, then 
 \[
 | \sigma_K (e^{2\pi r i m } ) + \sigma (X) - \frac{2r (m-r)}{m^2} [S]^2 | \leq b_2(X) + 2g 
 \]
 holds for $r \in \{1, \cdots, m-1\}$.  
 In particular, if $K$ is topologically H-slice, then 
 \[
 \max_{r \in \{1, \cdots, m-1\}} | \sigma_K (e^{2\pi r i m } ) + \sigma (X)  | \leq b_2(X) + 2g(S) 
 \]
 for any prime power $m=p^k$.
 \end{theo}

\begin{cor}
\label{cor: G signature connstraint on involution}
Let $W$ be an oriented topological compact 4-manifold with boundary.
Suppose that a locally linear involution $\iota$ on $W$ is given and that the fixed-point set of $\iota$ is of codimension-2.
Let $S$ be the fixed-point set of $\iota$ and set $K = S \cap \partial W$.
Assume that $S$ is connected.
Then we have 
\[
\frac{1}{2}\left|\sigma(W)+[S]^2+\sigma(K)\right|
\leq \frac{b_2(W)}{2} + g(S).
\]
\end{cor}

\begin{proof}
This follows from \cref{000} and the computation of $b^+, b^-$ of the branched cover in \cref{branched spin}.
(Note that, in the proof of \cref{branched spin}, the smoothness and spinness of the manifold and involution were not used to compute these quantities.)
\end{proof}

\section{Kappa invariant for prime knots with 8- or 9-crossings}
\label{Kappa invariant for knots with 8 and 9-crossings}

We provide several computations of the kappa invariant for prime knots with $8$- or $9$-crossings. Since we have already calculated kappa invariant for two bridge knots in \cref{two bridge kappa12}, we only focus on prime knots whose bridge indexes are grater than $2$.
We use the notations in the Rolfsen's table.  Our main tool is \cref{more computation}. (See \cite{KnotAtlas}. Again note that our sign convention of the knot signature is opposite to the Knot Atlas. )
\\

\begin{center}
\begingroup
\renewcommand{\arraystretch}{1.2}
\begin{tabular}{ |c|c|c||c|c|c||c|c|c| } 
\hline 
Knots  & $\kappa$ & $\sigma$ & Knots  & $\kappa$ & $\sigma$  & Knots  & $\kappa$ & $\sigma$ \\
\hline \hline
$8_{5}$ & $\frac{1}{4} \text{ or } \frac{5}{4}$ &  $-4$ & $9_{28}$ & $-\frac{1}{8}$ & $2$  & $9_{41}$& $0$  & $0$ \\
\hline 
$8_{10}$& $\frac{1}{8}$ & $-2 $& $9_{29}$& $-\frac{1}{8}$ \text{ or } $\frac{-9}{8}$ & $2$ & $9_{42}$& $\frac{1}{8}$ & $-2$ \\
\hline 
$8_{16}$ & $\frac{1}{8}$ \text{ or } $-\frac{7}{8}$&$-2 $ & $9_{30}$& $0$ & $0$& $9_{43}$& $\frac{1}{4}$ or $\frac{5}{4}$ & $-4$\\
\hline 
$8_{17}$ & $0$ & $0$ & $9_{32}$& $\frac{1}{8}$ & $-2$ & $9_{44}$& $0$ &$0$ \\
\hline 
$8_{18}$ &  0 & 0 & $9_{33}$& $0$ & $0 $ & $9_{45}$& $-\frac{1}{8}$ &  $2$ \\
\hline 
$8_{19}$ & $\frac{3}{8}$ or $\frac{11}{8}$ & $-6$ & $9_{34}$& $0$ &$0$ & $9_{46}$& $0$ & $ 0$\\ 
\hline 
$8_{20}$ &$0 $ & $0$& $9_{35}$ & $-\frac{1}{8} \text{ or } \frac{7}{8}$ & $2$ & $9_{47}$& $\frac{1}{8}$ &  $-2$\\
\hline 
$8_{21}$ & $-\frac{1}{8}$ & $2$ & $9_{36}$ & $\frac{1}{4} \text{ or } \frac{5}{4}$ & $-4$ & $9_{48}$& $\frac{1}{8}$&$-2$ \\ 
\hline 
$9_{16}$ & $\frac{3}{8}$ & $-6$ & $9_{37}$& 0 & 0& $9_{49}$& $\frac{1}{4}$ & $-4$ \\ 
\hline 
$9_{22}$ & $\frac{1}{8}$ \text{ or } $-\frac{7}{8}$ & $-2$ & $9_{38}$& $-\frac{1}{4}$ & $4$ & &  &\\
\hline 
$9_{24}$ &$0$ & $0$& $9_{39}$& $\frac{1}{8}$ & $-2$& &  & \\ 
\hline 
$9_{25}$ & $-\frac{1}{8}$ & $2$ & $9_{40}$& $-\frac{1}{8}$ \text{ or } $\frac{7}{8}$ & $2$ & &  &\\ 
\hline
\end{tabular}
\\
\endgroup
\end{center}

We omit the constructions of positive crossing changes to use \cref{more computation}. 
For example, there are the following positive crossing changes: $6_3 \text{ or }6_3^*\xrightarrow{p.c.c.} 8_{10}$, $8_{10} \xrightarrow{p.c.c.} T(2,5)$  ($\sigma (8_{10})=-2$, $\sigma (6_3)=0$, $\sigma (T(2,5))= -4$). Thus, from \cref{more computation}, we have $\kappa (8_{10} ) = \frac{1}{8}$.
  Also, in the computation of $\kappa(9_{49})$, we used several non-trivial facts related to the hyperbolic 3-manifold $\Sigma (9_{49})$ in \cref{hyperbolic example}. Similarly, for several knots with hyperbolic branched covers, the kappa invariants are described in terms of the Fr\o yshov invariant $\delta (\Sigma(K))$  in \cref{several hype ex}. See also Table \ref{table: hyperbolic1}.

\bibliographystyle{plain}
\bibliography{tex}

\end{document}